\documentclass[a4paper, 10pt, american]{amsart}

\usepackage[colorlinks,pdfpagelabels,pdfstartview = FitH,bookmarksopen
= true,bookmarksnumbered = true,linkcolor = blue,plainpages =
false,hypertexnames = false,citecolor = red,pagebackref=false]{hyperref}
\usepackage{times,latexsym,amssymb}
\usepackage{amsmath,amsthm,bm}
\usepackage{graphics, epsfig}
\usepackage{color}
\usepackage{appendix}
\usepackage[margin=1.5in]{geometry}

\usepackage{amsbsy}
\usepackage{amstext}
\usepackage{amssymb}
\usepackage{esint}
\usepackage{stmaryrd}

\let\TeXchi\chi
\newbox\chibox
\setbox0 \hbox{\mathsurround0pt $\TeXchi$}
\setbox\chibox \hbox{\raise\dp0 \box 0 }
\def\chi{\copy\chibox}

\renewcommand{\d}{\mathrm{d}}
\newcommand{\dx}{\mathrm{d}x}

\newcommand{\dt}{\mathrm{d}t}
\newcommand{\ds}{\mathrm{d}s}
\newcommand{\dtau}{\mathrm{d}\tau}

\renewcommand{\epsilon}{\varepsilon}
\renewcommand{\rho}{\varrho}

\newcommand{\power}[2]{\bm{#1^{\mbox{\unboldmath{\scriptsize$#2$}}}}}

\author[V. B\"ogelein]{Verena B\"{o}gelein}
\address{Verena B\"ogelein\\
Fachbereich Mathematik, Universit\"at Salzburg\\
Hellbrunner Str. 34, 5020 Salzburg, Austria}
\email{verena.boegelein@sbg.ac.at}

\author[F. Duzaar]{Frank Duzaar}
\address{Frank Duzaar\\
Department Mathematik, Universit\"at Erlangen--N\"urnberg\\
Cauerstrasse 11, 91058 Erlangen, Germany}
\email{frank.duzaar@fau.de}

\author[N. Liao]{Naian Liao}
\address{Naian Liao\\
Fachbereich Mathematik, Universit\"at Salzburg\\
Hellbrunner Str. 34, 5020 Salzburg, Austria}
\email{naian.liao@sbg.ac.at}



\begin{document}
\newtheorem{proposition}{Proposition}[section]
\newtheorem{theorem}{Theorem}[section]
\newtheorem{lemma}{Lemma}[section]
\newtheorem{corollary}{Corollary}[section]
\newtheorem{remark}{Remark}[section]
\newtheorem{definition}{Definition}[section]
\renewcommand{\thesection}{\arabic{section}}
\renewcommand{\theequation}{\thesection.\arabic{equation}}
\renewcommand{\thetheorem}{\thesection.\arabic{theorem}}
\numberwithin{equation}{section}
\numberwithin{theorem}{section}
\numberwithin{proposition}{section}
\numberwithin{lemma}{section}
\numberwithin{remark}{section}
\numberwithin{definition}{section}
\setcounter{secnumdepth}{3}
\newcommand{\cl}{\centerline}
\newcommand{\sms}{\smallskip}
\newcommand{\ms}{\medskip}
\newcommand{\bs}{\bigskip}
\newcommand{\noi}{\noindent}
\newcommand{\itl}[1]{\textit{#1}}
\newcommand{\blf}[1]{\textbf{#1}}
\newcommand{\dsty}{\displaystyle}
\newcommand{\txty}{\textstyle}
\newcommand{\ssty}{\scriptstyle}
\newcommand{\tty}{\texttt}


\newcommand\Par{\mathhexbox278\,}


\newcommand{\al}{\alpha}
\newcommand{\Al}{\Alpha}
\newcommand{\be}{\beta}
\newcommand{\Be}{\Beta}
\newcommand{\Gm}{\Gamma}
\newcommand{\gm}{\gamma}
\newcommand{\dl}{\delta}
\newcommand{\Dl}{\Delta}
\newcommand{\lm}{\lambda}
\newcommand{\Lm}{\Lambda}
\newcommand{\kp}{\kappa}
\newcommand{\varep}{\varepsilon}
\newcommand{\eps}{\epsilon}
\newcommand{\vp}{\varphi}
\newcommand{\sig}{\sigma}
\newcommand{\Sig}{\Sigma}
\newcommand{\om}{\omega}
\newcommand{\Om}{\Omega}
\newcommand{\uom}{\mbox{\boldmath$\omega$}}
\newcommand{\btau}{\mbox{\boldmath$\tau$}}
\newcommand{\bnu}{\mbox{\boldmath$\nu$}}
\newcommand{\up}{\upsilon}
\newcommand{\z}{\zeta}


\newcommand{\df}[1]{\buildrel\mbox{\small def}\over{#1}}
\newcommand{\op}[1]{\buildrel\mbox{\tiny o}\over{#1}}
\newcommand{\db}{\prime\prime}
\newcommand{\bsl}{\backslash}
\newcommand{\lb}{\lbrack\!\lbrack}
\newcommand{\rb}{\rbrack\!\rbrack}
\newcommand\la{\langle}
\newcommand\ra{\rangle}
\newcommand{\ev}{\equiv}
\newcommand{\nev}{\not\equiv}
\newcommand{\nn}{\mathbb{N}}
\newcommand{\qq}{\mathbb{Q}}
\newcommand{\zz}{\mathbb{Z}}
\newcommand{\rr}{\mathbb{R}}
\newcommand{\rn}{\rr^N}
\newcommand{\cc}{\mathbb{C}}
\newcommand{\id}{\mathbb{I}}
\newcommand{\bo}{\mathbb{O}}

\newcommand{\amsb}[1]{\mathbb{#1}}
\newcommand{\mcl}[1]{\mathcal{#1}}
\newcommand{\bl}[1]{\mathbf{#1}}
\newcommand{\ov}[1]{\overline{#1}}
\newcommand{\wt}[1]{\widetilde{#1}}
\newcommand{\wh}[1]{\widehat{#1}}

\newcommand{\llra}{\leftrightarrow}
\newcommand{\lra}{\longrightarrow}
\newcommand{\LLR}{\Longleftrightarrow}
\newcommand{\LRA}{\Longrightarrow}
\newcommand{\LLA}{\Longleftarrow}


\newcommand{\bbox}{\vrule height.6em width.6em 
depth0em} 
\newcommand{\os}{\vbox{\hrule \hbox{\vrule 
height.6em depth0pt 
\hskip.6em \vrule height.6em depth0em}
\hrule}} 


\newcommand{\dvg}{\operatorname{div}}
\newcommand{\curl}{\operatorname{curl}}
\newcommand{\supp}{\operatorname{supp}}
\newcommand{\essup}{\operatornamewithlimits{ess\,sup}}
\newcommand{\essinf}{\operatornamewithlimits{ess\,inf}}
\newcommand{\essosc}{\operatornamewithlimits{ess\,osc}}
\newcommand{\osc}{\operatornamewithlimits{osc}}
\newcommand{\sign}{\operatorname{sign}}
\newcommand{\loc}{\operatorname{loc}}
\newcommand{\diam}{\operatorname{diam}}
\newcommand{\dist}{\operatorname{dist}}
\newcommand{\card}{\operatorname{card}}
\newcommand{\meas}{\operatorname{meas}}
\newcommand{\spn}{\operatorname{span}}
\newcommand{\dtm}{\operatorname{det}}
%


\newcommand{\overlim}{\mathop{\overline{\lim}}\limits}
\newcommand{\underlim}{\mathop{\underline{\lim}}\limits}
\newcommand{\ttop}[2]{\genfrac{}{}{0pt}{}{#1}{#2}}
\newcommand{\bcu}{\mathop{\txty{\bigcup}}\limits}
\newcommand{\bca}{\mathop{\txty{\bigcap}}\limits}
\newcommand{\bsu}{\mathop{\txty{\sum}}\limits}
\newcommand{\pro}{\mathop{\txty{\prod}}\limits}


\newcommand{\pl}{\partial}
\newcommand{\ptt}{\frac{\pl}{\pl t}}
\newcommand{\ppx}{\frac\pl{\pl x}}
\newcommand{\dds}{\frac d{ds}}
\newcommand{\ddt}{\frac d{dt}}

\newcommand{\intl}{\int\limits}
\newcommand{\iintl}{\iint\limits}
\def\Xint#1{\mathchoice
    {\XXint\displaystyle\textstyle{#1}}%
    {\XXint\textstyle\scriptstyle{#1}}%
    {\XXint\scriptstyle\scriptscriptstyle{#1}}%
    {\XXint\scriptscriptstyle\scriptscriptstyle{#1}}%
    \!\int}
\def\XXint#1#2#3{\setbox0=\hbox{$#1{#2#3}{\int}$}
    \vcenter{\hbox{$#2#3$}}\kern-0.5\wd0}
\def\bint{\Xint-}
\def\dashint{\Xint{\raise4pt\hbox to7pt{\hrulefill}}}
\def\dashiint{\bint\kern-0.15cm\bint}

\newcommand{\ovl}[3]{\int_{#1}^{#2}\kern-#3pt\raise4pt\hbox to7pt{\hrulefill}\ }

\newcommand{\ovll}[3]{\intl_{#1}^{#2}\kern-#3pt\raise4pt\hbox to7pt{\hrulefill}\ }

\newcommand{\tvl}[2]{\iint_{#1}\kern-#2pt\raise4pt\hbox to7pt{\hrulefill}\ }



\newcommand{\omt}{\Om_T}
\newcommand{\plo}{\partial\Omega}
\newcommand{\ovo}{\bar{\Om} }

%
\newcommand{\ci}[1]{C^\infty\!\left({#1}\right)}
\newcommand{\cio}[1]{C_o^\infty\!\left({#1}\right)}
\newcommand{\lloc}[1]{L_{\loc}\!\left({#1}\right)}
\newcommand{\xy}{|x-y|}


\newcommand{\intom}{\intl_{\Om}}
\newcommand{\intbo}{\intl_{\plo}}
\newcommand{\inom}{\int_{\Om}}
\newcommand{\inbo}{\int_{\plo}}
\newcommand{\intrn}{\intl_{\rn}}


\newcommand{\bye}{\end{document}}



%
%
%

%
%
%
%
%

%
%

\title{On the H\"older regularity of signed solutions to a doubly nonlinear equation}

\date{}
\maketitle
\begin{abstract}
We establish the interior and boundary H\"older continuity of possibly sign-changing solutions to
a class of doubly nonlinear parabolic equations whose prototype is
\[
\partial_t\big(|u|^{p-2}u\big)-\Dl_p u=0,\quad p>1.
\]
The proof relies on the property of expansion of positivity and
the method of intrinsic scaling, all of which are realized by De Giorgi's iteration.
Our approach, while emphasizing the distinct roles of sub(super)-solutions,
is flexible enough to obtain the H\"older regularity of solutions
to initial-boundary value problems of Dirichlet type or Neumann type in a cylindrical domain, 
up to the parabolic boundary. In addition, based on the expansion of positivity,
we are able to give an alternative proof of Harnack's inequality for non-negative solutions.
Moreover, as a consequence of the interior estimates, 
we also obtain a Liouville-type result.
\vskip.2truecm
\end{abstract}

{\small\tableofcontents}

\section{Introduction and Main Results}
Let $E$ be an open set in $\rn$. For $T>0$ let $E_T$ denote the
cylindrical domain $E\times(0,T]$. 
We shall consider quasi-linear, parabolic partial differential equations of the form
\begin{equation}  \label{Eq:1:1}
	\partial_t\big(|u|^{p-2}u\big)-\dvg\bl{A}(x,t,u, Du) = 0\quad \mbox{ weakly in $ E_T$}
\end{equation}
where the function $\bl{A}(x,t,u,\z)\colon E_T\times\rr^{N+1}\to\rn$ is only assumed to be
measurable with respect to $(x, t) \in E_T$ for all $(u,\z)\in \rr\times\rn$,
continuous with respect to $(u,\z)$ for a.e.~$(x,t)\in E_T$,
and subject to the structure conditions
\begin{equation}\label{Eq:1:2p}
	\left\{
	\begin{array}{c}
		\bl{A}(x,t,u,\z)\cdot \z\ge C_o|\z|^p \\[5pt]
		|\bl{A}(x,t,u,\z)|\le C_1|\z|^{p-1}%
	\end{array}
	\right .
	\qquad \mbox{for a.e.~$(x,t)\in E_T$, $\forall\,u\in\rr$, $\forall\,\z\in\rn$,}
\end{equation}
where $C_o$ and $C_1$ are given positive constants, and $p>1$.
The prototype equation is
\begin{equation}\label{prototype}
	\partial_t\big(|u|^{p-2}u\big)-\Dl_p u=0\quad\mbox{ weakly in $ E_T$.}
\end{equation}
Here $\Dl_p:=\dvg (|Du|^{p-2}Du)$ is the $p$-Laplace operator.
When $p=2$ it becomes the heat equation.

The motivations to study such an equation will be explored in Section~\ref{S:NS}.
We however proceed to present our main results on the interior regularity in Section~\ref{S:interior}
and the boundary regularity in Section~\ref{S:boundary}.

When we speak of the structural {\it data}, we refer to the set of parameters
$\{p,\,N,\,C_o,\,C_1\}$.
We also write $\boldsymbol \gm$ as a generic positive constant that can be quantitatively
determined a priori only in terms of the data and that can change from line to line.


\subsection{Interior Regularity}\label{S:interior}
Let $\Gm:=\pl E_T-\overline{E}\times\{T\}$
be the parabolic boundary of $E_T$, and for a compact set $\mathcal{K}\subset E_T$
introduce the parabolic $p$-distance from $\mathcal{K}$ to $\Gm$ by
\begin{equation*}
	\begin{aligned}
		\dist_p(\mathcal{K};\,\Gm)&\df{=}\inf_{\substack{(x,t)\in \mathcal{K}\\(y,s)\in\Gm}}
		\left\{|x-y|+|t-s|^{\frac1p}\right\}.
	\end{aligned}
\end{equation*}
For $\varrho>0$ let $K_\varrho(x_o)$ be the cube with center at $x_o\in\rn$
and edge $\varrho$. When $x_o=0$ we simply write $K_\varrho$.
We define backward cylinders scaled by a positive parameter $\theta$ by
\begin{equation*}
	(x_o,t_o)+Q_\varrho(\theta)=
	(x_o,t_o)+ K_\varrho (0)\times (-\theta\varrho^p,0]
	=
	K_{\varrho}(x_o)\times(t_o-\theta\varrho^p,t_o].
\end{equation*}
If $\theta=1$, we simply write $Q_\varrho$. 

We postpone the formal definition of local weak solution to
Section~\ref{S:1:2}. It is however noteworthy to mention here 
that local boundedness of local weak solutions is inherent
in our notion of local solution (cf. Section~\ref{S:1:4:2}). Thus we may always
work with locally bounded solutions.

Now we state our main result concerning the interior H\"older continuity of weak solutions
to \eqref{Eq:1:1}, subject to the structure conditions \eqref{Eq:1:2p}.
\begin{theorem}\label{Thm:1:1}
	Let $u$ be a bounded, local, weak solution to \eqref{Eq:1:1} -- \eqref{Eq:1:2p} in $E_T$.
	Then $u$ is locally H\"older continuous in $E_T$. More precisely,
	there exist constants $\boldsymbol\gm>1$ and $\be\in(0,1)$ that can be determined a priori
	only in terms of the data, such that for every compact set $\mathcal{K}\subset E_T$,
	\begin{equation*}
	\big|u(x_1,t_1)-u(x_2,t_2)\big|
	\le
	\boldsymbol \gm\|u\|_{\infty,E_T}
	\left(\frac{|x_1-x_2|+|t_1-t_2|^{\frac1p}}{\dist_p(\mathcal{K};\Gm)}\right)^{\be},
	\end{equation*}
for every pair of points $(x_1,t_1), (x_2,t_2)\in \mathcal{K}$.
\end{theorem}
\begin{remark}\label{Rmk:1:1}\upshape
We have stated Theorem~\ref{Thm:1:1} for globally bounded weak solutions.
However the proof has a local thrust. As a matter of fact, we will show the following
oscillation decay:
\[
\essosc_{(x_o,t_o)+Q_r}u\le\boldsymbol\gm\essosc_{(x_o,t_o)+Q_\rho}u\,\left(\frac{r}{\rho}\right)^\be,
\]
for any pair of cylinders $(x_o,t_o)+Q_r\Subset(x_o,t_o)+Q_\rho\Subset E_T$.
The conclusion of Theorem~\ref{Thm:1:1} can be derived from this oscillation estimate
via a standard covering argument.
\end{remark}

The oscillation decay in Remark~\ref{Rmk:1:1}, while local in nature, has a global implication.
Indeed, let $u$ be a {\it bounded}, local weak solution to \eqref{Eq:1:1} -- \eqref{Eq:1:2p} in the semi-infinite strip
$\mathcal{S}_T:=\rr^N\times(-\infty,T)$ for some $T\in\rr$. Then we have
\[
\essosc_{(x_o,t_o)+Q_r}u\le\boldsymbol\gm\|u\|_{\infty, \mathcal{S}_T}\,\left(\frac{r}{\rho}\right)^\be,
\]
for any pair of cylinders $(x_o,t_o)+Q_r\Subset(x_o,t_o)+Q_\rho\Subset \mathcal{S}_T$.
Now fixing $r$ and letting $\rho\to\infty$, we immediately arrive at a Liouville-type result.
\begin{corollary}\label{Cor:1:1}
A bounded, local weak solution to \eqref{Eq:1:1} -- \eqref{Eq:1:2p} in $\mathcal{S}_T$
must be a constant.
\end{corollary}
\begin{remark}\upshape
One-sided boundedness of solutions in $\mathcal{S}_T$ is generally not sufficient to imply they are constants.
This is evident from the non-negative solution $u(x,t)=e^{x+t}$ to the one dimensional heat equation.
\end{remark}
\begin{remark}\upshape
The global boundedness condition in Corollary~\ref{Cor:1:1} can be easily relaxed to allow $u$ to grow
slower than $(|x|+|t|^{\frac1p})^\be$ as $|x|\to\infty$ and $t\to-\infty$.
Other variants of Liouville-type results may be obtained, for which we refer to \cite{DBGV-Liouville, DBGV-mono}.
\end{remark}
\subsection{Boundary Regularity}\label{S:boundary}
We will establish regularity of weak solutions to \eqref{Eq:1:1} -- \eqref{Eq:1:2p}
up to the lateral boundary $S_T:=\pl E\times(0,T]$,
provided the solution satisfies proper Dirichlet or Neumann boundary data
and $\pl E$ possesses certain geometry or smoothness.
Likewise, regularity of weak solutions up to the initial level $t=0$ can
also be obtained, provided the given initial value is regular enough.

The arguments employed will be local in nature. As a result, it suffices to require the boundary 
data to be taken just on a portion of the parabolic boundary.
Nevertheless, we choose to present the results globally for simplicity,
in terms of initial-boundary value problems.

To this end, let us first consider formally the following initial-boundary value problem of Dirichlet type:
\begin{equation}\label{Dirichlet}
\left\{
\begin{array}{c}
	\partial_t\big(|u|^{p-2}u\big)-\dvg\bl{A}(x,t,u, Du) = 0\quad \mbox{weakly in $ E_T$,}\\[7pt]
	u(\cdot,t)\Big|_{\partial E}=g(\cdot,t)\Big|_{\partial E}\quad \mbox{for a.e.~$ t\in(0,T]$,}\\[7pt]
	u(\cdot,0)=u_o(\cdot),
\end{array}
\right.
\end{equation}
where the structure conditions \eqref{Eq:1:2p} are retained. 
Regarding the Dirichlet datum $g$ and the initial datum $u_o$ we assume
\begin{align}
\tag{\bf D}\label{D} & \mbox{$\dsty g\in L^p\big(0,T;W^{1,p}( E)\big)$, and $g$ is continuous on $S_T$ with modulus of continuity
  				$\om_g(\cdot)$;} \\
\tag{\bf{I}}\label{I} & \mbox{$u_o$ is continuous in $\overline{E}$ with modulus of continuity $\om_{o}(\cdot)$.}
\end{align}
In order to establish H\"older regularity of $u$ up to $S_T$,
we need to impose some geometric conditions on $\pl E$.
For this purpose, we introduce the property of {\it positive geometric density}
of $\pl E$, i.e.,
\begin{equation}\label{geometry}
\left\{\;\;
	\begin{minipage}[c][1.5cm]{0.7\textwidth}
	there exists $\al_*\in(0,1)$ and $\rho_o>0$, such that for all $x_o\in\pl E$,
	for every cube $K_\rho(x_o)$ and $0<\rho\le\rho_o$, there holds
	$$
	|E\cap K_{\rho}(x_o)|\le(1-\al_*)|K_\rho|.
	$$
	\end{minipage}
\right.
\end{equation}
Intuitively, this means one can place an exterior cone whose vertex is attached to $x_o$ (uniformly
with respect to $x_o$). 

Next, we consider the Neumann problem. In order to deal with possible variational data on $S_T$,
we assume $\pl E$ is of class $C^1$, 
such that the outward unit normal, which we denote by {\bf n},
is defined on $\pl E$.
Let us consider the initial-boundary value problem of Neumann type:
\begin{equation}\label{Neumann}
\left\{
\begin{array}{c}
	\partial_t\big(|u|^{p-2}u\big)-\dvg\bl{A}(x,t,u, Du) = 0\quad \mbox{weakly in $ E_T$,}\\[5pt]
	\bl{A}(x,t,u, Du)\cdot {\bf n}=\psi(x,t, u)\quad \mbox{on $S_T$,}\\[5pt]
	u(\cdot,0)=u_o(\cdot),
\end{array}
\right.
\end{equation}
where the structure conditions \eqref{Eq:1:2p} and the initial condition $({\bf I})$ are retained. 
On the Neumann datum $\psi$ we assume for simplicity that, for some absolute constant $C_2$,
there holds
\begin{equation}\label{N-data}\tag{\bf{N}}
|\psi(x,t, u)|\le C_2\quad \text{ for a.e. }(x,t, u)\in S_T\times\rr.
\end{equation}
More general conditions should also work (cf.~Sec. 2, Chap. II, \cite{DB}).
The formal definitions of weak solutions to \eqref{Dirichlet} and \eqref{Neumann} will be given in Section~\ref{S:1:2}.
Now we are ready to present the results concerning regularity of solutions to \eqref{Dirichlet}
or \eqref{Neumann} up to the parabolic boundary $\Gm$.
\subsubsection{Near the Initial Time}
\begin{theorem}\label{Thm:1:2}
Let $u$ be a bounded weak solution to
the Dirichlet problem 
\eqref{Dirichlet} under the assumption \eqref{Eq:1:2p}.
Assume \eqref{I} holds. Then $u$ is continuous in $K\times[0,T]$ for any compact set $K\subset E$.
More precisely, there is a modulus of continuity $\boldsymbol\om(\cdot)$,
determined by the data, $\dist(K,\pl E)$, $\|u\|_{\infty,E_T}$ and $\boldsymbol\om_{o}(\cdot)$, such that
	\begin{equation*}
	\big|u(x_1,t_1)-u(x_2,t_2)\big|
	\le
	\boldsymbol\om\!\!\left(|x_1-x_2|+|t_1-t_2|^{\frac1p}\right),
	\end{equation*}
for every pair of points $(x_1,t_1), (x_2,t_2)\in K\times[0,T]$.
In particular, if $u_o$ is H\"older continuous with exponent $\be_{o}$,
then $\boldsymbol\om(r)=\boldsymbol\gm r^{\be} \|u\|_{\infty,E_T}$ with some $\boldsymbol\gm>0$ and 
$\be\in(0,\be_{o}]$ depending on the data, $\dist(K,\pl E)$ and $\be_{o}$.
\end{theorem}
\subsubsection{Near $S_T$--Dirichlet Type Data}
\begin{theorem}\label{Thm:1:3}
Let $u$ be a bounded weak solution to the Dirichlet problem 
\eqref{Dirichlet} under the assumption \eqref{Eq:1:2p}. Assume \eqref{D} and \eqref{geometry} hold. Then $u$ is continuous in any compact set 
$\mathcal{K}\subset\overline{E}_T$.
More precisely, there is a modulus of continuity $\boldsymbol\om(\cdot)$,
determined by the data, $\al_*$, $\rho_o$, $\dist(\mathcal{K};\{t=0\})$, $\|u\|_{\infty,E_T}$ and $\boldsymbol\om_{g}(\cdot)$, such that
	\begin{equation*}
	\big|u(x_1,t_1)-u(x_2,t_2)\big|
	\le
	\boldsymbol\om\!\!\left(|x_1-x_2|+|t_1-t_2|^{\frac1p}\right),
	\end{equation*}
for every pair of points $(x_1,t_1), (x_2,t_2)\in \mathcal{K}$.
In particular, if $g$ is H\"older continuous with exponent $\be_{g}$,
then $\boldsymbol\om(r)=\boldsymbol\gm r^{\be} \|u\|_{\infty,E_T}$ with some $\boldsymbol\gm>0$ and 
$\be\in(0,\be_{g}]$ depending on the data, $\al_*$, $\rho_o$, $\dist(\mathcal{K};\{t=0\})$ and $\be_{g}$.
\end{theorem}
\subsubsection{Near $S_T$--Neumann Type Data}
\begin{theorem}\label{Thm:1:4}
Let $u$ be a bounded weak solution to the Neumann problem
\eqref{Neumann}.
 Assume $\pl E$ is of class $C^1$ and \eqref{N-data} holds. 
 Then $u$ is H\"older continuous in any compact set 
$\mathcal{K}\subset\overline{E}_T$.
More precisely, there exist constants $\boldsymbol\gm>1$ and $\be\in(0,1)$
determined by the data, $C_2$, $\dist(\mathcal{K};\{t=0\})$ and the structure of $\pl E$, such that
	\begin{equation*}
	\big|u(x_1,t_1)-u(x_2,t_2)\big|
	\le
	\boldsymbol \gm
	\|u\|_{\infty,E_T}\,
	\left(|x_1-x_2|+|t_1-t_2|^{\frac1p}\right)^\be,
	\end{equation*}
for every pair of points $(x_1,t_1), (x_2,t_2)\in \mathcal{K}$.
\end{theorem}
\subsection{Novelty and Significance}\label{S:NS}
The equation \eqref{Eq:1:1} -- \eqref{Eq:1:2p} has been referred to as a doubly nonlinear
parabolic equation in the literature, due to the nonlinearity of both the solution and its spatial gradient.
It is a particular form of a more general equation whose prototype is
\begin{equation}\label{general}
\partial_t\big(|u|^{m-1}u\big)-\Dl_p u=0,\quad p>1,\,m>0.
\end{equation}
The interest in such an equation stems from its mathematical structure,
in understanding doubly nonlinear phenomena that generate mixed types of degeneracy
and/or singularity in partial differential equations,
and its connection to physical models, including 
dynamics of glaciers (\cite{glacier}), shallow water flows (\cite{DSW1,DSW2,DSW3}) 
and friction dominated flow in a gas network (\cite{pipe}).

In particular, the prototype equation \eqref{prototype} is naturally connected to the nonlinear eigenvalue
problem $-\Dl_p u=\lm |u|^{p-2}u$ (cf. \cite{LL}), which plays an important role in the {\it nonlinear potential theory}.

The equation \eqref{Eq:1:1} -- \eqref{Eq:1:2p} has been observed by Trudinger (\cite{Trud0}), via Moser's iteration,
to possess a Harnack inequality for non-negative solutions, analogous to the one
for the heat equation. See also \cite{GV,KK}. Such a Harnack inequality has been used to
establish the interior H\"older regularity for non-negative solutions in \cite{Trud1,Trud2}. 

Our main contribution is to remove the sign restriction on solutions for the H\"older regularity to hold.
The Harnack inequality seems not applicable in this setting due to changing signs of solutions and the power-like
nonlinearity with respect to the solution itself. 
Instead, we employ a more basic tool -- expansion of positivity -- to handle the current situation.
Our approach emphasizes the different roles played by sub-solutions and super-solutions.
As a by-product, the expansion of positivity also leads to an alternative proof of the Harnack inequality.
See Appendix~\ref{Append:2}.
The interior estimates also give us a Liouville type result for global solutions,
which seems new in the literature.
Moreover, our approach is flexible enough to obtain the H\"older regularity of solutions
to the initial-boundary value problems of both Dirichlet type and Neumann type, 
up to the parabolic boundary. As far as we know, the boundary regularity has not ever been dealt with
in the literature even in the case of non-negative solutions.

Our proofs of H\"older regularity -- interior or boundary -- all unfold along two main cases, i.e., when the solution is close to zero
or when it is away from zero, through comparisons between the oscillation
and the supremum/infimum of the solution. In the first case, we will take advantage of the scaling invariant
property of the equation and obtain the expansion of positivity -- Proposition~\ref{Prop:1:1} -- without intrinsic scaling techniques.
This treatment parallels the classical parabolic theory ($p=2$) in \cite{LSU}, the new input being
that we need to trace the competition between the oscillation and the extrema of the solution (see Remark~\ref{Rmk:4:4}).
Whereas in the second case, the solution behaves like the one to the parabolic $p$-Laplacian equation, i.e.~$u_t=\Dl_p u$.
Thus this latter case hinges upon the possibility to treat such a degenerate ($p>2$) or singular ($1<p<2$) equation,
for which we exploit the existing theory in \cite{DB, DBGV-mono}.

The H\"older regularity for doubly nonlinear equations has also been considered in \cite{Ivanov-1, Ivanov-2, Ivanov-Mkrtychyan, Vespri, Vespri-Vestberg}, under various conditions on the structure of the equation. 
The local regularity theory for the doubly nonlinear equation \eqref{general}
seems fragmented and it deserves future investigations.

\subsection{Notations and Definitions}\label{S:1:2}

\subsubsection{Notion of Local Solution}\label{S:1:2:1}
A function
\begin{equation}  \label{Eq:1:3p}
	u\in C\big(0,T;L^p_{\loc}(E)\big)\cap L^p_{\loc}\big(0,T; W^{1,p}_{\loc}(E)\big)
\end{equation}
is a local, weak sub(super)-solution to \eqref{Eq:1:1} with the structure
conditions \eqref{Eq:1:2p}, if for every compact set $K\subset E$ and every sub-interval
$[t_1,t_2]\subset (0,T]$
\begin{equation}  \label{Eq:1:4p}
	\int_K |u|^{p-2}u\z \,\dx\bigg|_{t_1}^{t_2}
	+
	\iint_{K\times (t_1,t_2)} \big[-|u|^{p-2}u\z_t+\bl{A}(x,t,u,Du)\cdot D\z\big]\dx\dt
	\le(\ge)0
\end{equation}
for all non-negative test functions
\begin{equation*}
\z\in W^{1,p}_{\loc}\big(0,T;L^p(K)\big)\cap L^p_{\loc}\big(0,T;W_o^{1,p}(K)%
\big).
\end{equation*}
This guarantees that all the integrals in \eqref{Eq:1:4p} are convergent.

A function $u$ that is both a local weak sub-solution and a local weak super-solution
to \eqref{Eq:1:1} -- \eqref{Eq:1:2p} is a local weak solution.

\subsubsection{Notion of Parabolicity and Local Boundedness of Solutions}\label{S:1:4:2}
 For any $k\in\rr$, let
\[
(u-k)_-=\max\{-(u-k),0\},\qquad(u-k)_+=\max\{u-k,0\}.
\]
Accordingly, we notice that
\[
k-(u-k)_-=\min\{u,k\},\qquad k+(u-k)_+=\max\{u,k\}.
\]
Using \eqref{Eq:1:2p}$_1$ and employing a similar method as in
({\bf A}$_6$) of \cite[Chapter II]{DB} or Lemma~1.1 of \cite[Chapter 3]{DBGV-mono},
we can show that the equation \eqref{Eq:1:1} with \eqref{Eq:1:2p}
is {\it parabolic}, in the sense that 
\begin{equation*}
	\left\{
	\begin{aligned}
		&\mbox{whenever $u$ is a local weak sub(super)-solution,}\\ 
		&\mbox{the function $k\pm(u-k)_\pm$ is a local weak sub(super)-solution, for all $\ k\in\rr$.}
	\end{aligned}
	\right.
\end{equation*}
We will give a proof of this claim in Appendix~\ref{Append:1}.
In particular, when $u$ is a local weak solution,
 $u_{+}$ and $u_-$ are non-negative, local weak sub-solutions to \eqref{Eq:1:1} -- \eqref{Eq:1:2p}.
Since it has been shown that non-negative, local sub-solutions
are locally bounded, we may always work with locally bounded solutions.
See \cite{GV, KK} in this regard.

\subsubsection{Notion of Solution to the Dirichlet Problem}\label{S:1:4:3}
A function
\begin{equation*}  
	u\in C\big(0,T;L^p(E)\big)\cap L^p\big(0,T; W^{1,p}(E)\big)
\end{equation*}
is a weak sub(super)-solution to \eqref{Dirichlet}, 
if for every sub-interval
$[t_1,t_2]\subset (0,T]$,
\begin{equation*} 
\begin{aligned}
	\int_{E} |u|^{p-2}u\z \,\dx\bigg|_{t_1}^{t_2}
	&+
	\iint_{E\times(t_1,t_2)} \big[-|u|^{p-2}u\z_t+\bl{A}(x,t,u,Du)\cdot D\z\big]\dx\dt
	\le(\ge)0
\end{aligned}
\end{equation*}
for all non-negative test functions
\begin{equation*}
\z\in W_{\loc}^{1,p}\big(0,T;L^p(E)\big)\cap L_{\loc}^p\big(0,T;W_o^{1,p}(E)%
\big).
\end{equation*}
Moreover, setting $\hat{p}:=\min\{2,p\}$,
the initial datum is taken in the sense that for any compact set $K\Subset E$,
\[
\int_{K\times\{t\}}(u-u_o)^{\hat{p}}_{\pm}\,\dx\to0\quad\text{ as }t\downarrow0.
\]
The Dirichlet datum $g$ is attained under $u\le(\ge)g$
on $\pl E$ in the sense that the traces of $(u-g)_{\pm}$
vanish as functions in $W^{1,p}(E)$ for a.e. $t\in(0,T]$, i.e. $(u-g)_{\pm}\in L^p(0,T; W^{1,p}_o(E))$.
Notice that no {\it a priori} information is assumed on the smoothness of $\pl E$.

A function $u$ that is both a weak sub-solution and a weak super-solution
to \eqref{Dirichlet} is a weak solution.
\subsubsection{Notion of Solution to the Neumann Problem}\label{S:1:4:4}
A function
\begin{equation*}  
	u\in C\big(0,T;L^p(E)\big)\cap L^p\big(0,T; W^{1,p}(E)\big)
\end{equation*}
is a weak sub(super)-solution to \eqref{Neumann}, 
if for every compact set $K\subset \rr^N$ and every sub-interval
$[t_1,t_2]\subset (0,T]$,
\begin{equation*}  
\begin{aligned}
	\int_{K\cap E} |u|^{p-2}u\z \,\dx\bigg|_{t_1}^{t_2}
	&+
	\iint_{\{K\cap E\}\times(t_1,t_2)} \big[-|u|^{p-2}u\z_t+\bl{A}(x,t,u,Du)\cdot D\z\big]\dx\dt\\
	&\le(\ge)\iint_{\{K\cap\pl E\}\times(t_1,t_2)}\psi(x,t,u)\z\,\d\sig\dt
\end{aligned}
\end{equation*}
for all non-negative test functions
\begin{equation*}
\z\in W_{\loc}^{1,p}\big(0,T;L^p(K)\big)\cap L_{\loc}^p\big(0,T;W_o^{1,p}(K)%
\big).
\end{equation*}
Here $\d\sig$ denotes the surface measure on $\pl E$.
The Neumann datum $\psi$ is reflected in the boundary integral on the right-hand side.
Moreover, the initial datum is taken 
as in the Dirichlet problem.

A function $u$ that is both a weak sub-solution and a weak super-solution
to \eqref{Neumann} is a weak solution.


\medskip
{\it Acknowledgement.} 
V.~B\"ogelein and N.~Liao have been supported by the FWF-Project P31956-N32
``Doubly nonlinear evolution equations".

\section{Some Technical Tools}
For $k, w\in\rr$ we define two quantities
\begin{equation*}
	\mathfrak g_\pm (w,k)=\pm (p-1)\int_{k}^{w}|s|^{p-2}(s-k)_\pm\,\ds.
\end{equation*}
Note that $\mathfrak g_\pm (w,k)\ge 0$.
For $b\in\rr$ and $\al>0$, we will embolden $\boldsymbol{b}^\al$ to denote the
signed $\al$-power of $b$ as 
\begin{align}\label{Eq:power}
\boldsymbol{b}^\al=
\left\{
\begin{array}{cl}
|b|^{\al-1}b, &b\neq0,\\[5pt]
0, &b=0.
\end{array}\right.
\end{align}

The following lemma can be found in the literature; cf. \cite[Lemma~2.2]{Acerbi-Fusco} for $\alpha\in(0,1)$ and \cite[inequality~(2.4)]{Giaquinta-Modica} for $\alpha>1$.

\begin{lemma}\label{lem:Acerbi-Fusco}
For any $\alpha>0$, there exists a constant $\boldsymbol\gamma=\boldsymbol\gamma(\alpha)$ such that,
for all $a,b\in\rr$, the following inequality holds true:
\begin{align*}
	\tfrac1{\boldsymbol\gamma}\big|\power{b}{\alpha} - \power{a}{\alpha}\big|
	\le
	\big(|a| + |b|\big)^{\alpha-1}|b-a|
	\le
	\boldsymbol\gamma \big|\power{b}{\alpha} - \power{a}{\alpha}\big|.
\end{align*}
\end{lemma}
\noi Based on Lemma~\ref{lem:Acerbi-Fusco}, we prove the following.
\begin{lemma}\label{lem:g}
There exists a constant $\boldsymbol\gamma=\boldsymbol\gamma(p)$ such that,
for all $w,k\in\rr$, the following inequality holds true:
\begin{align*}
	\tfrac1{\boldsymbol\gamma} \big(|w| + |k|\big)^{p-2}(w-k)_\pm^2
	\le
	\mathfrak g_\pm (w,k)
	\le
	\boldsymbol\gamma \big(|w| + |k|\big)^{p-2}(w-k)_\pm^2
\end{align*}
\end{lemma}

\begin{proof}
We only consider $\mathfrak g_-$, since the estimate for $\mathfrak g_+$ is analogous. If $k\le w$, then $\mathfrak g_-(w,k)=0=(w-k)_-$. Therefore it is enough to consider $w,k\in\rr$ with $w<k$. Here, we have 
\begin{align*}
	\mathfrak g_- (w,k)
	&=
	(p-1)\int_{w}^k|s|^{p-2}(k-s)\,\ds \\
	&\ge
	(p-1)\int_{w}^{\frac12(k+w)}|s|^{p-2}(k-s)\,\ds \\
	&\ge
	\tfrac{p-1}2 (k-w) \int_{w}^{\frac12(k+w)}|s|^{p-2}\,\ds .
\end{align*}
Note that $p-2>-1$ and therefore the integral on the right-hand side exists. With Lemma~\ref{lem:Acerbi-Fusco} we thus obtain
\begin{align*}
	\mathfrak g_- (w,k)
	&\ge
	\tfrac{1}2 (k-w) |s|^{p-2}s \Big|_{w}^{\frac12(k+w)}\\
	&\ge
	\tfrac{1}{\boldsymbol\gamma (p)} (k-w) 
	\big(|w|+\tfrac12|k+w|\big)^{p-2} \big(\tfrac12(k+w)-w\big) \\
	&=
	\tfrac{1}{2\boldsymbol\gamma (p)}(k-w)^2 
	\big(|w|+\tfrac12|k+w|\big)^{p-2} \\
	&\ge 
	\tfrac{1}{\boldsymbol\gamma (p)}(k-w)^2 
	\big(|w|+|k|\big)^{p-2} .
\end{align*}
In the last line we have used $\tfrac12(|w|+|k|)\le |w|+\tfrac12|k+w|\le 2(|w|+|k|)$. 
This proves the lower bound for $\mathfrak g_- (w,k)$. For the upper bound we again apply Lemma~\ref{lem:Acerbi-Fusco} and obtain
\begin{align*}
	\mathfrak g_- (w,k)
	&=
	(p-1)\int_{w}^k|s|^{p-2}(k-s)\,\ds \\
	&\le
	(p-1)(k-w) \int_{w}^{k}|s|^{p-2}\,\ds \\
	&=
	(k-w) |s|^{p-2}s \Big|_{w}^{k}\\
	&\le
	\boldsymbol\gamma (p)(k-w)^2 
	\big(|w|+|k|\big)^{p-2} .
\end{align*}
This finishes the proof of the lemma.
\end{proof}
The time derivative of a weak solution exists in the sense of distribution only.
However we often need to use $u$ in the test function and thus
the term $u_t$ appears in the integral weak formulation of solution, 
which is not granted by the preset notion of solution. 
In order to overcome the lack of regularity in the time variable,
we define the following mollification in time:
\begin{equation}\label{def:mol}
	\llbracket v \rrbracket_h(x,t)
	\df{=} 
	\tfrac 1h \int_0^t \mathrm e^{\frac{s-t}h} v(x,s) \, \ds\quad\text{ for any }v\in L^1(E_T).
\end{equation}
Properties of this mollification can for instance be found in \cite{KL}.

\section{Energy Estimates}

In this section we exploit the property of weak sub(super)-solutions 
in order to deduce certain energy estimates.
We emphasize the different roles played by sub-solutions and super-solutions.
When we state {\it ``$u$ is a sub(super)-solution...''}
and use $``\pm"$ or $``\mp"$ in what follows, we mean the sub-solution corresponds to
the upper sign and the super-solution corresponds to the lower sign in the statement.

First of all, we present energy estimates for local weak sub(super)-solutions defined in Section~\ref{S:1:2:1}.

\begin{proposition}\label{Prop:2:1}
	Let $u$ be a  local weak sub(super)-solution to \eqref{Eq:1:1} -- \eqref{Eq:1:2p} in $E_T$.
	There exists a constant $\boldsymbol \gm (C_o,C_1,p)>0$, such that
 	for all cylinders $Q_{R,S}=K_R(x_o)\times (t_o-S,t_o)\Subset E_T$,
 	every $k\in\rr$, and every non-negative, piecewise smooth cutoff function
 	$\z$ vanishing on $\pl K_{R}(x_o)\times (t_o-S,t_o)$,  there holds
\begin{align*}
	\essup_{t_o-S<t<t_o}&\int_{K_R(x_o)\times\{t\}}	
	\z^p\mathfrak g_\pm (u,k)\,\dx
	+
	\iint_{Q_{R,S}}\z^p|D(u-k)_\pm|^p\,\dx\dt\\
	&\le
	\boldsymbol \gm\iint_{Q_{R,S}}
		\Big[
		(u-k)^{p}_\pm|D\z|^p + \mathfrak g_\pm (u,k)|\partial_t\z^p|
		\Big]
		\,\dx\dt\\
	&\phantom{\le\,}
	+\int_{K_R(x_o)\times \{t_o-S\}} \z^p \mathfrak g_\pm (u,k)\,\dx.
\end{align*}
\end{proposition}

\begin{proof}We only consider the case of a local weak sub-solution. 
Recall the definition \eqref{Eq:power} for the $\al$-power of a possibly negative number and define $ w_h$ via $\power{w_h}{p{-}1}:=\llbracket \power{u}{p{-}1}\rrbracket_h$, where $\llbracket \cdot\rrbracket_h$ denotes the time mollification from \eqref{def:mol}.
From the weak form of the differential inequality for sub-solutions we deduce the mollified version (cf. \cite{KL})
\begin{align}\label{mol-eq} 
	\iint_{E_T} 
	\Big[\partial_t \power{w_h}{p{-}1} \varphi &+ 
	\llbracket\mathbf A(x,t,u,Du)\rrbracket_h\cdot D\varphi \Big]\dx\dt 
	\nonumber\\
	&\le
	\int_E \power{u}{p{-}1}(x,0)\cdot \tfrac1h\int_0^T \mathrm e^{-\frac sh}\varphi(x,s)\,\ds\,\dx,
\end{align}
for any non-negative test function $\varphi\in L^p(0,T;W^{1,p}_0(E))$.  
Let $Q_{R,S}=K_R (x_o)\times (t_o-S,t_o]\Subset E_T$
as in the statement of the energy estimate.  Let $\zeta\in C^1(Q_{R,S},[0,1])$ 
be a cutoff function vanishing on $\pl K_{R}(x_o)\times (t_o-S,t_o)$.
Furthermore, for  fixed $t_o-S<t_1<t_2 <t_o$ and $\varepsilon >0$ small enough we define the cutoff  function in time $\psi_\varepsilon \in W^{1,\infty} \big((t_o-S,t_o),[0,1]\big)$ by 

\begin{align*}
\psi_\varepsilon(t) :=
\left\{
\begin{array}{cl}
	0, & \text{ for } t_o - S \le t \leq t_1-\varepsilon, \vspace{5pt}\\
	1+ \frac{t-t_1}{\epsilon},  & \text{ for } t_1-\varepsilon < t \leq t_1, \vspace{5pt}\\
	1,	& \text{ for } t_1 < t \leq t_2,\vspace{5pt}\\
	1-\frac{t-t_2}{\epsilon}, & \text{ for } t_2 < t \leq t_2 +\epsilon, \vspace{5pt}\\
	0, & \text{ for } t_2+\epsilon < t \le t_o.
\end{array}
\right.
\end{align*} 
Now, we choose in \eqref{mol-eq}  the testing function
\begin{equation}\label{Test}
        Q_{R,S}\ni(x,t)\mapsto
	\varphi(x,t) 
	= 
	\zeta^p(x,t)\psi_\varepsilon(t) \big(u(x,t)-k\big)_+. 
\end{equation}
In the following we omit in the notation the reference to the center $z_o=(x_o,t_o)$. 
Moreover, we observe that 
\begin{equation*}
	\mathfrak g_+ (w,k)= (p-1)\int_{k}^{w}|s|^{p-2}(s-k)_+\,\ds
	=\int_{{\bm k}^{p-1}}^{{\bm w}^{p-1}}\big(\power{s}{\frac{1}{p-1}}-k\big)_+\,\ds,
\end{equation*}
which can be seen by substituting $\sigma:=\power{s}{p-1}$. Note that the mapping
$\rr\ni s\mapsto \phi(s)=\power{s}{p-1}$ is increasing with derivative $\phi'(s)=(p-1)|s|^{p-2}$
($s\not=0$ in the case $p<2$). For the integral in \eqref{mol-eq} containing the time derivative we compute 
\begin{align*}
	\iint_{E_T} 
	\partial_t \power{w_h}{p{-}1} \varphi \,\dx\dt 
	&= 
	\iint_{Q_{R,S}}
	\zeta^p  \psi_\varepsilon \partial_t \power{w_h}{p{-}1} (w_h-k)_+ 
	\dx\dt \\
	&\phantom{=\,} +
	\iint_{Q_{R,S}}
	\zeta^p  \psi_\varepsilon \partial_t \power{w_h}{p{-}1} \big( (u-k)_+-(w_h-k)_+\big) 
	\dx\dt \\
	&\ge 
	\iint_{Q_{R,S}}
	\zeta^p  \psi_\varepsilon 
	\partial_t\mathfrak g_+ (w_h,k) \dx\dt \\
	& = 
	- \iint_{Q_{R,S}} 
	\big(\zeta^p  \psi_\varepsilon'+ \psi_\varepsilon\partial_t\zeta^p\big) \mathfrak g_+ (w_h,k) 
	\dx\dt.
\end{align*}
Here we used in the second line the identity
\begin{align}\label{dt-wh}
	 \partial_t \power{w_h}{p{-}1}  
	 &=
	 \tfrac{1}{h} \big(\power{u}{p{-}1}- \power{w_h}{p-1}\big) ,
\end{align}
and the fact that the map $\tau\mapsto (\power{\tau}{\frac1{p-1}}-k)_+$ is a monotone increasing function, implying that the term in the second last line of the above inequality is non-negative. Since $\power{w_h}{p{-}1}=\llbracket \power{u}{p{-}1}\rrbracket_h\to \power{u}{p{-}1}$ in $L^{\frac{p}{p-1}}(\Omega_T)$ we can pass to the limit $h\downarrow 0$ in the integral on the right-hand side. We therefore get
\begin{align*}
	\liminf_{h\downarrow 0}
	\iint_{Q_{R,S}} 
	\partial_t \power{w_h}{p{-}1}\varphi \,\dx\dt 
	&\ge
	- \iint_{Q_{R,S}}
	\big(\zeta^p \psi_\varepsilon' + \psi_\varepsilon\partial_t\zeta^p\big)
	\mathfrak g_+ (u,k) \,\dx\dt \\
	&=: 
	-\big[\mathbf{I}_{\varepsilon} +\mathbf{II}_{\varepsilon}\big],
\end{align*}
with the obvious meaning of $\mathbf{I}_{\varepsilon}$ and $\mathbf{II}_{\varepsilon}$. We now pass to the limit $\varepsilon \downarrow 0$.
 For the term $\mathbf{I}_{\varepsilon}$ we obtain for any $t_o-S<t_1<t_2<t_o$
 that
$$
 	\lim_{\varepsilon \downarrow 0}\mathbf{I}_\varepsilon 
	=
	\int_{K_R} \zeta^p(x,t_1) \mathfrak g_+(u(x,t_1),k)\,\dx 
	-
	\int_{K_R} \zeta^p(x,t_2) \mathfrak g_+(u(x,t_2),k)\,\dx 
	,
$$
while for $\mathbf{II}_{\varepsilon}$ we have
\begin{align*}
	\lim_{\varepsilon \downarrow 0}\mathbf{II}_{\varepsilon}
	=
	\iint_{K_R\times (t_1,t_2)} \partial_t\zeta^p\mathfrak g_+(u,k)\,\dx\dt.
\end{align*}
Next, we observe that the boundary term in \eqref{mol-eq} disappears as $h\downarrow 0$, since  by construction $\varphi(\cdot ,0)\equiv 0$ on $E$, i.e.~we have
\begin{equation*}	
	\lim_{h\downarrow 0}
	\int_E \power{u}{p{-}1}(x,0)\cdot \tfrac1h\int_0^T \mathrm e^{-\frac sh}\varphi(x,s)\,\ds\,\dx
	=
	\int_E \power{u}{p{-}1}(x,0)\varphi(x,0)\,\dx=0.
\end{equation*}
It remains to consider the diffusion term.
After passing to the limit $h\downarrow 0$, we use the ellipticity and growth assumption \eqref{Eq:1:2p} for the vector-field $\mathbf A$, and subsequently Young's inequality to the integral containing $(u-k)_+$ and $D(u-k)_+$. In this way we obtain one term that we can absorb in the term arising from the ellipticity condition, the other one is shifted later on to the right hand side.
\begin{align*}
	&\lim_{h\downarrow 0}
	\iint_{E_T} 
	\llbracket\mathbf A(x,t,u,Du)\rrbracket_h\cdot D\varphi\, \dx\dt \\
	&\qquad=
	\iint_{Q_{R,S}} 
	 \psi_\varepsilon \mathbf A (x,t,u,Du) \cdot  
	\big[\zeta^p D(u-k)_+ + p\zeta^{p-1} (u-k)_+D\zeta\big] 
	\dx\dt \\
	&\qquad\geq  
	C_o \iint_{Q_{R,S}}
	\zeta^p \psi_\varepsilon|D(u-k)_+|^p \dx\dt\\
	&\qquad\qquad\qquad -
	C_1p\iint_{Q_{R,S}} 
	\zeta^{p-1} \psi_\varepsilon |D\zeta| (u-a)_+ |D(u-k)_+|^{p-1}
	\dx\dt \\
	&\qquad\geq   
	\frac{C_o}{p} \iint_{Q_{R,S}}
	\zeta^p  \psi_\varepsilon |D(u-k)_+|^p \dx\dt -
	\boldsymbol\gamma \iint_{Q_{R,S}}\psi_\varepsilon|D\zeta|^p (u-k)^p_+ \dx\dt.
\end{align*}
Combining the preceding estimates and letting $\varepsilon\downarrow 0$  we arrive at
\begin{align*}
		&\int_{K_R\times\{t_2\}} \zeta^p \mathfrak g_+(u,k)\,\dx
		+
		\frac{C_o}{p} \iint_{K_R\times (t_1,t_2)} \zeta^p |D(u-k)_+|^p \dx\dt\\
		&\qquad\le 
	          \iint_{K_R\times (t_1,t_2)}\Big[ \boldsymbol\gamma |D\zeta|^p(u-k)^p_+ 
	         +
	          \partial_t\zeta^p\mathfrak g_+(u,k)\Big]\,\dx\dt +
	         \int_{K_R\times\{t_1\}} \zeta^p \mathfrak g_+(u,k)\,\dx,
\end{align*}
whenever $t_o-S<t_1<t_2<t_o$. The constant $\boldsymbol\gamma$  in the first integral on the right hand side depends only on $p, C_o$ and $C_1$. At this point, a standard argument finishes the proof. We first pass in the last inequality to the limit $t_1\downarrow t_o-S$.  This poses no problem since $u\in C\big( 0,T;L^p_{\rm loc}(E)\big)$. We obtain 
\begin{align*}
		&\int_{K_R\times\{t_2\}} \zeta^p \mathfrak g_+(u,k)\,\dx
		+
		\frac{C_o}{p} \iint_{K_R\times (t_o-S,t_2)} \zeta^p |D(u-k)_+|^p \dx\dt\\
		&\qquad\le 
	          \iint_{Q_{R,S}}\Big[ \boldsymbol\gamma |D\zeta|^p (u-k)^p_+
	         +
	          \partial_t\zeta^p\mathfrak g_+(u,k)\Big]\,\dx\dt
	           +
	         \int_{K_R\times\{t_o-S\}} \zeta^p \mathfrak g_+(u,k)\,\dx,
\end{align*}
Here, we discard the second integral on the left-hand side and take  then the essential supremum with respect to $t_2\in (t_o-S,t_o)$. This leads to an estimate of the essential supremum of the first integral. On the other hand, discarding the first integral and passing to the limit $t_2\uparrow t_o$ we deduce a similar estimate for the second integral of the left-hand side. 
Together, this gives
 \begin{align*}
		&\essup_{t_o-S<t<t_o}\int_{K_R\times\{t\}} \zeta^p \mathfrak g_+(u,k)\,\dx
		+
		\frac{C_o}{p} \iint_{K_R\times (t_o-S,t_o)} \zeta^p |D(u-k)_+|^p \dx\dt\\
		&\qquad\le 
	          \iint_{K_{R,S}}\Big[ \boldsymbol\gamma|D\zeta|^p (u-k)^p_+ 
	         +
	          \partial_t\zeta^p\mathfrak g_+(u,k)\Big]\,\dx\dt
	          +
	         \int_{K_R\times\{t_o-S\}} \zeta^p \mathfrak g_+(u,k)\,\dx.
\end{align*}
This finishes the proof of the energy estimate.
\end{proof}
Next, we consider the situation near the initial level $t=0$
when a continuous datum $u_o$ is prescribed.
We work in a cylinder $K_R(x_o)\times (0,S)\subset E_T$, which lies
on the bottom of $E_T$. Let $\z$ be a non-negative, piecewise smooth cutoff function
that is independent of $t$ and vanishes on $\pl K_{R}(x_o)$.
The conclusion of Proposition~\ref{Prop:2:1} holds in any cylinder satisfying
$K_R(x_o)\times (t_1,S)\Subset E_T$. 
Suppose the level $k$ satisfies
\begin{equation}\label{Eq:3:2}
	\left\{
	\begin{aligned}
	&k\ge\sup_{K_R(x_o)}u_o\quad\text{ for sub-solutions},\\
	&k\le\inf_{K_R(x_o)}u_o\quad\text{ for super-solutions}.
	\end{aligned}
	\right.
\end{equation}
Then in view of the initial datum $u_o$ being taken in the topology of $L^{\hat{p}}_{\loc}(E)$
and letting $t_1\downarrow0$, it is not hard to verify that the space integral on the right-hand side at the time level $t_1$
will tend to zero. Consequently, we arrive at
\begin{proposition}\label{Prop:2:2}
	Let $u$ be a  local weak sub(super)-solution to \eqref{Dirichlet} with \eqref{Eq:1:2p} in $E_T$.
	There exists a constant $\boldsymbol \gm (C_o,C_1,p)>0$, such that
 	for all cylinders $K_R(x_o)\times (0,S)\subset E_T$,
 	every $k\in\rr$ satisfying \eqref{Eq:3:2}
	and every non-negative, piecewise smooth cutoff function
 	$\z$ independent of $t$ and vanishing on $\pl K_{R}(x_o)$,  there holds
\begin{align*}
	\essup_{0<t<S}&\int_{K_R(x_o)\times\{t\}}	
	\z^p\mathfrak g_\pm (u,k)\,\dx
	+
	\iint_{K_R(x_o)\times (0,S)}\z^p|D(u-k)_\pm|^p\,\dx\dt\\
	&\le
	\boldsymbol \gm\iint_{K_R(x_o)\times (0,S)}
		(u-k)^{p}_\pm|D\z|^p 
		\,\dx\dt
\end{align*}
\end{proposition}

Now we turn our attention to the energy estimates near $S_T$.
We first deal with Dirichlet data.
When we run the calculation in the proof of Proposition~\ref{Prop:2:1}
within $Q_{R,S}=K_R(x_o)\times (t_o-S,t_o)$ for some $(x_o,t_o)\in S_T$, we need to assume
certain restrictions on the level $k$, i.e.,
\begin{equation}\label{Eq:3:3}
	\left\{
	\begin{aligned}
	&k\ge\sup_{Q_{R,S}\cap S_T}g\quad\text{ for sub-solutions},\\
	&k\le\inf_{Q_{R,S}\cap S_T}g\quad\text{ for super-solutions}.
	\end{aligned}
	\right.
\end{equation}
In such a way, the test functions in \eqref{Test}
\[
Q_{R,S}\cap E_T\ni(x,t)\mapsto\varphi(x,t) 
	= 
	\zeta^p(x,t)\psi_\varepsilon(t) \big(u(x,t)-k\big)_\pm
\]
become admissible as the functions $x\mapsto\big(u(x,t)-k\big)_\pm$ vanish
on $Q_{R,S}\cap S_T$ in the sense of traces for a.e. $t\in(t_o-S,t_o)$.
This fact does not require any smoothness of $\pl E$ (cf. \cite[Lemma~2.1]{GLL}). 
As a result, we have
\begin{proposition}\label{Prop:2:3}
	Let $u$ be a  local weak sub(super)-solution to \eqref{Dirichlet} with \eqref{Eq:1:2p} in $E_T$.
	There exists a constant $\boldsymbol \gm (C_o,C_1,p)>0$, such that
 	for all cylinders $Q_{R,S}=K_R(x_o)\times (t_o-S,t_o)$ with the vertex $(x_o,t_o)\in S_T$,
 	every $k\in\rr$ satisfying \eqref{Eq:3:3}, and every non-negative, piecewise smooth cutoff function
 	$\z$ vanishing on $\pl K_{R}(x_o)\times (t_o-S,t_o)$,  there holds
\begin{align*}
	\essup_{t_o-S<t<t_o}&\int_{\{K_R(x_o)\cap E\}\times\{t\}}	
	\z^p\mathfrak g_\pm (u,k)\,\dx
	+
	\iint_{Q_{R,S}\cap E_T}\z^p|D(u-k)_\pm|^p\,\dx\dt\\
	&\le
	\boldsymbol \gm\iint_{Q_{R,S}\cap E_T}
		\Big[
		(u-k)^{p}_\pm|D\z|^p + \mathfrak g_\pm (u,k)|\partial_t\z^p|
		\Big]
		\,\dx\dt\\
	&\phantom{\le\ }
	+\int_{\{K_R(x_o)\cap E\}\times \{t_o-S\}} \z^p \mathfrak g_\pm (u,k)\,\dx.
\end{align*}
\end{proposition}

Finally, we deal with the energy estimates for the Neumann problem \eqref{Neumann}.
Like before, we consider the problem in $Q_{R,S}=K_R(x_o)\times(t_o-S,t_o)$ with $(x_o,t_o)\in S_T$
and may assume $(x_o,t_o)=(0,0)$. 
  For a cutoff function $\z$ as in Proposition~\ref{Prop:2:3},
  a similar procedure as in the proof of Proposition~\ref{Prop:2:1}
  will give us that
   \begin{align*}
		\essup_{t_o-S<t<t_o}&\int_{\{K_R\cap E\}\times\{t\}} \zeta^p \mathfrak g_+(u,k)\,\dx
		+
		 \iint_{Q_{R,S}\cap E_T} \zeta^p |D(u-k)_+|^p \dx\dt\\
		&\le 
	         \boldsymbol\gamma \iint_{Q_{R,S}\cap E_T}\Big[ (u-k)^p_+|D\zeta|^p 
	         +
	          \partial_t\zeta^p\mathfrak g_+(u,k)\Big]\,\dx\dt
	          \\
	         &\phantom{\le\,}
	         +\boldsymbol\gamma\iint_{Q_{R,S}\cap S_T}\z^p\psi(x,t,u)(u-k)_+\,\d \sig\dt\\
	        &\phantom{\le\,}
	         +
	         \int_{\{K_R\cap E\}\times\{t_o-S\}} \zeta^p \mathfrak g_+(u,k)\,\dx.
\end{align*}
Now we make use of \eqref{N-data}, apply the {\it trace inequality} (cf. \cite[Proposition~18.1]{DB-RA})
for each time slice and then integrate in time, and use Young's inequality
to estimate the boundary integral:
\begin{align*}
&\iint_{Q_{R,S}\cap S_T}\psi(x,t,u)(u-k)_+\z^p\,\d\sig\dt\\
&\quad\le C_2\iint_{\pl (K_R\cap E)\times (t_o-S,t_o)}(u-k)_+\z^p\,\d \sig\dt\\
&\quad\le \boldsymbol\gamma C_2\iint_{Q_{R,S}\cap E_T}
\Big[|D(u-k)_+|\z^p+(u-k)_+\big(\z^p+|D\z^p|\big)\Big]\,\dx\dt\\
&\quad\le\frac{C_o}{2p}\iint_{Q_{R,S}\cap E_T} \zeta^p |D(u-k)_+|^p \dx\dt
  +\boldsymbol\gamma \iint_{Q_{R,S}\cap E_T} (u-k)^p_+|D\zeta|^p\,\dx\dt\\
  &\qquad+\boldsymbol\gamma  C_2^{\frac{p}{p-1}}\iint_{Q_{R,S}\cap E_T} \z^p\chi_{\{u>k\}}\,\dx\dt.
\end{align*}
Hence, collecting the above estimates we arrive at
\begin{proposition}\label{Prop:2:4}
	Let $u$ be a  local weak sub(super)-solution to \eqref{Neumann} with \eqref{Eq:1:2p} in $E_T$.
	Assume $\pl E$ is of class $C^1$ and \eqref{N-data} holds.
	There exists a constant $\boldsymbol \gm (C_o,C_1,p)>0$, such that
 	for all cylinders $Q_{R,S}=K_R(x_o)\times (t_o-S,t_o)$ with the vertex $(x_o,t_o)\in S_T$,
 	every $k\in\rr$, and every non-negative, piecewise smooth cutoff function
 	$\z$ vanishing on $\pl K_{R}(x_o)\times (t_o-S,t_o)$,  there holds
 \begin{align*}
		\essup_{t_o-S<t<t_o}&\int_{\{K_R(x_o)\cap E\}\times\{t\}} \zeta^p \mathfrak g_\pm(u,k)\,\dx
		+
		 \iint_{Q_{R,S}\cap E_T} \zeta^p |D(u-k)_\pm|^p \dx\dt\\
		&\le 
	         \boldsymbol\gamma \iint_{Q_{R,S}\cap E_T}\Big[ (u-k)^p_\pm|D\zeta|^p 
	         +
	         \mathfrak g_\pm(u,k)| \partial_t\zeta^p|\Big]\,\dx\dt
	          \\
	         &\phantom{\le\,}
	     +\boldsymbol\gamma C_2^{\frac{p}{p-1}}\iint_{Q_{R,S}\cap E_T} \z^p\chi_{\{(u-k)_{\pm}>0\}}\,\dx\dt\\
	        &\phantom{\le\,}
	         +
	         \int_{\{K_R(x_o)\cap E\}\times\{t_o-S\}} \zeta^p \mathfrak g_\pm(u,k)\,\dx.
\end{align*}
\end{proposition}
\section{Expansion of Positivity}

We first introduce the notation that is used throughout this section. 
For a compact set $K\subset\rr^N$ and
 a cylinder $\mathcal{Q}\df{=}K\times(T_1,T_2]\subset E_T$
we introduce numbers $\boldsymbol\mu^{\pm}$ and $\boldsymbol\om$ satisfying
\begin{equation*}
	\boldsymbol\mu^+\ge\essup_{\mathcal{Q}}u,
	\quad 
	\boldsymbol\mu^-\le\essinf_{\mathcal{Q}} u,
	\quad
	\boldsymbol\om\ge\boldsymbol\mu^+-\boldsymbol\mu^-.
\end{equation*}
We also assume $(x_o,t_o)\in \mathcal{Q}$, such that the forward
cylinder 
\begin{equation}\label{Eq:1:5}
	K_{8\varrho}(x_o)\times\big(t_o,t_o+(8\varrho)^p\big]\subset\mathcal{Q}.
\end{equation}
Next we state our main proposition of this section, which will be the main ingredient
in the proof of Theorem~\ref{Thm:1:1}.

\begin{proposition}\label{Prop:1:1}
	Let $u$ be a locally bounded, local, weak sub(super)-solution to \eqref{Eq:1:1} -- \eqref{Eq:1:2p} in $E_T$.
	Suppose for some $(x_o,t_o)\in E_T$, $M>0$, $\al\in(0,1)$ and $\varrho>0$ we have \eqref{Eq:1:5} and
	\begin{equation*}
		\left|\left\{\pm\big(\boldsymbol \mu^{\pm}-u(\cdot, t_o)\big)\ge M\right\}\cap K_\varrho(x_o)\right|
		\ge
		\al \big|K_\varrho\big|.
	\end{equation*}
There exist constants $\xi$, $\dl$ and $\eta$ in $(0,1)$ depending only on the data and $\al$,
such that either
\begin{equation*}
	\big|\boldsymbol\mu^{\pm}\big|>\xi M
\end{equation*}
 or
\begin{equation*}
	\pm\big(\boldsymbol \mu^{\pm}-u\big)\ge\eta M
	\quad
	\mbox{a.e.~in $K_{2\varrho}(x_o)\times\big(t_o+\dl(\tfrac12\varrho)^p,t_o+\dl\varrho^p\big],$}
\end{equation*}
where 
$$
	\xi
	=
	\left\{
	\begin{array}{cl}
	2\eta, & \mbox{if $p>2$,} \\[5pt]
	8, & \mbox{if $1<p\le 2$.}
	\end{array}
	\right.
$$
\end{proposition}
The proof of Proposition~\ref{Prop:1:1} is a straightforward consequence
of Lemmas~\ref{Lm:3:1} -- \ref{Lm:3:3} in the following sections.
Before presenting proofs, some remarks are in order.
\begin{remark}\label{Rem:1.1}\upshape
	By repeated applications of Proposition \ref{Prop:1:1},
	we could conclude that for an arbitrary $A>1$,
	there exists some $\bar\eta\in (0,1)$ depending on $\al$, the data and also on $A$, such that
	\begin{equation*}
	\pm\big(\boldsymbol\mu^{\pm}-u\big)\ge\bar\eta M
	\quad
	\mbox{a.e.~in $K_{2\varrho}(x_o)\times(t_o+\varrho^p,t_o+A\varrho^p]$,}
\]
provided this cylinder is included in $\mathcal{Q}$ and $|\boldsymbol\mu^{\pm}|<\xi M$, where $\xi$ is the parameter from Proposition~\ref{Prop:1:1}.
\end{remark}
\begin{remark}\upshape
Proposition~\ref{Prop:1:1} exhibits the spread of pointwise positivity both in time and in space.
Nonetheless, in the proof of H\"older regularity 
we only need the positivity in a smaller cube than the one of the initial measure information assigned.
Incidentally, Proposition~\ref{Prop:1:1} can be exploited to give an alternative proof of Harnack's inequality
for non-negative solutions (cf. \cite{GV, KK,Trud0}). We will present it in Appendix~\ref{Append:2}.
\end{remark}
\begin{remark}\upshape
The statement of Proposition~\ref{Prop:1:1} presents
an either-or form. This is typical when power-like nonlinearity appears in an equation
and as a result, adding a constant to a solution to \eqref{Eq:1:1} -- \eqref{Eq:1:2p}
in general does not yield another solution to the same equation.
We mention that a proposition of similar type has been used in \cite{Liao}
to deal with the H\"older regularity for the porous medium type equation.
\end{remark}
\begin{remark}\label{Rmk:4:4}\upshape
Up to a proper adjustment of coefficients, 
the prototype equation of \eqref{Eq:1:1} -- \eqref{Eq:1:2p} can be written
in a formal way as
\[
|u|^{p-2} u_t=\Dl_p u.
\]
Seemingly the equation is homogeneous in $u$, and  cylinders of the type $Q_\rho=K_\rho\times(-\rho^p,0]$
appear to be the correct ones to examine the equation.
However, a careful inspection of proofs of Lemma~\ref{Lm:3:1} -- \ref{Lm:3:3}
will reveal a subtle yet crucial difference between the role of $u$ in the absolute value on the left-hand side 
and that of others. Loosely speaking, $|u|$ represents the extrema, while other $u$'s stand for the oscillation.
Notice also that when we apply Proposition~\ref{Prop:1:1} to prove Theorem~\ref{Thm:1:1},
the typical $M$ will be $a\boldsymbol\om$ for some $a$ in $(0,1)$.
Under this point of view, the above equation can be interpreted in a probably improper but heuristic manner that
\[
[\boldsymbol\mu]^{p-2}\frac{[\boldsymbol\om]}{[t]}=\frac{[\boldsymbol\om]^{p-1}}{[x]^p},
\]
assuming the symbols are self-suggestive.
This hints the correct cylinder to examine the equation is actually
\[
Q_{\rho}(\theta)=K_\rho\times(-\theta\rho^p,0]
\quad\text{ where }\theta=\left(\frac{[\boldsymbol\om]}{[\boldsymbol\mu]}\right)^{2-p}.
\]
The time scaling by $\theta$ reflects the competition between $[\boldsymbol\om]$ and $[\boldsymbol\mu]$
via either-or alternatives.
A closer inspection of the proofs of Lemma~\ref{Lm:3:1} and Lemma~\ref{Lm:3:3}
shows it is required that $\theta\simeq1$, while the proof of Lemma~\ref{Lm:3:2}
uses $\theta\lesssim1$ only. This explains why the quantities $\boldsymbol\mu^{\pm}$
enter into Proposition~\ref{Prop:1:1} via an either-or form, such that
they are comparable with $\xi\boldsymbol\om$ for
$\xi= 2\eta$ when $p>2$, whereas $\xi=8$ suffices when $1<p<2$.
\end{remark}

\subsection{Propagation of Positivity in Measure}
\begin{lemma}\label{Lm:3:1}
	Let $M>0$ and $\al\in(0,1)$. Then, there exist $\dl$ and $\eps$ in $(0,1)$,
	depending only on the data and $\al$, such that whenever $u$ is a locally bounded, local, 
	weak sub(super)-solution to \eqref{Eq:1:1} -- \eqref{Eq:1:2p} in $E_T$ satisfying
	\begin{equation*}
	\Big|\Big\{
		\pm\big(\boldsymbol \mu^{\pm}-u(\cdot, t_o)\big)\ge M
		\Big\}\cap K_{\varrho}(x_o)\Big|
		\ge\al \big|K_{\varrho}\big|,
	\end{equation*}
	then either $$|\boldsymbol \mu^{\pm}|>8M$$ or
	\begin{equation}\label{Eq:3:1}
	\Big|\Big\{
	\pm\big(\boldsymbol \mu^{\pm}-u(\cdot, t)\big)\ge \eps M\Big\} \cap K_{\varrho}(x_o)\Big|
	\ge\frac{\al}2 |K_\varrho|
	\quad\mbox{ for all $t\in(t_o,t_o+\dl\varrho^p]$.}
\end{equation}
\end{lemma}
\begin{proof} We only show the case of super-solutions, the other case of sub-solutions being similar.
Assume $(x_o,t_o)=(0,0)$ and $|\boldsymbol \mu^{-}|\le 8M$. Otherwise there is nothing to prove.
Use the energy estimate in Proposition~\ref{Prop:2:1}
in the cylinder $Q=K_{\varrho}\times(0,\dl\varrho^p]$, with
$k=\boldsymbol\mu^-+M$ and choose a standard non-negative cutoff function 
$\z(x,t)\equiv \z(x)$ independent of time that equals $1$ on $K_{(1-\sig)\varrho}$ with $\sigma\in(0,1)$ to be chosen later 
and vanishes on $\pl K_{\varrho}$ satisfying
$|D\z|\le(\sig\varrho)^{-1}$;
in such a case, we have for all $0<t<\dl\varrho^p$, that
\begin{align}\label{start:en}
	\int_{K_\varrho\times\{t\}}\int_{u}^k &|s|^{p-2}(s-k)_-\,\ds \z^p\,\dx\nonumber\\
	&\le
	\int_{K_\varrho\times\{0\}}\int_{u}^k |s|^{p-2}(s-k)_-\,\ds\z^p\dx
	+
	\boldsymbol \gm\iint_{Q}(u-k)^{p}_-|D\z|^p\,\dx\dt.
\end{align}
In order to estimate the first integral on the right-hand side, we 
take into consideration the measure theoretical information
at the initial time $t=0$ and the fact $u\ge\boldsymbol\mu^-$. This leads to
\begin{align*}
	\int_{K_\varrho\times\{0\}}\int_{u}^k |s|^{p-2}(s-k)_-\,\ds\z^p\,\dx
	\le 
	(1-\al)|K_\varrho|\int_{\boldsymbol\mu^-}^{k} |s|^{p-2}(s-k)_-\,\ds.
\end{align*}
The second term on the right-hand side of \eqref{start:en} is estimated by
\begin{equation*}
	\iint_{Q}(u-k)^{p}_-|D\z|^p\,\dx\dt
	\le
	\frac{M^p}{(\sigma\varrho )^p}|Q| 
	= \frac{\dl M^{p}}{\sig^p} |K_\varrho |.
\end{equation*}
The left-hand side of the energy estimate \eqref{start:en} can be bounded from below by
\begin{align*}
	\int_{K_\varrho\times\{t\}}\int_{u}^k |s|^{p-2}(s-k)_-\,\ds \z^p\,\dx
	\ge 
	\big|A_{k_\eps,(1-\sig)\varrho}(t)\big| \int_{k_\eps}^k |s|^{p-2}(s-k)_-\,\ds
\end{align*}
where we have defined
\begin{equation*}
	A_{k_\eps,(1-\sig)\varrho}(t)
	=
	\big\{ u(\cdot,t)\le k_\eps\big\} \cap K_{(1-\sig)\varrho},
	\quad 
	\mbox{and}
	\quad
	k_\eps=\boldsymbol\mu^-+\eps M,
\end{equation*}
with $\epsilon\in(0,\frac12)$ to be chosen later. 
Due to Lemma~\ref{lem:g} and the fact that 
$\frac12M\le(1-\eps)M=k-k_\eps\le |k_\eps|+|k|\le 2(|\boldsymbol\mu^-|+ M)\le 18 M$ 
we can further estimate from below
\begin{align}\label{prop_1}
	\int_{k_\eps}^k|s|^{p-2}(s-k)_-\,\ds
	=
	\tfrac{1}{p-1}\, g_-(k_\eps,k)
	\ge 
	\tfrac{1}{\boldsymbol\gamma (p)}
	\big(|k_\eps|+|k|\big)^{p-2} (k-k_\eps)^2
	\ge
	\tfrac{1}{\boldsymbol\gamma (p)} M^p.
\end{align}
Notice that
\begin{align*}
	\big|A_{k_\eps,\varrho}(t)\big|
	&=
	\big|A_{k_\eps,(1-\sig)\varrho}(t)\cup (A_{k_\eps,\varrho}(t)\setminus A_{k_\eps,(1-\sig)\varrho}(t))\big|\\
	&\le 
	\big|A_{k_\eps,(1-\sig)\varrho}(t)\big|+|K_\varrho\setminus K_{(1-\sig)\varrho}|\\
	&\le
	\big |A_{k_\eps,(1-\sig)\varrho}(t)\big|+N\sig |K_\varrho|.
\end{align*}
Collecting all the above estimates yields that
\begin{align*}
	|A_{k_\eps,\varrho}(t)|
	&\le 
	\frac{\dsty\int_{\boldsymbol \mu^-}^k |s|^{p-2}(s-k)_-\,\ds }{\dsty\int_{k_\eps}^k |s|^{p-2}(s-k)_-\,\ds }
	(1-\al)|K_\varrho|
	+
	\frac{\boldsymbol\gm\dl}{\sig^p}|K_\varrho| +N\sig|K_\varrho|,
\end{align*}
for a constant $\boldsymbol\gamma =\boldsymbol\gamma (p,C_o,C_1)$.
The fractional number in the preceding  inequality can be rewritten in the form
\[
	1+I_\eps \quad\text{ where }\quad
	I_\eps =
	\frac{\dsty\int_{\boldsymbol \mu^-}^{k_\eps} |s|^{p-2}(s-k)_-\,\ds}{\dsty\int_{k_\eps}^k  |s|^{p-2}(s-k)_-\,\ds}.
\]
At this stage, we need to bound $I_\eps$ from above. Keeping in mind $|\boldsymbol \mu^-|\le 8M$ and $|k_\eps|\le 9M$ and applying Lemma~\ref{lem:Acerbi-Fusco}, we have

\[
	\int_{\boldsymbol \mu^-}^{k_\eps} |\tau|^{p-2}(\tau-k)_-\,\dtau
	\le 
	M\int_{\boldsymbol \mu^-}^{k_\eps}|\tau|^{p-2}\,\d\tau
	=
	M |s|^{p-2} s\Big|_{\boldsymbol \mu^-}^{k_\eps}
	\le
	\boldsymbol\gm (p)M^p\eps.
\]
Together with inequality \eqref{prop_1} we obtain 
\begin{align*}
	I_\eps
	\le
	\boldsymbol\gm (p) \eps.
\end{align*}
This allows us to choose the various parameters quantitatively. Indeed, 
we may  choose $\eps\in (0,1)$ small enough such that
\begin{equation*}
	(1-\al)(1+\boldsymbol \gm\eps)\le 1-\tfrac34\al.
\end{equation*}
This fixes $\eps$ as a constant depending only on $p$ and $\al$. Next, we define $\sig:=\frac{\al}{8N}$.
Finally, we choose $\delta\in (0,1)$ small enough so that
$$
	\frac{\boldsymbol \gm\dl}{\sig^p}\le\frac{\al}{8}.
$$
Note that this specifies $\dl$ as a constant depending  on the data and $\al$. With these choices we have  $|A_{k_\eps,\varrho}(t)|\le
(1-\tfrac{\al}2)|K_\varrho|$. This proves the asserted propagation of positivity \eqref{Eq:3:1}, as long as $0<t\le\delta\varrho^p$.
\end{proof}

\subsection{A Shrinking Lemma}

\begin{lemma}\label{Lm:3:2}
Suppose that in Lemma \ref{Lm:3:1} the second alternative \eqref{Eq:3:1} holds, let
$Q=K_{\varrho}(x_o)\times\left(t_o,t_o+\dl\rho^p\right]$ be the corresponding cylinder and let $\widehat Q=K_{4\varrho}(x_o)\times\left(t_o,t_o+\dl\rho^p\right]\Subset E_T$.
There exists $\boldsymbol \gm>0$ depending only on the data and $\alpha$, such that for any positive integer $j_*$,
if $1<p<2$, we have 
\begin{equation*}
	\bigg|\bigg\{
	\pm\big(\boldsymbol \mu^{\pm}-u\big)\le\frac{\eps M}{2^{j_*}}\bigg\}\cap \widehat Q\bigg|
	\le\frac{\boldsymbol\gm}{j_*^{\frac{p-1}p}}|\widehat Q|,\quad
\end{equation*}
while in the case $p>2$, the same conclusion holds provided
$|\boldsymbol \mu^{\pm}|<\eps M2^{-j_*}$.
\end{lemma}
\begin{proof}
We only show the case of super-solutions, the case of sub-solutions  being similar.
Moreover, we assume $(x_o,t_o)=(0,0)$.
We employ the energy estimate in Proposition~\ref{Prop:2:1} in $K_{8\rho}\times (0,\delta\rho^p]$ 
with levels
\begin{equation*}
	k_j:=\boldsymbol \mu^-+\frac{\eps M}{2^{j}},\quad j=0,1,\cdots, j_*,
\end{equation*}
 and introduce
 a  cutoff function $\z$ in $K_{8\varrho}$ (independent of $t$) that is equal to $1$ in $K_{4\varrho}$ and vanishes
on $\pl K_{8\varrho}$, such that $|D\z|\le\varrho^{-1}$.
Then, we obtain
\begin{align*}
	\iint_{\widehat Q} |D(u-k_j)_-|^p\,\dx\dt
	\le
	\int_{K_{8\varrho}\times\{0\}} \mathfrak g_- (u,k_j) \,\dx
	+ 
	\frac{\boldsymbol\gm}{\rho^p}
	\iint_{K_{8\rho}\times (0,\delta\rho^p]}(u-k_j)_-^p \,\dx\dt.
\end{align*}
Now we treat the individual terms of the right side separately. We begin with the first one. Due to Lemma~\ref{lem:g} we have 
\begin{equation*}
	\mathfrak g_- (u,k_j)
	\le
	\boldsymbol\gm \big(|u|+|k_j|\big)^{p-2}(u-k_j)_-^2.
\end{equation*}
When $p\ge 2$, we use $(u-k_j)_-\le |u|+|k_j|$ as well as $u\ge \boldsymbol \mu^-$ and $|\boldsymbol\mu^{-}|\le\eps M 2^{-j_*}$ to estimate
\begin{equation*}
	\mathfrak g_- (u,k_j)
	\le
	\boldsymbol\gm \big(|u|+|k_j|\big)^p \chi_{\{u\le k_j\}}
	\le
	\boldsymbol\gm\left(\frac{\eps M}{2^j}\right)^{p}.
\end{equation*}
When $1<p<2$, we again use $(u-k_j)_-\le |u|+|k_j|$ and $u\ge \boldsymbol \mu^-$ to obtain
\begin{align*}
	\mathfrak g_- (u,k_j)
	\le
	\boldsymbol\gm (u-k_j)_-^p
	\le
	\boldsymbol\gm\left(\frac{\eps M}{2^j}\right)^{p},
\end{align*}
for a constant $\boldsymbol\gamma$ depending only on $p$.
This implies in particular that
\begin{equation*}
	\mathfrak g_- (u,k_j)
	\le
	 \frac{\boldsymbol \gm}{\dl\varrho^p}
	 \left(\frac{\eps M}{2^j}\right)^{p} |\widehat Q|
\end{equation*}
holds true in any case. In the second integral appearing on the right-hand side of the energy estimate, we utilize the bound
$(u-k_j)_-\le \eps M 2^{-j}$.
Therefore, in all cases the above estimate yields
\begin{equation*}
	\iint_{\widehat Q}|D(u-k_j)_-|^p\,\dx\dt\le\frac{\boldsymbol\gm}{\dl\varrho^p}\left(\frac{\eps M}{2^j}\right)^p|\widehat Q|.
\end{equation*}
Next, we apply \cite[Chapter I, Lemma 2.2]{DB} slice wise 
to $u(\cdot,t)$ for 
$t\in\left(0,\dl\rho^p\right]$
 over the cube $K_\varrho$,
for levels $k_{j+1}<k_{j}$.
Taking into account the measure theoretical information
\begin{equation*}
	\Big|\Big\{u(\cdot, t)>\boldsymbol \mu^-+\eps M\Big\}\cap K_{\varrho}\Big|\ge\frac{\al}2 |K_\varrho|
	\qquad\mbox{for all $t\in(0,\dl\rho^p]$,}
\end{equation*}
this gives
\begin{align*}
	&(k_j-k_{j+1})\big|\big\{u(\cdot, t)<k_{j+1} \big\}
	\cap K_{4\varrho}\big|
	\\
	&\qquad\le
	\frac{\boldsymbol\gm \varrho^{N+1}}{\big|\big\{u(\cdot, t)>k_{j}\big\}\cap K_{4\varrho}\big|}	
	\int_{\{ k_{j+1}<u(\cdot,t)<k_{j}\} \cap K_{4\varrho}}|Du(\cdot,t)|\,\dx\\
	&\qquad\le
	\frac{\boldsymbol \gm\varrho}{\al}
	\bigg[\int_{\{k_{j+1}<u(\cdot,t)<k_{j}\}\cap K_{4\varrho}}|Du(\cdot,t)|^p\,\dx\bigg]^{\frac1p}
	\big|\big\{ k_{j+1}<u(\cdot,t)<k_{j}\big\}\cap K_{4\varrho}\big|^{1-\frac1p}
	\\
	&\qquad=
	\frac{\boldsymbol \gm\varrho}{\al}
	\bigg[\int_{K_{8\varrho}}|D(u-k_j)_-(\cdot,t)|^p\,\dx\bigg]^{\frac1p}
	\big[ |A_j(t)|-|A_{j+1}(t)|\big]^{1-\frac1p}.
\end{align*}
Here we used in the last line the short hand notation $ A_j(t):= \big\{u(\cdot,t)<k_{j}\big\}\cap K_{4\varrho}$.
We now integrate the last inequality with respect to $t$ over  $(0,\dl\rho^p]$ and apply H\"older's inequality in time.
With the abbreviation $A_j=\{u<k_j\}\cap \widehat Q$ this procedure leads to
\begin{align*}
	\frac{\eps M}{2^{j+1}}\big|A_{j+1}\big|
	&\le
	\frac{\boldsymbol \gm\varrho}{\al}\bigg[\iint_{\widehat Q}|D(u-k_j)_-|^p\,\dx\dt\bigg]^\frac1p
	\big[|A_j|-|A_{j+1}|\big]^{1-\frac{1}p}\\
	&\le
	\boldsymbol\gm \frac{\eps M}{2^j}|\widehat Q|^{\frac1p}\big[|A_j|-|A_{j+1}|\big]^{1-\frac{1}p}.
\end{align*}
Recall that $\delta$ depends on the data and $\alpha$. Therefore $\boldsymbol\gm$ depends on only the data and $\alpha$. 
Now take the power $\frac{p}{p-1}$ on both sides of the above inequality to obtain
\[
	\big|A_{j+1}\big|^{\frac{p}{p-1}}
	\le
	\boldsymbol \gm|\widehat Q|^{\frac1{p-1}}\big[|A_j|-|A_{j+1}|\big].
\]
Add these inequalities from $0$ to $j_*-1$ to obtain
\[
	j_* |A_{j_*}|^{\frac{p}{p-1}}\le\boldsymbol\gm|\widehat Q|^{\frac{p}{p-1}}.
\]
From this we conclude
\[
	|A_{j_*}|\le\frac{\boldsymbol\gm}{j_*^{\frac{p-1}p}}|\widehat Q|.
\]
This completes the proof.
\end{proof}
\subsection{A DeGiorgi-type Lemma}

Here we prove a DeGiorgi-type Lemma on cylinders of the form $Q_\rho(\theta)$. In the application $\theta$ will be a universal constant depending only on the data, in particular $\theta$ will be independent of the solution.

\begin{lemma}\label{Lm:3:3}
 Let $u$ be a locally bounded, local sub(super)-solution to \eqref{Eq:1:1} -- \eqref{Eq:1:2p} in $E_T$ and
$(x_o,t_o) + Q_\varrho(\theta) =K_\varrho (x_o)\times (t_o-\theta\varrho^p, t_o]\Subset E_T$. There exists a constant $\nu\in(0,1)$ depending only on 
 the data and $\theta$, such that if
\begin{equation*}
	\Big|\Big\{
	\pm\big(\boldsymbol \mu^{\pm}-u\big)\le M\Big\}\cap (x_o,t_o)+Q_{\varrho}(\theta)\Big|
	\le
	\nu|Q_{\varrho}(\theta)|,
\end{equation*}
then either
\begin{equation*}
	|\boldsymbol\mu^{\pm}|>8M,
\end{equation*}
or
\begin{equation*}
	\pm\big(\boldsymbol\mu^{\pm}-u\big)\ge\tfrac{1}2M
	\quad
	\mbox{a.e.~in $(x_o,t_o)+Q_{\frac{1}2\varrho}(\theta)$.}
\end{equation*}
\end{lemma}

\begin{proof}
We prove the case of super-solutions only, the case of sub-solutions being similar.
Assume $(x_o,t_o)=(0,0)$ and $|\boldsymbol \mu^-|\le 8M$. Otherwise there is nothing to prove. 
In order to employ the energy estimate in Proposition~\ref{Prop:2:1}, 
we notice first that due to Lemma~\ref{lem:g} we have
\begin{equation*}
	\mathfrak g_-(u,k)
	\le
	\boldsymbol\gamma \big(|u|+|k|\big)^{p-2}(u-k)_-^2
	\le
	\boldsymbol\gamma \big(|u|+|k|\big)^{p-1}(u-k)_-
\end{equation*}
and for $\tilde k<k$ there holds $(u-k)_-\ge (u-\tilde k)_-$. Therefore, the energy estimate yields 
\begin{align*}
	&\essup_{-\theta\varrho^p<t<0}
	\int_{K_\varrho}\z^p\big(|u|+|k|\big)^{p-2}(u-\tilde{k})_-^2\,\dx
	+
	\iint_{Q_\varrho(\theta)}\z^p|D(u-\tilde{k})_-|^p\,\dx\dt\\
	&\qquad\le
	\boldsymbol\gm\iint_{Q_\varrho(\theta)}(u-k)^{p}_-|D\z|^p\,\dx\dt
	+
	\boldsymbol\gm \iint_{Q_\varrho(\theta)}\big(|u|+|k|\big)^{p-1} (u-k)_-|\partial_t\z^p|\,\dx\dt,
\end{align*}
for any non-negative piecewise smooth cutoff function $\zeta$ vanishing on the parabolic boundary of $Q_\varrho(\theta)$. 
In order to use this energy estimate, we set
\begin{align}\label{choices:k_n}
	\left\{
	\begin{array}{c}
	\displaystyle k_n=\boldsymbol\mu^-+\frac{M}2+\frac{M}{2^{n+1}},\quad \tilde{k}_n=\frac{k_n+k_{n+1}}2,\\[5pt]
	\displaystyle \varrho_n=\frac{\varrho}2+\frac{\varrho}{2^{n+1}},
	\quad\tilde{\varrho}_n=\frac{\varrho_n+\varrho_{n+1}}2,\\[5pt]
	\displaystyle K_n=K_{\varrho_n},\quad \widetilde{K}_n=K_{\tilde{\varrho}_n},\\[5pt] 
	\displaystyle Q_n=Q_{\rho_n}(\theta),\quad
	\widetilde{Q}_n=Q_{\tilde\rho_n}(\theta).
	\end{array}
	\right.
\end{align}
Recall that $Q_{\rho_n}(\theta)=K_n\times(-\theta\varrho_n^p,0]$ and $Q_{\tilde\rho_n}(\theta)=\widetilde{K}_n\times(-\theta\tilde{\varrho}_n^p,0]$. 
Introduce the cutoff function $0\le\z\le 1$ vanishing on the parabolic boundary of $Q_{n}$ and
equal to identity in $\widetilde{Q}_{n}$, such that
\begin{equation*}
	|D\z|\le\boldsymbol\gm\frac{2^n}{\varrho}
	\quad\text{and}\quad 
	|\z_t|\le\boldsymbol\gm\frac{2^{pn}}{\theta\varrho^p}.
\end{equation*}
In this setting, the energy estimate may be written as
\begin{align*}
	&\essup_{-\theta\tilde{\varrho}_n^p<t<0}
	\int_{\widetilde{K}_n} \big(|u|+|k_n|\big)^{p-2}(u-\tilde{k}_n)_-^2\,\dx
	+
	\iint_{\widetilde{Q}_n}|D(u-\tilde{k}_n)_-|^p \,\dx\dt\\
	&\qquad\le
	\boldsymbol\gm\frac{2^{pn}}{\rho^p}
	\iint_{Q_n}(u-k_n)^{p}_- \,\dx\dt
	+
	\boldsymbol\gm\frac{2^{pn}}{\theta\rho^p}
	\iint_{Q_n}\big(|u|+|k_n|\big)^{p-1} (u-k_n)_- \,\dx\dt \\
	&\qquad\le
	\boldsymbol\gm \frac{2^{pn}}{\varrho^p}M^{p}|A_n|,
\end{align*}
where $\boldsymbol\gm$ depends on the data and $\theta$. 
Here, we used $\boldsymbol\mu^-\le u\le k_n\le \boldsymbol\mu^-+M$ on $A_n$, 
where
\begin{equation*}
	A_n=\big\{u<k_n\big\}\cap Q_n.
\end{equation*}
On the other hand, we recall $|\boldsymbol \mu^-|\le 8M$, so that $u\le\tilde k_n$ implies $|u|+|k_n|\le 18 M$ and $|u|+|k_n|\ge k_n-u\ge k_n-\tilde k_n=2^{-(n+3)}M$. Inserting this above, we find that 
\begin{equation}\label{Eq:sample}
\begin{aligned}
	\frac{M^{p-2}}{2^{p(n+3)}} \essup_{-\theta\tilde{\varrho}_n^p<t<0}
	\int_{\widetilde{K}_n} (u-\tilde{k}_n)_-^2\,\dx
	+
	\iint_{\widetilde{Q}_n}|D(u-\tilde{k}_n)_-|^p \,\dx\dt
	\le
	\boldsymbol\gm \frac{2^{pn}}{\varrho^p}M^{p}|A_n|.
\end{aligned}
\end{equation}
Now setting $0\le\phi\le1$ to be a cutoff function which vanishes on the parabolic boundary of $\widetilde{Q}_n$
and equals the identity in $Q_{n+1}$, an application of the H\"older inequality  and the Sobolev imbedding
\cite[Chapter I, Proposition~3.1]{DB} gives that
\begin{align*}
	\frac{M}{2^{n+3}}
	|A_{n+1}|
	&\le 
	\iint_{\widetilde{Q}_n}\big(u-\tilde{k}_n\big)_-\phi\,\dx\dt\\
	&\le
	\bigg[\iint_{\widetilde{Q}_n}\big[\big(u-\tilde{k}_n\big)_-\phi\big]^{p\frac{N+2}{N}}
	\,\dx\dt\bigg]^{\frac{N}{p(N+2)}}|A_n|^{1-\frac{N}{p(N+2)}}\\
	&\le\boldsymbol\gm
	\bigg[\iint_{\tilde{Q}_n}\big|D\big[(u-\tilde{k}_n)_-\phi\big]\big|^p\,
	\dx\dt\bigg]^{\frac{N}{p(N+2)}}\\
	&\quad\ 
	\times\bigg[\essup_{-\theta\tilde{\varrho}_n^p<t<0}
	\int_{\widetilde{K}_n}\big(u-\tilde{k}_n\big)^{2}_-\,\dx\bigg]^{\frac{1}{N+2}}
	 |A_n|^{1-\frac{N}{p(N+2)}}\\
	&\le 
	\boldsymbol\gm 
	\bigg(\frac{2^{pn}}{\varrho^p}M^p\bigg)^{\frac{N}{p(N+2)}}
	\bigg(\frac{2^{p(2n+3)}}{\varrho^p}M^2\bigg)^{\frac{1}{N+2}}
	|A_n|^{1+\frac{1}{N+2}} \\
	&=
	\boldsymbol\gm
	\frac{2^{\frac{(2p+N)n}{N+2}}}{\rho^\frac{N+p}{N+2}}
	M |A_n|^{1+\frac{1}{N+2}}.
\end{align*}
In the second last line we used the above energy estimate.
In terms of $\boldsymbol  Y_n=|A_n|/|Q_n|$, this can be rewritten as
\begin{equation*}
	\boldsymbol  Y_{n+1}
	\le
	\boldsymbol\gm \boldsymbol b^n \boldsymbol  Y_n^{1+\frac{1}{N+2}},
\end{equation*}
for a constant $\boldsymbol\gm$ depending only on the data and with $\boldsymbol b\equiv 2^\frac{2(N+p+1)}{N+2}$.
Hence, by \cite[Chapter I, Lemma~4.1]{DB}, 
there exists
a positive constant $\nu$ depending only on the data, such that
$\boldsymbol  Y_n\to0$ if we require that $\boldsymbol  Y_o\le \nu$, which is the same as
assuming
\begin{equation*}
	|A_o|=\big|\big\{u<k_o\big\}\cap Q_o\big|
	= 
	\Big|\Big\{u<\boldsymbol\mu^-+M\Big\}\cap Q_\varrho(\theta)\Big|
	\le
	\nu\big| Q_\varrho(\theta)\big|.
\end{equation*}
Since $\boldsymbol Y_n\to 0$ in the limit $n\to\infty$ we have
\begin{equation*}
	\Big|\Big\{u<\boldsymbol\mu^-+\tfrac12 M \Big\}\cap Q_{\frac12 \varrho}(\theta)\Big|
	= 
	0.
\end{equation*}
This concludes the proof of the lemma.
\end{proof}

\subsection{Proof of the Expansion of Positivity}

We now have all the prerequisites at hand to prove the main result of this section.

\begin{proof}[Proof of Proposition~\ref{Prop:1:1}]
We only show the case of super-solutions, the other case of sub-solutions being similar.
Assume $(x_o,t_o)=(0,0)$.
By $\delta,\epsilon\in(0,1)$ and $\boldsymbol \gm>0$ we denote the corresponding constants from Lemma~\ref{Lm:3:1} and Lemma~\ref{Lm:3:2} depending on the data and $\alpha$ and by $\nu\in(0,1)$ we denote the constant from Lemma~\ref{Lm:3:3} applied with $\theta=\delta$. Then, $\nu$ depends on the data and $\alpha$. 
Next, we choose an integer $j_*$ in such a way that 
$$
	\frac{\boldsymbol\gm}{j_*^{\frac{p-1}p}}
	\le
	\nu.
$$ 
Then, $j_*$ depends only on the data and $\alpha$. We let $\xi=8$ if $1<p\le 2$ and $\xi=\epsilon 2^{-j_*}$ if $p>2$. In the following we may assume that $|\boldsymbol \mu^{-}|\le \xi M$, since otherwise there is nothing to prove. Applying in turn Lemma~\ref{Lm:3:1} and Lemma~\ref{Lm:3:2} we infer that
\begin{equation*}
	\bigg|\bigg\{
	u\le \boldsymbol \mu^-+\frac{\eps M}{2^{j_*}}\bigg\}\cap \widehat Q\bigg|
	\le
	\nu|\widehat Q|,\quad
\end{equation*}
where $\widehat Q=K_{4\varrho}\times\left(0,\dl\rho^p\right]$. 
Applying Lemma~\ref{Lm:3:3} with $M$ replaced by $\frac{\eps M}{2^{j_*}}$ yields 
\begin{equation*}
	u\ge \boldsymbol\mu^- + \frac{\eps M}{2^{j_*+1}}
	\quad
	\mbox{a.e.~in $K_{2\rho}\times\big(\delta(\tfrac12\rho)^p,\delta\rho^p\big]$.}
\end{equation*}
This proves the assertion of Proposition~\ref{Prop:1:1} for $\eta=\frac{\eps}{2^{j_*+1}}$ depending only on the data and $\alpha$. 
Let us point out in fact we have chosen $\xi=2\eta$ when $p>2$.
\end{proof}

\section{Proof of Theorem \ref{Thm:1:1} When $1<p<2$}\label{S:5}

\subsection{The Proof Begins}
Fix $(x_o,t_o)\in E_T$ and define
\begin{equation*}
	Q_o\df{=}K_{\varrho}(x_o)\times(t_o-\varrho^{p}, t_o]\Subset E_T.
\end{equation*}
We may assume that $(x_o,t_o)$ coincides with the origin. 
Set
\begin{equation*}
	\boldsymbol \mu^+=\essup_{Q_o}u,
	\quad
	\boldsymbol\mu^-=\essinf_{Q_o}u,
	\quad
	\boldsymbol\om=\boldsymbol\mu^+-\boldsymbol\mu^-.
\end{equation*}
Our proof unfolds along two main
cases, namely 
\begin{equation}\label{Eq:Hp-main}
\left\{
\begin{array}{c}
\mbox{when $u$ is near zero:  $\boldsymbol\mu^-\le\boldsymbol\om$ and 
$\boldsymbol\mu^+\ge-\boldsymbol\om$};\\[5pt]
\mbox{when $u$ is away from zero: $\boldsymbol\mu^->\boldsymbol\om$ or $\boldsymbol\mu^+<-\boldsymbol\om$.}
\end{array}\right.
\end{equation}
Note that \eqref{Eq:Hp-main}$_1$ is equivalent to the condition that $-2\boldsymbol\omega
\le \boldsymbol\mu^-\le\boldsymbol\mu^+\le 2\boldsymbol\omega$ and therefore  $|\boldsymbol\mu^\pm|\le2\boldsymbol\om$.
When this case  holds, we deal with it in Section~\ref{S:5:2}
via a simple application of Proposition~\ref{Prop:1:1},
thanks to the possibility to choose $\xi=8$. No intrinsic scaling whatsoever is used in Section~\ref{S:5:2},
though it seems unavoidable when we deal with the second case of \eqref{Eq:Hp-main}$_2$
in Section~\ref{S:5:3}
\subsection{Reduction of Oscillation Near Zero}\label{S:5:2}
In this section assume that the first case in \eqref{Eq:Hp-main} holds.
Observe that one of the following must be true: either 
\begin{equation}\label{Eq:5:1}
	\Big|\Big\{u\big(\cdot,-\tfrac12\rho^p\big)-\boldsymbol\mu^->\tfrac14 \boldsymbol\om\Big\}
	\cap 
	K_{\varrho}\Big|
	\ge
	\tfrac12 |K_{\varrho}|,
\end{equation}
or 
\begin{equation*}
	\Big|\Big\{\boldsymbol\mu^+ - u\big(\cdot,-\tfrac12\rho^p\big)>\tfrac14 \boldsymbol\om\Big\}
	\cap 
	K_{\varrho}\Big|
	\ge
	\tfrac12 |K_{\varrho}|,
\end{equation*}
Since both cases can be treated similarly, we restrict ourselves to the case \eqref{Eq:5:1}. 
As mentioned above, $|\boldsymbol\mu^\pm|\le2\boldsymbol\om$ always holds.
An application of Proposition \ref{Prop:1:1} (note also Remark \ref{Rem:1.1}) gives  $\eta\in(0,1)$
depending only on the data, such that
\begin{equation*}
	u\ge\boldsymbol\mu^-+\eta\boldsymbol\om
	\quad\mbox{a.e.~in $
	Q_1\df{=}K_{\varrho_1}\times(-\varrho_1^p,0]$, with $\varrho_1=\tfrac12 \varrho$.}
\end{equation*} 
This yields a reduction of oscillation, i.e.~we have
\[
	\essosc_{Q_1}u\le(1-\eta)\boldsymbol\om\df{=}\boldsymbol\om_1.
\]
Now we may proceed by induction. Suppose up to $i=1,2,\cdots j-1$, we have built
\begin{equation*}
\left\{
	\begin{array}{c}
	\dsty\varrho_i=\tfrac12\varrho_{i-1},
	\quad 
	\boldsymbol\om_i=(1-\eta)\boldsymbol\om_{i-1},
	\quad 
	Q_i=K_{\varrho_i}\times(-\varrho_i^p,0],\\[5pt]
	\dsty\boldsymbol\mu_i^+=\essup_{Q_i}u,
	\quad
	\boldsymbol\mu_i^-=\essinf_{Q_i}u,
	\quad
	\essosc_{Q_i}u\le\boldsymbol\om_i.
	\end{array}
\right.
\end{equation*}
For all the indices $i=1,2,\cdots j-1$, we alway assume the first case in \eqref{Eq:Hp-main} holds, i.e.,
$$
	\boldsymbol\mu_i^-\le\boldsymbol\om_i\quad
	\text{ and }\quad\boldsymbol\mu_i^+\ge-\boldsymbol\om_i.
$$
In this way the argument at the beginning can be repeated and we have for all $i=1,2,\cdots j$,
\[
	\essosc_{Q_i}u\le(1-\eta)\boldsymbol\om_{i-1}=\boldsymbol\om_i.
\]
Consequently, iterating the above recursive inequality we obtain for all $i=1,2,\cdots j$,
\begin{equation}\label{Eq:case1}
	\essosc_{Q_i}u\le(1-\eta)^i\boldsymbol\om
	=
	\boldsymbol\om\left(\frac{\rho_i}{\rho}\right)^{\be_o}
	\quad\text{ where }
	\beta_o=\frac{-\ln(1-\eta)}{\ln2}.
\end{equation}
\subsection{Reduction of Oscillation Away From Zero}\label{S:5:3}
In this section, let us suppose $j$ is the first index satisfying the second case in \eqref{Eq:Hp-main}, i.e.,
\[
	\mbox{either \quad$\boldsymbol\mu_j^-> \boldsymbol \om_j\;$\quad
	or\quad
	$\;\boldsymbol\mu_j^+<-\boldsymbol \om_j$.}
\]
Let us treat for instance $\boldsymbol\mu_j^->\boldsymbol\om_j$,
for the other case is analogous.
We observe that since $j$ is the first index for this to happen,
one should have $\boldsymbol\mu_{j-1}^-\le \boldsymbol \om_{j-1}$.
Moreover, one estimates
\[
	\boldsymbol\mu_{j}^-
	\le \boldsymbol\mu_{j-1}^- +\boldsymbol \om_{j-1}-\boldsymbol \om_{j}
	\le
	2\boldsymbol \om_{j-1}-\boldsymbol \om_{j}=\frac{1+\eta}{1-\eta}\boldsymbol \om_{j}.
\]
As a result, we have
\begin{equation}\label{Eq:5:4}
	\boldsymbol \om_{j}
	<
	\boldsymbol \mu_{j}^-\le\frac{1+\eta}{1-\eta}\boldsymbol \om_{j}.
\end{equation}
The condition \eqref{Eq:5:4} indicates that starting from $j$ the equation \eqref{Eq:1:1}
resembles the parabolic $p$-Laplacian type equation in $Q_j$. Therefore, the reduction of oscillation
hinges upon the possibility to treat the parabolic $p$-Laplacian type equation.
To render this technically, we drop the suffix $j$ from our notation temporarily for simplicity,
and introduce $v:=u/\boldsymbol\mu^-$ in $Q=K_{\rho}\times(-\rho^p,0]$.
It is straightforward to verify that $v$ satisfies
\begin{equation*}
	\pl_tv^{p-1}-\dvg\bar{\bl{A}}(x,t,v, Dv)=0\quad\text{ weakly in }Q,
\end{equation*}
where, for $ (x,t)\in Q$, $v\in\rr$ and $\z\in\rn$, we have defined 
$$
	\bar{\bl{A}}(x,t,v, \z)
	=
	(\boldsymbol\mu^-)^{1-p}\bl{A}(x,t,\boldsymbol\mu^-v, \boldsymbol\mu^-\z),
$$
which is subject to the structure conditions
\begin{equation*}
	\left\{
	\begin{array}{c}
		\bar{\bl{A}}(x,t,v,\z)\cdot \z\ge C_o|\z|^p \\[5pt]
		|\bar{\bl{A}}(x,t,v,\z)|\le C_1|\z|^{p-1}
	\end{array}
	\right .
	\qquad \mbox{ for a.e.~$(x,t)\in Q$, $\forall v\in\rr$, $\forall\z\in\rn$.}
\end{equation*}
Moreover,
\begin{equation}\label{Eq:5:5}
1\le v\le2\quad\text{ a.e. in }Q.
\end{equation}
In order to use the known regularity theory for the parabolic $p$-Laplacian 
(see \cite{DB, DBGV-mono} for an account of the theory),
it turns out to be more convenient to consider the equation satisfied by $w:=v^{p-1}$,
i.e.
\begin{equation}\label{Eq:5:6}
\partial_tw-\dvg\widetilde{\bl{A}}(x,t,w, Dw)=0\quad\text{ weakly in }Q,
\end{equation}
where for $ (x,t)\in Q$, $y\in\rr$ and $\zeta\in\rn$, we have defined 
\begin{equation*}\label{def:tilde-A}
	\widetilde{\bl{A}}(x,t,y, \zeta)
	=
	\bar{\bl{A}}\Big( x,t,\widetilde y^{\frac{1}{p-1}},\tfrac{1}{p-1} \widetilde y^{\frac{2-p}{p-1}}\zeta\Big).
\end{equation*}
In the last line we used the abbreviation
\[
	\widetilde y\df{=} \min\big\{ \max\big\{ y,\tfrac12\big\}, 2^p\big\}.
\]
It is easy to see that $w$ belongs to the same kind of functional space \eqref{Eq:1:3p}
as $u$ and $v$ due to \eqref{Eq:5:5} which yields $1\le w\le 2^{p-1}$ in $Q$.
Employing \eqref{Eq:5:5} again
one can verify that there exist
absolute positive constants $\widetilde{C}_o=\boldsymbol\gamma_o (p)C_o$ and $\widetilde{C}_1=\boldsymbol\gamma_1 (p)C_1$, such that
\begin{equation}\label{Eq:5:7}
		\widetilde{\bl{A}}(x,t,y,\zeta)\cdot \zeta\ge \widetilde{C}_o|\zeta|^p
		\quad\mbox{and}\quad
		|\widetilde{\bl{A}}(x,t,y,\zeta)|\le \widetilde{C}_1|\zeta|^{p-1},
\end{equation}
 for a.e.~$(x,t)\in Q$, any  $y\in\rr$, and any $\zeta\in\rn$. In other words, $w$ is a local weak solution to the parabolic $p$-Laplacian type equation.
First proved in \cite{DB86} for $p>2$ and then in \cite{ChenDB92} for $1<p<2$, 
the power-like oscillation decay for solutions to this kind of
degenerate or singular parabolic equation is well known by now.  
The proofs in \cite{ChenDB92, DB86} exploit the idea of intrinsic scaling.
We state the conclusion in the following proposition in a form that favors our application,
and refer to the monograph \cite{DB} for a comprehensive treatment of this issue.
\begin{proposition}\label{Prop:5:1}
Let $p>1$.
Suppose $w$ is a bounded, local, weak solution to 
\eqref{Eq:5:6} -- \eqref{Eq:5:7} in $Q:=Q_\rho$
and define
$$
	\boldsymbol{\widetilde{\om}}=\essosc_Q w.
$$
If for some constant $\sig$ in $(0,1)$, there holds
\begin{equation}\label{Eq:5:8}
	\essosc_{Q_{\sig\varrho}(\theta)}w\le\boldsymbol{\widetilde{\om}}
	\quad
	\text{ where }
	\theta=\boldsymbol{\widetilde{\om}}^{2-p},
\end{equation}
then, there exist constants $\beta_1$ in $(0,1)$ and $\boldsymbol\gm>1$ depending only on the data
$N,p,\widetilde C_o, \widetilde C_1$ and $\sig$, such that for all $0<r<\rho$, we have
\[
	\essosc_{Q_r(\theta)}w\le\boldsymbol\gm\boldsymbol{\widetilde{\om}}\left(\frac{r}{\rho}\right)^{\be_1}.
\]
\end{proposition}
\begin{remark}\upshape
This proposition has been stated for all $p>1$.
However the proofs in \cite{DB86} for $p>2$ and in \cite{ChenDB92} for $1<p<2$
are remarkably different. We mention a recent attempt in \cite{Liao} to find a unified approach.
\end{remark}
To use this proposition properly when $1<p<2$, we first check the condition \eqref{Eq:5:8}
is satisfied. Indeed, recalling $v=u/\boldsymbol\mu^-$, $w=v^{p-1}$ and $\boldsymbol\om=\essosc_Q u$, 
we first use \eqref{Eq:5:5} and the mean value theorem to obtain
\[
	(p-1)2^{p-2}\essosc_Qv\le\boldsymbol{\widetilde{\om}}=\essosc_Q w\le (p-1)\essosc_Qv.
\]
Since $\essosc_Qv=\boldsymbol\om/\boldsymbol\mu^-$, this amounts to
\[
	(p-1)2^{p-2}\frac{\boldsymbol\om}{\boldsymbol\mu^-}
	\le
	\boldsymbol{\widetilde{\om}}	
	\le
	(p-1)\frac{\boldsymbol{\om}}{\boldsymbol\mu^-}.
\]
Then by \eqref{Eq:5:4}, we have
\[
	c\df{=}\frac{1-\eta}{1+\eta}(p-1)2^{p-2}\le\boldsymbol{\widetilde{\om}}\le (p-1)\le1.
\]
Now since $\boldsymbol{\widetilde{\om}}\le1$, we have $Q_{\rho}(\theta)\subset Q_\rho$.
Thus  the condition \eqref{Eq:5:8} in Proposition~\ref{Prop:5:1} is fulfilled for $\sig=1$.
As a result, the conclusion of Proposition~\ref{Prop:5:1} is obtained.
Moreover, the above lower bound of $\boldsymbol{\widetilde{\om}}$ actually allows us to obtain the set inclusion
\[
	Q_{r}(\theta_o)\subset Q_r(\theta)\quad\text{ where }\theta_o=c^{2-p}.
\]
Using this set inclusion and rephrasing the oscillation decay of Proposition~\ref{Prop:5:1} in terms of $u$, we have
 for all $0<r<\rho$,
\[
	\essosc_{Q_r(\theta_o)}u
	\le
	\boldsymbol\gm\boldsymbol\om\left(\frac{r}{\rho}\right)^{\be_1}.
\]
Now we revert to using the suffix $j$. The above oscillation estimate reads:
for all $0<r<\rho_j$, we have
\begin{equation}\label{Eq:case2}
	\essosc_{Q_r(\theta_o)}u
	\le
	\boldsymbol\gm\boldsymbol\om_j\left(\frac{r}{\rho_j}\right)^{\be_1}.
\end{equation}
Combining \eqref{Eq:case1} and \eqref{Eq:case2},
we arrive at the desired conclusion, i.e., for all $0<r<\rho$, there holds
\[
	\essosc_{Q_r(\theta_o)}u
	\le
	\boldsymbol\gm\boldsymbol\om\left(\frac{r}{\rho}\right)^{\be}\quad\text{ where }\be=\min\{\beta_o,\beta_1\}.
\]
A proper rescaling gives the oscillation decay in Remark~\ref{Rmk:1:1} and finishes the proof of Theorem \ref{Thm:1:1}
in the case $1<p<2$.
\section{Proof of Theorem \ref{Thm:1:1} When $p>2$}\label{S:6}
\subsection{The Proof Begins}
Fix $(x_o,t_o)\in E_T$. Let $A\ge1$ to be determined later in terms of the data
and $\varrho>0$ be so small that 
\begin{equation*}
	Q_o\df{=}K_{\varrho}(x_o)\times\big(t_o-A\varrho^{p}, t_o\big]\Subset E_T.
\end{equation*}
We may assume that $(x_o,t_o)$ coincides with the origin. 
Set
\begin{equation*}
	\boldsymbol \mu^+=\essup_{Q_o}u,
	\quad
	\boldsymbol\mu^-=\essinf_{Q_o}u,
	\quad
	\boldsymbol\om=\boldsymbol\mu^+-\boldsymbol\mu^-.
\end{equation*}
Like when $1<p<2$, our proof unfolds along two main
cases, namely 
\begin{equation}\label{Eq:Hp-main1}
\left\{
\begin{array}{c}
	\mbox{when $u$ is near zero: $\boldsymbol\mu^-\le\xi\boldsymbol\om\;$
		 and $\;\boldsymbol\mu^+\ge-\xi\boldsymbol\om$};\\[5pt]
	\mbox{when $u$ is away from zero: $\boldsymbol\mu^->\xi\boldsymbol\om\;$
	or $\;\boldsymbol\mu^+<-\xi\boldsymbol\om$.}
\end{array}
\right.
\end{equation}
Strictly speaking, the above $\xi$ should be $\xi/8$
for $\xi$ chosen as in Proposition~\ref{Prop:1:1} (note also Remark~\ref{Rem:1.1}) 
depending on the data, $A$ and $\al=\frac12\nu$, whereas $\nu$ is the absolute constant
determined in Lemma~\ref{Lm:3:3} with $\theta=1$ there. 
It will be clear shortly from the proof where the various dependences come from.
Meanwhile, we will keep using $\xi$ to denote $\xi/8$ for ease of notation
bearing in mind the actual meaning of $\xi$,
and this substitution will not spoil our reasoning in the following.

When $p>2$, the number $\xi$ from Proposition~\ref{Prop:1:1} is in general a very small number.
This brings additional technical complication to reducing the oscillation in the first case of \eqref{Eq:Hp-main1}.
As we will see in Section~\ref{S:6:2} and Section~\ref{S:6:3},
the method of intrinsic scaling is employed.
Whereas in Section~\ref{S:4:4} where we deal the second case of \eqref{Eq:Hp-main1},
the treatment more or less parallels Section~\ref{S:5:3} for $1<p<2$.
\subsection{Reduction of Oscillation Near Zero--Part I}\label{S:6:2}
In this section we assume the first case of \eqref{Eq:Hp-main1} holds.
We work with $u$ as a super-solution near its
infimum. 

Suppose that for some $\bar{t}\in\big(-(A-1)\rho^p,0\big]$,
\begin{equation}\label{Eq:4:2}
	\Big|\Big\{u\le\boldsymbol\mu^-+\tfrac14 \boldsymbol\om\Big\}
	\cap 
	(0,\bar{t})+Q_{\varrho}\Big|\le \nu|Q_{\varrho}|,
\end{equation}
where $\nu$ is the absolute constant appearing in Lemma~\ref{Lm:3:3}.
Taking $M=\frac14\boldsymbol\om$, then
according to Lemma~\ref{Lm:3:3},
we have 
\[
	u\ge\boldsymbol \mu^-+\tfrac{1}8\boldsymbol\om
	\quad
	\mbox{a.e.~in $(0,\bar{t})+Q_{\frac12 \varrho}$,}
\]
since the other alternative, i.e., $|\boldsymbol\mu^-|\ge 2\boldsymbol \om$,
does not hold due to \eqref{Eq:Hp-main1}$_1$.
An application of Proposition \ref{Prop:1:1} (note also Remark \ref{Rem:1.1}) applied with $2^pA$ instead of $A$ gives  $\xi,\,\eta\in(0,1)$
depending only on the data and $A$, such that
either $|\boldsymbol\mu^-|>\xi\boldsymbol \om$ or
\begin{equation}\label{def-Q1}
	u\ge\boldsymbol\mu^-+\eta\boldsymbol\om
	\quad\mbox{a.e.~in $
	\widetilde{Q}_1\df{=}K_{\frac12\varrho}\times(-(\tfrac12 \varrho)^p,0]$.}
\end{equation}
If the above line holds, we immediately obtain a reduction of oscillation, i.e.~we have
\[
	\essosc_{\widetilde{Q}_1}u\le(1-\eta)\boldsymbol\om.
\]
The case $\boldsymbol\mu^->\xi\boldsymbol\om$ does not hold due to \eqref{Eq:Hp-main1}$_1$.
 Therefore it remains to deal with the case $\boldsymbol\mu^-<-\xi\boldsymbol\om$. 
Due to the restriction on $\boldsymbol\mu^+$ in \eqref{Eq:Hp-main1}$_1$, we must also have $\boldsymbol\mu^->-2\boldsymbol\om$.
Thus,  we proceed further with the assumptions
\begin{equation}\label{Eq:H1}
\left\{
	\begin{array}{c}
	-2\boldsymbol\om<\boldsymbol\mu^-<-\xi\boldsymbol\om,\\[5pt]
	\mbox{$\dsty u\big(\cdot, \bar{t}-(\tfrac12\rho)^p\big)\ge\boldsymbol\mu^-+\tfrac{1}8\boldsymbol\om\;\;$a.e.~in $K_{\frac12\varrho}$.}
\end{array}
\right.
\end{equation}
In the next lemma we establish that the pointwise information in \eqref{Eq:H1}$_2$
propagates to the top of the cylinder $Q_o$.
\begin{lemma}\label{Lm:6:1}
Suppose the hypothesis \eqref{Eq:H1} holds.
Then there exists a constant $\eta_1\in(0,1)$ depending on $\xi$, $A$ and the data, such that
\[
	u\ge\boldsymbol\mu^-+\eta_1\boldsymbol\om\quad\mbox{a.e. ~in 
	$K_{\frac14\varrho}\times\big(\bar{t}-(\tfrac12\rho)^p,0\big]$.}
\]
As a result, we have a reduction of oscillation 
\[
	\essosc_{\widehat{Q}_1}u\le(1-\eta_1)\boldsymbol\om\quad\mbox{where $\widehat{Q}_1
	=K_{\frac14\rho}\times\big(-(\tfrac12\rho)^p,0\big]$.}
\]
\end{lemma}
\begin{proof} For ease of notation we set $\bar{t}-(\tfrac12\rho)^p=0$.
Define $k_n$, $\tilde{k}_n$, $\varrho_n$, $\tilde{\varrho}_n$, $K_n$ and $\widetilde{K}_n$, according to  \eqref{choices:k_n}
(cf.~the proof of Lemma~\ref{Lm:3:3}) with $M$ and $\varrho$ 
replaced by $2\eta_1\boldsymbol\om$ and $\frac12\varrho$ respectively, for some $0<\eta_1<\frac18\xi$ and $\theta>0$ to be determined later. The only difference is that now the cylinders 
$Q_n$ and $\widetilde{Q}_n$ are of forward type whose vertices are attached to the origin,
i.e.,~$Q_n=K_n\times(0,\theta\tilde{\rho}_n^p]$ and $\widetilde{Q}_n=\widetilde{K}_n\times(0,\theta\tilde{\rho}_n^p]$.
Since we know the ``initial datum'' at $t=0$ as in \eqref{Eq:H1}$_2$, we may choose
a cutoff function $\z$ in $K_n$ independent of $t$, such that
it equals $1$ on $\widetilde{K}_n$ and vanishes on $\pl K_n$,
satisfying $|D\z|\le\boldsymbol\gm2^{n}\rho^{-1}$. Note that the boundary term at $t=0$ on the right-hand side of the energy 
inequality vanishes on $K_n$, since
\begin{equation*}
	u(\cdot, 0) 
	\ge 
	\boldsymbol \mu^-+\tfrac18\boldsymbol\om
	\ge 
	\boldsymbol \mu^-+2\eta_1\boldsymbol\om
	\ge 
	\boldsymbol \mu^-+\eta_1\boldsymbol\om+\eta_1\frac{\boldsymbol\om}{2^n}
	=
	k_n
\end{equation*}
a.e.~on $K_{\frac12\varrho}$.
This requires $2\eta_1<\tfrac18$. 
In this way, the terms
on the right-hand side of the energy  estimate  in Proposition \ref{Prop:2:1} involving the initial time and $\z_t$ vanish.
Thus, using the condition $-2\boldsymbol\om<\boldsymbol\mu^-<-\xi\boldsymbol\om$, which leads to a lower bound for the sup-term in the energy estimate, we obtain
\begin{align}\label{Eq:sample*}
	\boldsymbol \om^{p-2}\essup_{0<t<\theta\varrho^p} 
	&
	\int_{\widetilde{K}_n}\big(u-\tilde{k}_n\big)^2_-\,\dx
	+
	\iint_{\widetilde{Q}_n}\big|D(u-\tilde{k}_n)_-\big|^p\,\dx\dt
	\le
	\boldsymbol\gm\frac{2^{pn}}{\varrho^p}(\eta_1\boldsymbol\om)^{p}|A_n|,
\end{align}
where 
\[
	A_n=\big\{u<k_n\big\}\cap Q_n.
\]
Now setting $\z$ to be a cutoff function which vanishes on the parabolic boundary of $\widetilde{Q}_n$
and equals identity in $Q_{n+1}$, an application of the Sobolev imbedding
\cite[Chapter I, Proposition~3.1]{DB} with $q=p\frac{N+2}{N}$ and $m=2$ gives that
\begin{align*}
	\bigg(\frac{\eta_1\boldsymbol\om}{2^{n+2}}\bigg)^p|A_{n+1}|
	&\le 
	\iint_{\widetilde{Q}_n}\big[(u-\tilde{k}_n)_-^p\z\big]^p\,\dx\dt\\
	&\le
	\bigg[\iint_{\widetilde{Q}_n}[(u-\tilde{k}_n)_-\z]^{p\frac{N+2}{N}}\,\dx\dt\bigg]^{\frac{N}{N+2}}
	|A_n|^{\frac{2}{N+2}}\\
	&\le
	\boldsymbol\gm\bigg[\iint_{\widetilde{Q}_n}\big|D\big[(u-\tilde{k}_n)_-\z\big]\big|^p
	\,\dx\dt\bigg]^{\frac{N}{N+2}}\\
	&\quad
	\times\bigg[\essup_{0<t<\theta\varrho^p}\int_{\widetilde{K}_n}\big(u-\tilde{k}_n\big)^2_-\,\dx\bigg]^{\frac{p}{N+2}}
 	|A_n|^{\frac{2}{N+2}}\\
	&\le 
	\boldsymbol\gm\boldsymbol\om^{\frac{p(2-p)}{N+2}}
	\bigg(\frac{2^{pn}}{\varrho^p}(\eta_1\boldsymbol\om)^{p}\bigg)^{\frac{N+p}{N+2}}
	|A_n|^{1+\frac{p}{N+2}}.
\end{align*}
Setting $\boldsymbol Y_n=|A_n|/|Q_n|$, we arrive at
\begin{equation*}
	\boldsymbol Y_{n+1}
	\le
	\boldsymbol\gm \Big( 2^{p(1+\frac{N+p}{N+2})}\Big)^n\Big(\eta_1^{p-2}\theta\Big)^\frac{p}{N+2}
	\boldsymbol Y_n^{1+\frac{p}{N+2}}
	\equiv \boldsymbol \gm \boldsymbol b^n\Big(\eta_1^{p-2}\theta\Big)^\frac{p}{N+2}
	\boldsymbol Y_n^{1+\frac{p}{N+2}}.
\end{equation*}
The meaning of  $\boldsymbol b$ is clear in this context. Note that
$\boldsymbol b$ and $\boldsymbol \gm$ depend only on the data.
Hence by \cite[Chapter I, Lemma~4.1]{DB}, there exists a constant $\nu_o\in(0,1)$
depending only on the data, such that $\boldsymbol Y_n\to0$ in the limit $n\to\infty$ if we require that
\begin{equation*}
	\boldsymbol Y_o\le \nu_o\frac{\eta_1^{2-p}}{\theta}.
\end{equation*}
To finish the proof, we fix $\theta= 2^p A$ 
and choose $\eta_1$ so small that $\nu_o\frac{\eta_1^{2-p}}{\theta}\ge 1$. 
The latter is implied if $\eta_1^{p-2}<\frac{\nu_o}{2^p A}$.
Together with the former bound for $\eta_1$ determined in the course of the proof, 
we have to require that
\begin{equation*}
	\eta_1<\min\bigg\{ \tfrac1{16}, \tfrac18\xi, \Big(\frac{\nu_o}{2^p A}\Big)^\frac1{p-2}\bigg\}.
\end{equation*}
This proves the asserted claim.
\end{proof}
\subsection{Reduction of Oscillation Near Zero--Part II}\label{S:6:3}
In this section we still assume the first case of \eqref{Eq:Hp-main1} holds.
However, now we work with $u$ as a sub-solution near its
supreme.

Suppose contrary to \eqref{Eq:4:2} that 
\begin{equation}\label{Eq:5:3:1}
	\Big|\Big\{u\le\boldsymbol\mu^-+\tfrac14 \boldsymbol\om\Big\}
	\cap 
	(0,\bar{t})+Q_{\varrho}\Big|> \nu|Q_{\varrho}|,\qquad\forall\, \bar{t}\in\big(-(A-1)\rho^p,0\big].
\end{equation}
Then for any such $\bar{t}$, there exists some $s\in\big[
	\bar{t}-\varrho^p,\bar{t}-\tfrac{1}2\nu\varrho^p\big]
$
with
\begin{equation*}
	\Big|\Big\{u(\cdot, s)\le\boldsymbol \mu^-+\tfrac{1}4\boldsymbol\om\Big\}\cap K_{\varrho}\Big|>\tfrac{1}2\nu |K_{\varrho}|.
\end{equation*}
Indeed, if the above inequality does not hold for any $s$ in the given 
interval, then 
\begin{align*}
	\Big|\Big\{u\le\boldsymbol \mu^-+\tfrac{1}4\boldsymbol\om\Big\}\cap (0,\bar t)+Q_{\varrho}\Big|
	&=
	\int_{\bar t-\varrho^p}^{\bar t-\frac12\nu\varrho^p}\Big|\Big\{u(\cdot, s)
	\le\boldsymbol \mu^-+\tfrac{1}4\boldsymbol\om\Big\}\cap K_{\varrho}\Big|\,\ds\\
	&\phantom{=\,}
	+\int^{\bar t}_{\bar t-\frac12\nu\varrho^p}\Big|\Big\{u(\cdot, s)\le\boldsymbol \mu^-+\tfrac{1}4\boldsymbol\om\Big\}\cap K_{\varrho}\Big|\,\ds\\
	&<\tfrac12\nu |K_{\varrho}|\big(\varrho^p-\tfrac12\nu\varrho^p\big) +\tfrac12\nu\varrho^p
	|K_\varrho |
	<\nu |Q_\varrho|,
\end{align*}
contradicting \eqref{Eq:5:3:1}.
Since $\boldsymbol\mu^+-\frac14\boldsymbol\om>\boldsymbol\mu^-+\frac14\boldsymbol\om$ always holds, this implies
\begin{equation*}
	\Big|\Big\{u(\cdot,s)\le\boldsymbol \mu^+-\tfrac{1}4\boldsymbol\om\Big\}\cap K_\varrho\Big|
	\ge\tfrac12\nu|K_{\varrho}|.
\end{equation*}
By Proposition \ref{Prop:1:1}, there exist $\xi,\eta_2\in(0,1)$,
such that either $|\boldsymbol\mu^+|>\xi\boldsymbol\om$ or
\begin{equation*}
	u\le \boldsymbol\mu^+-\eta_2\boldsymbol\om\quad\mbox{a.e.~in $\widetilde{Q}_1$,}
\end{equation*}
where $\widetilde{Q}_1$ is defined in \eqref{def-Q1}. 
This implies a reduction of oscillation
\begin{equation*}
	\essosc_{\widetilde{Q}_1}u\le(1-\eta_2)\boldsymbol\om.
\end{equation*}
The case $\boldsymbol\mu^+<-\xi\boldsymbol\om$ does not hold due to \eqref{Eq:Hp-main1}$_1$.
Next we handle the case when $\boldsymbol\mu^+>\xi\boldsymbol\om$.
Due to the restriction on $\boldsymbol\mu^-$ in  \eqref{Eq:Hp-main1}$_1$, we must have $\boldsymbol\mu^+\le 2\boldsymbol\om$. 
Thus our assumptions  for the following Sections \ref{S:4:3:1} -- \ref{S:4:3:3} are
\begin{equation}\label{Eq:4:3a}
	\xi\boldsymbol\om\le\boldsymbol\mu^+\le2\boldsymbol\om,
\end{equation}
and
\begin{equation}\label{Eq:4:3b}
	\left\{
	\begin{array}{c}
	\mbox{for any $\bar{t}\in\big(-(A-1)\varrho^p,0\big]$ there exists $s\in\big[\bar{t}-\varrho^p, \bar{t}
	-\tfrac12 \nu\varrho^p
	\big]$}\\[6pt]
	\mbox{such that   	$\dsty\Big|\Big\{u(\cdot,s)\le\boldsymbol \mu^+-\tfrac{1}4\boldsymbol\om\Big\}\cap K_\varrho\Big|
	\ge\tfrac12\nu|K_{\varrho}|. $  }
	\end{array}
	\right.
\end{equation}
\subsubsection{Propagation of Measure Theoretical Information}\label{S:4:3:1}
\begin{lemma}\label{Lm:4:2}
Suppose \eqref{Eq:4:3a} and  \eqref{Eq:4:3b} are in force.
 There exists $\eps\in(0,1)$,
depending only on $\nu$, $\xi$ and the data, such that
\begin{equation*}
	\Big|\Big\{ u(\cdot, t)\le\boldsymbol \mu^+-\eps\boldsymbol \om\Big\}\cap K_{\varrho}\Big|
	\ge
	\tfrac14\nu |K_\varrho|
	\quad\mbox{for all $t\in(s,\bar{t}\,]$.}
\end{equation*}
\end{lemma}
\begin{proof} Assume $s=0$ for ease of notation.
Use the energy estimate in Proposition~\ref{Prop:2:1}
in the cylinder $Q:=K_{\varrho}\times(0,\dl\varep^{2-p}\varrho^p]$, with
$k=\boldsymbol\mu^+-\varep\boldsymbol \om$ for some $\dl>0$ and $0<\varep\le\frac12\xi$ to be determined later. Note that $k\ge\frac12\xi\boldsymbol\om$. 
Choose a standard non-negative time independent cutoff function 
$\zeta (x,t)\equiv\z(x)$  that equals $1$ on $K_{(1-\sig)\varrho}$ with $\sigma\in(0,1)$ 
and vanishes on $\pl K_{\varrho}$ satisfying
$|D\z|\le(\sig\varrho)^{-1}$;
in such a case, for all $0<t<\dl \varep^{2-p}\rho^p$ we have
\begin{align*}
	\int_{K_\varrho\times\{t\}}&\int_{k}^u s^{p-2}(s-k)_+\,\ds\z^p\,\dx\\
	&\le
	\int_{K_\varrho\times\{0\}}\int_{k}^u s^{p-2}(s-k)_+\,\ds\z^p\,\dx
	+
	\boldsymbol\gm\iint_{Q}(u-k)^{p}_+|D\z|^p\,\dx\dt.
\end{align*}
The first term on the right is bounded from above by taking \eqref{Eq:4:3b}
into consideration. Indeed we have
\begin{align*}
	\int_{K_\varrho\times\{0\}}\int_{k}^u s^{p-2}(s-k)_+\,\ds\z^p\,\dx
	\le
	\big(1-\tfrac12\nu\big)|K_{\rho}|
	\int^{\boldsymbol\mu^+}_k s^{p-2}(s-k)_+\,\ds.
\end{align*}
The second term on the right is estimated by
\begin{equation*}
	\iint_{Q}(u-k)^{p}_+|D\z|^p\,\dx\dt
	\le
	\frac{\boldsymbol \gm\dl}{\sig^p}\varep^{2-p}(\varep\boldsymbol\om)^p|K_\varrho|
	\le
	\frac{\boldsymbol \gm\dl}{\sig^p}\varep^{2}\boldsymbol\om^p|K_\varrho|.
\end{equation*}
For the left-hand side, we estimate from below by
\begin{align*}
		\int_{K_\varrho\times\{t\}}&\int_{k}^u s^{p-2}(s-k)_+\,\ds\z^p\,\dx
	\ge\big|\big\{ u(\cdot, t)>k_{\tilde\eps}\big\}\cap K_{(1-\sig)\varrho}\big|
	\int^{k_{\tilde\eps}}_k s^{p-2}(s-k)_+\,\ds,
\end{align*}
where $k_{\tilde\eps}=\boldsymbol \mu^+-\tilde\eps\varep\om$ for some $\tilde\eps\in(0,\frac12)$.
Noticing $\xi\boldsymbol\om\le\boldsymbol\mu^+\le2\boldsymbol\om$, we may estimate
\begin{equation*}
	\int^{k_{\tilde\eps}}_k s^{p-2}(s-k)_+\,\ds
	\ge
	\boldsymbol\gm\boldsymbol \om^{p-2}(\varep\boldsymbol\om)^2
	=
	\boldsymbol\gm\varep^2\boldsymbol \om^{p}.
\end{equation*}
A similar consideration as in Lemma \ref{Lm:3:1} then gives
\begin{align*}
	\big|\big\{ u(\cdot, t)>k_{\tilde\eps}\big\}\cap K_{(1-\sig)\varrho}\big|
	\le 
	\frac{\dsty\int^{\boldsymbol\mu^+}_k s^{p-2}(s-k)_+\,\ds}{\dsty\int^{k_{\tilde\eps}}_k s^{p-2}(s-k)_+\,\ds}
	\big(1-\tfrac12\nu\big)|K_{\rho}| +\frac{\boldsymbol\gm\dl}{\sig^p}|K_\rho|.
\end{align*}
The fractional number of integral on the right can be rewritten as
\begin{align*}
	1+I_\eps
	\quad\text{ where }\quad
	I_\eps
	=
	\frac{\dsty\int^{\boldsymbol\mu^+}_{k_{\tilde\eps}} s^{p-2}(s-k)_+\,\ds}
	{\dsty\int^{k_{\tilde\eps}}_{k} s^{p-2}(s-k)_+\,\ds}.
\end{align*}
We estimate by using $\xi\boldsymbol\om\le\boldsymbol\mu^+\le2\boldsymbol\om$ and $k\ge\frac12\xi\boldsymbol\om$ to obtain the bound
$I_\eps\le \boldsymbol\gm\tilde\eps$, where $\boldsymbol\gm$ depends only on $p$. Inserting this above leads to the
inequality
\begin{align*}
	\big|\big\{ u(\cdot, t)>k_{\tilde\eps}\big\}\cap K_{\varrho}\big|
	\le 
	\big(1-\tfrac12\nu\big)\big(1+\boldsymbol\gm\tilde\eps\big)
	|K_{\rho}|
	+
	\frac{\boldsymbol\gm\dl}{\sig^p}|K_\varrho|
	+
	N\sig |K_\varrho|.
\end{align*}
Now we first choose $\tilde\eps$ small enough so that
\begin{equation*}
	\big(1-\tfrac12\nu\big)\big(1+\boldsymbol\gm\tilde\eps\big)
	\le 
	1-\tfrac38\nu.
\end{equation*}
This fixes $\tilde\eps$ in dependence on $p$ and $\nu$. Then we fix $\sig:=\frac{\nu}{16N}$ 
and choose $\delta$ small enough to have $\frac{\boldsymbol\gm\dl}{\sig^p}\le\tfrac1{16} \nu$. Finally, the paramter
$\varep$ is chosen such that  $\dl\varep^{2-p}\ge1$.
The proof can now be finished by redefining $\tilde\eps\varep$ as $\eps$.
\end{proof}
Since $\bar{t}$ is arbitrary, we actually obtain the measure theoretical information
\begin{equation}\label{Eq:4:4}
	\Big|\Big\{ u(\cdot, t)\le\boldsymbol\mu^+-\eps\boldsymbol\om\Big\}
	\cap K_{\varrho}\Big|\ge\tfrac{1}4 \nu |K_\varrho|
\quad\mbox{ for all $t\in\big(-(A-1)\varrho^p,0\big]$.}
\end{equation}
\subsubsection{Shrinking the Measure Near the Supremum}\label{S:4:3:2}
By $\epsilon\in (0,1)$ we denote the constant from Lemma~\ref{Lm:4:2} depending only on the data. The number $A$ is still to be determined. We choose $A$ in the form $A=2^{j_*(p-2)}+1$
with some $j_*$ to be fixed later and define $Q_\varrho(\theta)=K_\varrho\times (-\theta\varrho^p,0]$
with $\theta=2^{j_*(p-2)}$. 
\begin{lemma}\label{Lm:A:4}
Suppose \eqref{Eq:4:3a} 
and \eqref{Eq:4:4} hold.
There exists $\boldsymbol \gm>0$ depending only on the data, such
that for any positive integer $j_*$, we have
\begin{equation*}
	\Big|\Big\{ u\ge\boldsymbol \mu^+-\frac{\eps\boldsymbol\om}{2^{j_*}}\Big\}\cap Q_\varrho(\theta)\Big|\le
	\frac{\boldsymbol\gm}{j_*^{\frac{p-1}p}}|Q_\varrho(\theta)|.
\end{equation*}
\end{lemma}
\begin{proof} For $j=0,\dots, j_\ast-1$ we employ the energy estimate in Proposition~\ref{Prop:2:1} in the cylinder $K_{2\varrho}\times (-\theta\varrho^p,0]$ 
with levels $k_j=\mu^+-2^{-j}\eps\boldsymbol\om$
and a time independent cutoff function $\z(x,t)\equiv\zeta (x)$, such that $\z$ equals $1$ in $K_{\varrho}$, vanishes
on $\pl K_{2\varrho}$, and such that $|D\z|\le2\varrho^{-1}$. 
Then, we obtain
\begin{align*}
	\iint_{Q_\varrho(\theta)}&|D(u-k_j)_+|^p\,\dx\dt\\
	&
	\le
	\int_{K_{2\varrho}\times\{ -\theta\varrho^p\}} \z^p \mathfrak g_+ (u,k_j) \,\dx
	+\boldsymbol\gm\iint_{K_{2\varrho}\times (-\theta\varrho^p,0]}(u-k_j)_+^p|D\z|^p\,\dx\dt.
\end{align*}
The first term on the right is estimated by Lemma~\ref{lem:g} and using $\xi\boldsymbol\om\le\mu^+\le2\boldsymbol\om$. We obtain
\begin{align*}
	\int_{K_{2\varrho}\times\{ -\theta\varrho^p\}}\z^p \mathfrak g_+ (u,k_j)\,\dx
	&\le
	\boldsymbol\gm\boldsymbol\om^{p-2}\left(\frac{\eps\boldsymbol\om}{2^j}\right)^2|K_{2\varrho}|\\
	&\le
	\frac{\boldsymbol \gm}{\varrho^p\eps^{p-2}}\left(\frac{\eps\boldsymbol\om}{2^j}\right)^p|Q_\rho(\theta)|\\
	&\le
	\frac{\boldsymbol \gm}{\varrho^p}\left(\frac{\eps\boldsymbol\om}{2^j}\right)^p|Q_\rho(\theta)|.
\end{align*}
In the second last line we used the fact that the paramter
$\eps$ is already fixed in dependence on the data.
Hence, the above energy estimate yields
\begin{equation*}
	\iint_{Q_\varrho(\theta)}|D(u-k_j)_+|^p\,\dx\dt
	\le
	\frac{\boldsymbol\gm}{\varrho^p}\left(\frac{\eps\boldsymbol\om}{2^j}\right)^p|Q_\varrho(\theta)|.
\end{equation*}
Next, we apply \cite[Chapter I, Lemma 2.2]{DB}
slicewise to $u(\cdot,t)$ for 
$t\in( -\theta\varrho^p,0]$ over the cube $K_\varrho$,
for levels $k_{j+1}>k_{j}$ and take into account the measure theoretical information from \eqref{Eq:4:4}, i.e.~that
\[
	\Big|\Big\{ u(\cdot, t)\le\boldsymbol\mu^+-\eps\boldsymbol\om\Big\}
	\cap K_{\varrho}\Big|\ge\tfrac{1}4 \nu |K_\varrho|
	\quad\mbox{ for all $t\in(-\theta \varrho^p,0]$.}
\]
This leads to
\begin{align*}
	(k_{j+1}-k_{j})&\big|\big\{u(\cdot, t)>k_{j+1} \big\}
	\cap K_{\varrho}\big|
	\\
	&\le
	\frac{\boldsymbol\gm \varrho^{N+1}}{\big|\big\{u(\cdot, t)<k_{j}\big\}\cap K_{\varrho}\big|}	
	\int_{\{ k_{j}<u(\cdot,t)<k_{j+1}\}\cap  K_{\varrho}}\!\!\!|Du(\cdot,t)|\,\dx\\
	&\le
	\frac{\boldsymbol \gm\varrho}{\nu}
	\bigg[\int_{\{k_{j}<u(\cdot,t)<k_{j+1}\}\cap K_{\varrho}}\!\!\!|Du(\cdot,t)|^p\,\dx\bigg]^{\frac1p}
	\big|\big\{ k_{j}<u(\cdot,t)<k_{j+1}\big\}\cap K_{\varrho}\big|^{1-\frac1p}
	\\
	&=
	\frac{\boldsymbol \gm\varrho}{\nu}
	\bigg[\int_{\{k_j<u(\cdot,t)<k_{j+1}\}\cap K_{\varrho}}\!\!\!|Du(\cdot,t)|^p\,\dx\bigg]^{\frac1p}
	\Big[ |A_j(t)|-|A_{j+1}(t)|\big]^{1-\frac1p}.
\end{align*}
In the last line we used the abbreviation $ A_j(t):= \big\{u(\cdot,t)<k_{j}\big\}\cap K_\varrho$.
We now integrate the preceding  inequality with respect to $t$ over  $(-\theta\varrho^p,0]$ and apply H\"older's inequality slice-wise.
Setting $A_j=\{u<k_j\}\cap Q_\varrho(\theta)$ this leads to the measure estimate
\begin{align*}
	\frac{\eps \boldsymbol\om}{2^{j+1}}\big|A_{j+1}\big|
	&\le
	\frac{\boldsymbol \gm\varrho}{\nu}\bigg[\iint_{Q_\varrho(\theta)}|D(u-k_j)_-|^p\,\dx\dt\bigg]^\frac1p
	\big[|A_j|-|A_{j+1}|\big]^{1-\frac{1}p}\\
	&\le\boldsymbol\gm \frac{\eps \boldsymbol\om}{2^j}|Q_\varrho(\theta)|^{\frac1p}
	\big[|A_j|-|A_{j+1}|\big]^{1-\frac{1}p}.
\end{align*}
Now take the power $\frac{p}{p-1}$ on both sides. This gives
\[
	\big|A_{j+1}\big|^{\frac{p}{p-1}}\le\boldsymbol \gm|Q_\varrho(\theta)|^{\frac1{p-1}}\big[|A_j|-|A_{j+1}|\big]
	.
\]
To finish the proof, we proceed exactly as in the proof of Lemma \ref{Lm:3:2}. We add the inequalities with respect to $j$ from $0$ to $j_*-1$ and obtain
\[
	j_* \big|A_{j_*}\big|^{\frac{p}{p-1}}\le\boldsymbol\gm\big|Q_\varrho(\theta)\big|^{\frac{p}{p-1}},
\]
from which we deduce the claim, i.e.~that
\[
	\big|A_{j_*}\big|\le\frac{\boldsymbol\gm}{j_*^{\frac{p-1}p}}|Q_\varrho(\theta)|.
\]
This completes the proof.
\end{proof}
\subsubsection{A DeGiorgi-type Lemma}\label{S:4:3:3}
As before, $\epsilon\in (0,1)$ denotes the constant from Lemma~\ref{Lm:4:2} depending only on the data.
\begin{lemma}\label{Lm:5:3}
Suppose that the assumptions \eqref{Eq:4:3a} and \eqref{Eq:4:3b} hold true.
Then, there exists a constant $\nu_1\in(0,1)$ depending only on 
the data, such that if for some $j_*>1$, the measure bound
\begin{equation*}
	\Big|\big\{\boldsymbol\mu^{+}-u\le\frac{\eps\boldsymbol\om}{2^{j_*}}\Big\} 
	\cap Q_{\varrho}(\theta)
	\Big|
	\le\nu_1|Q_{\varrho}(\theta)|,
\end{equation*}
holds true, where $\theta=2^{j_*(p-2)}$, then 
\[
	\boldsymbol\mu^{+}-u\ge\frac{\eps\boldsymbol\om}{2^{j_*+1}}\quad\mbox{a.e. in $Q_{\frac12\varrho}(\theta)$.}
\]
\end{lemma}
\begin{proof} 
Let $M:=2^{-j_*}\eps\boldsymbol\om$ and 
\begin{equation*}
	k_n:= \boldsymbol\mu^+-\frac{M}2-\frac{M}{2^{n+1}}.
\end{equation*}
Then, define $\tilde{k}_n$, $\varrho_n$, $\tilde{\varrho}_n$, $K_n$, $\widetilde{K}_n$, $Q_n$ and $\widetilde{Q}_n$ as in the proof of Lemma~\ref{Lm:3:3}, i.e.~as in \eqref{choices:k_n}. Introduce the cutoff functions $\z$ vanishing on the parabolic boundary of $Q_{n}$ and
equal to identity in $\widetilde{Q}_{n}$, such that
\begin{equation*}
	|D\z|\le\boldsymbol\gm\frac{2^n}{\varrho}\quad\text{ and }\quad |\z_t|\le\boldsymbol\gm\frac{2^{pn}}{\theta\varrho^p}.
\end{equation*}
Thus, using again the condition $\xi\boldsymbol\om\le \boldsymbol\mu^+\le 2\boldsymbol\om$  to estimate the 
terms in the energy inequality (see the proof of Lemma \ref{Lm:3:3}) we obtain
\begin{align*}
	\boldsymbol\om^{p-2}\essup_{-\theta\tilde\varrho^p<t<0} 
	&
	\int_{\widetilde{K}_n}\big(u-\tilde{k}_n\big)^2_+\,\dx
	+
	\iint_{\widetilde{Q}_n}\big|D\big(u-\tilde{k}_n\big)_+\big|^p\,\dx\dt\\
	&\le
	\boldsymbol\gm\frac{2^{pn}}{\varrho^p}M^{p}\left(1+\frac{\boldsymbol\om^{p-2}}{\theta M^{p-2}}\right)|A_n|
	=
	\boldsymbol\gm\frac{2^{pn}}{\varrho^p}M^{p}\big(1+\eps^{2-p}\big)|A_n|,
\end{align*}
where we abbreviated 
\[
	A_n=\big\{u>k_n\big\}\cap Q_n.
\]
The constant $\boldsymbol\gamma$ depends on the data and $\xi$. The latter dependence enters due to the estimate
from below of the sup-term in the energy inequality. Note that $\xi$ is already determined in dependence on the data. 
Now setting $\z$ to be a cutoff function which vanishes on the parabolic boundary of $\widetilde{Q}_n$
and equals identity in $Q_{n+1}$, an application of the Sobolev imbedding
\cite[Chapter I, Proposition~3.1]{DB} and the preceding estimate imply that
\begin{align*}
	&\bigg(\frac{M}{2^{n+2}}\bigg)^p
	|A_{n+1}|\le \iint_{\widetilde{Q}_n}\!\!\!\big(u-\tilde{k}_n\big)_+^p\z^p\,\dx\dt\\
	&
	\qquad\le
	\bigg[\iint_{\widetilde{Q}_n}\!\!\!\big[(u-\tilde{k}_n)_+\z\big]^{p\frac{N+2}{N}}\,\dx\dt\bigg]^{\frac{N}{N+2}}
	|A_n|^{\frac{2}{N+2}}\\
	&\qquad\le
	\boldsymbol
	\gm\bigg[\iint_{\widetilde{Q}_n}\!\!\!\big|D\big[(u-\tilde{k}_n)_+\z\big]\big|^p\,\dx\dt\bigg]^{\frac{N}{N+2}}
	\bigg[\essup_{-\theta\tilde{\varrho}_n^p<t<0}\int_{\widetilde{K}_n}\!\!\!\big(u-\tilde{k}_n\big)^2_-\,\dx
	\bigg]^{\frac{p}{N+2}}|A_n|^{\frac{2}{N+2}}\\
	&\qquad\le 
	\boldsymbol\gm\boldsymbol\om^{\frac{p(2-p)}{N+2}}
	\bigg(\frac{2^{pn}}{\varrho^p}M^{p}\bigg)^{\frac{N+p}{N+2}}
	\big(1+\eps^{2-p}\big)^{\frac{N+p}{N+2}}
	|A_n|^{1+\frac{p}{N+2}}.
\end{align*}
Setting $\boldsymbol Y_n=|A_n|/|Q_n|$, we arrive at
\begin{align*}
	\boldsymbol Y_{n+1}
	&\le
	\boldsymbol\gm \boldsymbol b^n
	\left(\frac{\theta M^{p-2}}{\boldsymbol\om ^{p-2}}\right)^{\frac{p}{N+2}}
	\big(1+\eps^{2-p}\big)^{\frac{N+p}{N+2}}
	\boldsymbol Y_n^{1+\frac{p}{N+2}} \\
	&=
	\boldsymbol\gm \boldsymbol b^n
	\eps^{\frac{p(p-2)}{N+2}}
	\left(1+\eps^{2-p}\right)^{\frac{N+p}{N+2}}
	\boldsymbol Y_n^{1+\frac{p}{N+2}},
\end{align*}
where  $\boldsymbol b=4^p$ and $\boldsymbol \gm$  only depends on the data.
Hence by \cite[Chapter I, Lemma~4.1]{DB}, there exists a constant $\nu_1\in(0,1)$
depending only on the data, such that $\boldsymbol Y_n\to0$ if we require the smallness condition
$\boldsymbol Y_o\le \nu_1$.
\end{proof}

We are now ready to conclude the reduction of oscillation near the supremum in the final case where \eqref{Eq:4:3a} and \eqref{Eq:4:3b} are satisfied. By $\epsilon\in(0,1)$, $\boldsymbol\gamma>0$ and $\nu_1\in(0,1)$ we denote the corresponding constants from Lemmas~\ref{Lm:4:2}, \ref{Lm:A:4} and \ref{Lm:5:3}. Then, we choose a positive integer $j_*$ in such a way that 
\begin{equation*}
	\frac{\boldsymbol\gm}{j_*^{\frac{p-1}p}}
	\le
	\nu_1. 
\end{equation*}
Applying in turn Lemmas~\ref{Lm:4:2}, \ref{Lm:A:4} and \ref{Lm:5:3} then yields that 
\[
	\boldsymbol\mu^{+}-u
	\ge
	\frac{\eps\boldsymbol\om}{2^{j_*+1}}\quad\mbox{a.e. in $\widetilde{Q}_1$,}
\]
where $\widetilde{Q}_1$ is defined in \eqref{def-Q1}. 
Here we used the fact that $Q_{\frac12\varrho}(\theta)\supset \widetilde{Q}_1$ since $\theta>1$. 
This implies a reduction of oscillation, i.e.~we have
\[
	\essosc_{\widetilde{Q}_1}u
	\le
	\Big(1-\frac{\eps\boldsymbol}{2^{j_*+1}}\Big)\boldsymbol\om.
\]

\subsection{Reduction of Oscillation Near Zero Concluded}
Let us first define quantities
\[
\lm=\min\left\{\frac14,\,\frac1{2A^{\frac1p}}\right\},\quad \bar{\eta}=\min\left\{\eta,\,\eta_1,\,\eta_2,\,\frac{\eps\boldsymbol}{2^{j_*+1}}\right\}.
\]
Now we may proceed by induction. Suppose up to $i=1,2,\cdots j-1$, we have built
\begin{equation*}
\left\{
\begin{array}{c}
\dsty\varrho_i=\lm\varrho_{i-1},\quad \boldsymbol\om_i=(1-\bar{\eta})\boldsymbol\om_{i-1},
\quad Q_i=K_{\varrho_i}\times(-\varrho_i^p,0],\\[5pt]
\dsty\boldsymbol\mu_i^+=\essup_{Q_i}u,\quad\boldsymbol\mu_i^-=\essinf_{Q_i}u,\quad
\essosc_{Q_i}u\le\boldsymbol\om_i.
\end{array}
\right.
\end{equation*}
For all the indices $i=1,2,\cdots j-1$, we always assume the first case in \eqref{Eq:Hp-main1}, i.e.
$$\boldsymbol\mu_i^-\le \xi\boldsymbol \om_i\quad
\text{ and }\quad\boldsymbol\mu_i^+\ge \xi\boldsymbol \om_i,$$
where $\xi$ is determined in Proposition~\ref{Prop:1:1}.
In this way the argument in the previous sections can be repeated, and we have for all $i=1,2,\cdots j$,
\[
\essosc_{Q_i}u\le(1-\bar{\eta})\boldsymbol\om_{i-1}=\boldsymbol\om_i.
\]
Consequently, iterating this recursive inequality we obtain for all $i=1,2,\cdots j$,
\begin{equation}\label{Eq:case11}
\essosc_{Q_i}u\le(1-\bar\eta)^i\boldsymbol\om
=\boldsymbol\om\left(\frac{\rho_i}{\rho}\right)^{\be_o}\quad\text{ where }\beta_o=\frac{\ln(1-\bar\eta)}{\ln\lm}.
\end{equation}

\subsection{Reduction of Oscillation Away From Zero}\label{S:4:4}
In this section, let us suppose $j$ is the first index satisfying the second case in \eqref{Eq:Hp-main1}, i.e.
\[
\text{either \quad$\boldsymbol\mu_j^-> \xi\boldsymbol \om_j$ \quad or\quad $\boldsymbol\mu_j^+< -\xi\boldsymbol \om_j$.}
\]
Let us treat for instance $\boldsymbol\mu_j^->\xi\boldsymbol\om_j$,
for the other case is analogous.
We observe that since $j$ is the first index for this to happen,
one should have $\boldsymbol\mu_{j-1}^-\le \xi\boldsymbol \om_{j-1}$.
Moreover, one estimates
\[
\boldsymbol\mu_{j}^-\le \boldsymbol\mu_{j-1}^- +\boldsymbol \om_{j-1}-\boldsymbol \om_{j}
\le(1+\xi)\boldsymbol \om_{j-1}-\boldsymbol \om_{j}=\frac{\xi+\bar\eta}{1-\bar\eta}\boldsymbol \om_{j}.
\]
As a result, we have
\begin{equation}\label{Eq:6:9}
\xi\boldsymbol \om_{j}\le\boldsymbol \mu_{j}^-\le\frac{\xi+\bar\eta}{1-\bar\eta}\boldsymbol \om_{j}.
\end{equation}
The condition \eqref{Eq:6:9} indicates that starting from $j$ the equation \eqref{Eq:1:1}
resembles the parabolic $p$-Laplacian type equation in $Q_j$. 
Like when $1<p<2$ (cf. Section~\ref{S:5:3}), we drop the suffix $j$ from our notation for simplicity,
and introduce $v\df{=}u/\boldsymbol\mu^-$ in $Q=K_{\rho}\times(-\rho^p,0]$.
As in Section~\ref{S:5:3}, $v$ satisfies \eqref{Eq:1:1} -- \eqref{Eq:1:2p}
with $\bl{A}(x,t,u,Du)$ replaced by some properly defined $\bar{\bl{A}}(x,t,v,Dv)$,
which is subject to the structural conditions \eqref{Eq:1:2p}.
Moreover,
\begin{equation}\label{Eq:6:10}
1\le v\le\frac{1+\xi}{\xi}\quad\text{ a.e. in }Q.
\end{equation}
As in Section~\ref{S:5:3}, it turns out to be more convenient to consider the equation satisfied by $w:=v^{p-1}$,
i.e.
\begin{equation*}
\partial_tw-\dvg\widetilde{\bl{A}}(x,t,w, Dw)=0\quad\text{ weakly in }Q,
\end{equation*}
where similarly as in  \eqref{def:tilde-A} we define  the vector-field $\widetilde{\bf A}$ by
\[
	\widetilde{\bl{A}}(x,t,y, \zeta)
	=\
	\bar{\bl{A}}\Big( x,t,\widetilde y^{\frac{1}{p-1}},\tfrac{1}{p-1} \widetilde y^{\frac{2-p}{p-1}}
	\zeta\Big),
\]
for a.e.~$(x,t)\in Q$, 
any $y\in\rr$ and any $\zeta\in\rn$. This time $\widetilde y$
is defined by
\[
	\widetilde y \df{=}\min\bigg\{ \max\big\{y,\tfrac12\big\}, 2\bigg(\frac{1+\xi}{\xi}\bigg)^{p-1}\bigg\}.
\]
It is easy to see that $w$ is in the same kind of functional space \eqref{Eq:1:3p}
 as $u$ and $v$ due to \eqref{Eq:6:10}.
Employing \eqref{Eq:6:10} again, we verify exactly as in Section~\ref{S:5:3} that there exist
absolute positive constants $\widetilde{C}_o=\boldsymbol\gamma_o (p,\xi)C_o$ and $\widetilde{C}_1= \boldsymbol\gamma_1 (p,\xi)C_1$, such that
\begin{equation*}
		\widetilde{\bl{A}}(x,t,y,\zeta)\cdot \zeta\ge \widetilde{C}_o|\zeta|^p 
		\quad\mbox{and}
		\quad
		|\widetilde{\bl{A}}(x,t,y,\zeta)|\le \widetilde{C}_1|\zeta|^{p-1},
\end{equation*}
for a.e.~$(x,t)\in Q$, any $y\in\rr$, and any $\zeta\in\rn$. Note that
$\xi$ is already fixed in dependence of the data. This shows
that $w$ is a local weak solution to the parabolic $p$-Laplacian type equation in $Q$.
We tend to use Proposition~\ref{Prop:5:1}.
To order for that, we first check the condition \eqref{Eq:5:8}
is satisfied. Indeed, recalling $v=u/\boldsymbol\mu^-$, $w=v^{p-1}$ and $\boldsymbol\om=\essosc_Q u$, 
we first use \eqref{Eq:6:10} and the mean value theorem to obtain
\[
(p-1)\essosc_Q v\le\boldsymbol{\widetilde{\om}}=\essosc_Q w\le(p-1)\left(\frac{1+\xi}{\xi}\right)^{p-2}\essosc_Q v.
\]
Since $\essosc_Q v=\boldsymbol\om/\boldsymbol\mu^-$,
this amounts to 
\[
(p-1)\frac{\boldsymbol\om}{\boldsymbol\mu^-}\le\boldsymbol{\widetilde{\om}}
\le(p-1)\left(\frac{1+\xi}{\xi}\right)^{p-2}\frac{\boldsymbol\om}{\boldsymbol\mu^-}.
\]
Then by \eqref{Eq:6:9}, we have
\[
c\df{=}(p-1)\frac{1-\bar\eta}{\xi+\bar\eta}\le\boldsymbol{\widetilde{\om}}\le 
\frac{p-1}{\xi}\left(\frac{1+\xi}{\xi}\right)^{p-2}\df{=}C.
\]
Thus we only need to take $\sig\le c^{p-2}$, such that
$Q_{\sig\rho}(\theta)\subset Q_\rho$ and the condition \eqref{Eq:5:8}
 in Proposition~\ref{Prop:5:1} is fulfilled.
As a result, the conclusion of Proposition~\ref{Prop:5:1} is obtained.
Moreover, the above upper bound of $\boldsymbol{\widetilde{\om}}$ actually allows us to obtain the set inclusion
\[
Q_{r}(\theta_o)\subset Q_r(\theta)\quad\text{ where }\theta_o=C^{2-p}.
\]
Using this set inclusion and rephrasing the oscillation decay in Proposition~\ref{Prop:5:1} in terms of $u$, we have
 for all $0<r<\rho$, that  the oscillation decay estimate
\[
	\essosc_{Q_r(\theta_o)}u\le\boldsymbol{\gm\om}\left(\frac{r}{\rho}\right)^{\be_1}
\]
holds true.
Now we revert to using the suffix $j$. The above oscillation decay then reads
\begin{equation}\label{Eq:case12}
	\essosc_{Q_r(\theta_o)}u\le\boldsymbol\gm\boldsymbol\om_j\left(\frac{r}{\rho_j}\right)^{\be_1}
\end{equation}
whenever $0<r<\rho_j$.
Combining \eqref{Eq:case11} and \eqref{Eq:case12},
we arrive at the desired conclusion, i.e., for all $0<r<\rho$ we have
\[
\essosc_{Q_r(\theta_o)}u\le\boldsymbol{\gm\om}\left(\frac{r}{\rho}\right)^{\be}\quad\text{ where }\be=\min\{\beta_o,\beta_1\}.
\]
A proper rescaling gives the oscillation decay in Remark~\ref{Rmk:1:1} and  completes the proof
of Theorem \ref{Thm:1:1}.

\section{Proof of Boundary Regularity} 
The proofs of Theorems~\ref{Thm:1:2} -- \ref{Thm:1:4}
present many similarities with the interior case. Hoewver, contrary to the interior case we do not need to distinguish between the cases $p<2$ and $p>2$. 
All the technical tools needed near the parabolic boundary
have been presented previously. Therefore, we will give sketchy proofs only,
while keeping reference to the tools and strategies used in
the interior and highlighting the main modifications.
\subsection{Proof of Theorem~\ref{Thm:1:2}}
Consider the cylinder of forward type $Q_o=K_{\rho}(x_o)\times(0,\rho^p]\subset E_T$
whose vertex $(x_o,0)$ is attached to the bottom of $E_T$.
We may assume $x_o=0$ and
set
\begin{equation*}
	\boldsymbol \mu^+=\essup_{Q_o}u,
	\quad
	\boldsymbol\mu^-=\essinf_{Q_o}u,
	\quad
	\boldsymbol\om=\boldsymbol\mu^+-\boldsymbol\mu^-.
\end{equation*}
Like in the proof of interior regularity, there are two main
cases to consider, namely 
\begin{equation}\label{Hp-main-initial}
\left\{
\begin{array}{c}
\mbox{when $u$ is near zero:  $\boldsymbol\mu^-\le\boldsymbol\om$ and 
$\boldsymbol\mu^+\ge-\boldsymbol\om$};\\[5pt]
\mbox{when $u$ is away from zero: $\boldsymbol\mu^->\boldsymbol\om$ or $\boldsymbol\mu^+<-\boldsymbol\om$.}
\end{array}\right.
\end{equation}
Let us suppose the first case holds, which implies $|\boldsymbol\mu^{\pm}|\le2\boldsymbol\om$.
The proof continues with a comparison to the initial datum $u_o$, i.e., we may assume
\[
	\mbox{either \quad$\dsty\boldsymbol \mu^+-\tfrac14\boldsymbol \om>\sup_{K_\rho}u_o\;\;$ 
	or
	$\;\;\dsty\boldsymbol \mu^-+\tfrac14\boldsymbol \om<\inf_{K_\rho}u_o$.}
\]
For otherwise, we would arrive at
\[
\essosc_{Q_o}u\le2\essosc_{K_\rho}u_o.
\]
Let us assume for instance the second inequality with $\boldsymbol \mu^-$ holds
and work with $u$ as a super-solution.
Therefore, we let $\theta\in(0,1)$ to be chosen later and define $\tilde{k}_n$, $\varrho_n$, $\tilde{\varrho}_n$, $K_n$ and
$\widetilde{K}_n$,  
as in the proof of Lemma~\ref{Lm:3:3} according to \eqref{choices:k_n},
with $M$ replaced by $\frac14\boldsymbol\om$. The only difference is that now the cylinders 
$Q_n$ and $\widetilde{Q}_n$ are of forward type whose vertices are attached to the origin,
i.e.,~$Q_n=K_n\times(0,\theta\rho_n^p]$ and $\widetilde{Q}_n=\widetilde{K}_n\times(0,\theta\tilde{\rho}_n^p]$.
With these choices,
we may apply the energy estimates in Proposition~\ref{Prop:2:2} within $Q_n$, 
since the levels $k_n$ are admissible according to \eqref{Eq:3:2}, i.e.,
$$k_n\le\boldsymbol \mu^-+\tfrac14\boldsymbol \om<\inf_{K_\rho}u_o.$$
Using $|\boldsymbol\mu^{-}|\le2\boldsymbol\om$, 
a similar analysis as in the proof of Lemma~\ref{Lm:3:3} leads us to the analogue of \eqref{Eq:sample}, i.e.~to the energy estimate
\begin{align*}
	\frac{\boldsymbol\om^{p-2}}{2^{p(n+3)}} \essup_{0<t<\theta\varrho_n^p}
	\int_{\widetilde{K}_n} (u-\tilde{k}_n)_-^2\,\dx
	+
	\iint_{\widetilde{Q}_n}|D(u-\tilde{k}_n)_-|^p \,\dx\dt
	\le
	\boldsymbol\gm \frac{2^{pn}}{\varrho^p}\boldsymbol\om^{p}|A_n|,
\end{align*}
where we have abbreviated 
\[
	A_n=\big\{u>k_n\big\}\cap Q_n.
\]
Note that in \eqref{Eq:sample} we only have to replace $M$ by $\frac14\boldsymbol \om$.
Now we are in a situation similar to Lemma~\ref{Lm:6:1}. More precisely, 
\eqref{Eq:sample*}  holds true with $\eta_1=1$ and -- on the right-hand side -- $2^{pn}$ replaced by $4^{pn}$.
Now, applying the Sobolev imbedding as in the proof of Lemma~\ref{Lm:6:1} and rewriting the resulting
estimate in terms of $\boldsymbol Y_n=|A_n|/|Q_n|$, we  arrive at
\[
	\boldsymbol Y_{n+1}\le\boldsymbol\gm \boldsymbol b^n\theta^{\frac{p}{N+2}}
	\boldsymbol Y_n^{1+\frac{p}{N+2}},
\]
where $\boldsymbol\gm$ and $\boldsymbol b$ are positive constants depending only on the data.
Hence by \cite[Chapter I, Lemma~4.1]{DB}, there exists a constant $\nu_o\in(0,1)$
depending only on the data, such that $\boldsymbol Y_n\to0$ as $n\to\infty$ if we require that
\begin{equation*}
	\boldsymbol Y_o\le\frac{\nu_o}{\theta}.
\end{equation*}
Upon choosing $\theta=\nu_o$, the above line is automatically satisfied
and as a result we obtain
\[
	u\ge\boldsymbol\mu^-+\tfrac18\boldsymbol\om\quad\mbox{a.e.~in 
	$\widehat{Q}_1\df{=}K_{\frac12\rho}\times\big(0,\theta(\tfrac12\rho)^p\big]$.}
\]
This in turn yields a reduction of oscillation of the form
\[
	\essosc_{\widehat{Q}_1}u\le\tfrac78\boldsymbol\om .
\]
Consequently,  taking the initial datum into consideration, we obtain (see 
\cite[Chapter III, Lemma~11.1]{DB} for the corresponding estimate for weak solutions to parabolic $p$-Laplacian
equations)
\[
	\essosc_{\widehat{Q}_1}u\le\max\Big\{ \tfrac78\boldsymbol\om, 2\boldsymbol\om_{u_o}(\rho)\Big\}.
\]
Now we may proceed by induction. Define $\lm=\tfrac12\theta^{\frac1p}$, and 
suppose up to $i=1,2,\cdots j-1$, we have built sequences 
\begin{equation*}
\left\{
	\begin{array}{c}
	\dsty\varrho_i=\lm\varrho_{i-1},
	\quad 
	\boldsymbol\om_i=\max\Big\{\tfrac78\essosc_{Q_{i-1}}u,2\boldsymbol\om_{u_o}(\rho_{i-1})\Big\},
	\quad 
	Q_i=K_{\varrho_i}\times(0,\varrho_i^p],\\[5pt]
	\dsty\boldsymbol\mu_i^+=\essup_{Q_i}u,
	\quad
	\boldsymbol\mu_i^-=\essinf_{Q_i}u,
	\quad
	\essosc_{Q_i}u\le\boldsymbol\om_i.
	\end{array}
\right.
\end{equation*}
For all the indices $i=1,2,\cdots j-1$, we always assume the first alternative \eqref{Hp-main-initial}, i.e.~that
$$
	\boldsymbol\mu_i^-\le\essosc_{Q_i}u\quad
	\text{ and }\quad\boldsymbol\mu_i^+\ge-\essosc_{Q_i}u
$$
holds true. In this way the above argument can be repeated and we have for all $i=1,2,\cdots j$, the reduction of oscillation
\[
	\essosc_{Q_i}u
	\le
	\max\Big\{ \tfrac78\essosc_{Q_{i-1}}u, 2\boldsymbol\om_{u_o}(\rho_{i-1})\Big\}
	=\boldsymbol\om_i.
\]
Consequently, iterating the above recursive inequality, we obtain for all $i=1,2,\cdots j$,
\begin{align}\label{Eq:case-j}
	\essosc_{Q_i}u&\le
	\boldsymbol\om \bigg(\frac78\bigg)^i
	+
	2\boldsymbol\om_{u_o}(\rho)\sum_{j=0}^{i-1}\bigg(\frac78\bigg)^j
	\le
	\boldsymbol\om\left(\frac{\rho_i}{\rho}\right)^{\be_o} +16\boldsymbol\om_{u_o}(\rho),
\end{align}
where
\[
	\beta_o=\frac{\ln\frac78}{\ln\lm}. 
\]

In what follows, let us suppose $j$ is the first index satisfying 
\[
	\mbox{either $\;\boldsymbol\mu_j^-> \boldsymbol \om_j\;$
	or
	$\;\boldsymbol\mu_j^+<-\boldsymbol \om_j$.}
\]
Let us treat for instance $\boldsymbol\mu_j^->\boldsymbol\om_j$,
for the other case is analogous.
As in Section~\ref{S:5:3} we use the fact that $j$ is the first index for this to happen to obtain
\begin{equation*}
	\boldsymbol \om_{j}
	<
	\boldsymbol \mu_{j}^-\le\tfrac97\boldsymbol \om_{j};
\end{equation*}
cf.~the proof of \eqref{Eq:5:4}. 
Then,   for simplicity we drop the suffix $j$ from our notation temporarily,
and introduce $v:=u/\boldsymbol\mu^-$ and $w:=v^{p-1}$ in $Q=K_{\rho}\times(0,\rho^p]$.
In this way, the function $w$ will satisfy the parabolic $p$-Laplacian type equation \eqref{Eq:5:6} -- \eqref{Eq:5:7}.
Moreover, it attains the initial datum $w_o:=(u_o/\boldsymbol\mu^-)^{p-1}$ in the sense of $L^2(K_\rho)$.
Next, we state in the following proposition concerning the regularity of solutions
to the parabolic $p$-Laplacian type equation up to the initial time (cf. \cite{DB}). 

\begin{proposition}\label{Prop:7:1}
Let $p>1$ and $\sig$ in $(0,1)$.
Suppose $w$ is a bounded, local, weak solution to 
\eqref{Eq:5:6}  in $Q:=Q_\rho$ such that the structure conditions \eqref{Eq:5:7} are in force, and such that
$w(\cdot,t)\to w_o$ as $t\downarrow0$ in the sense of $L^2(K_\rho)$.
Assume $w_o$ is continuous in $K_\rho$ with modulus of continuity $\boldsymbol\om_{w_o}(\cdot)$. Let
$$
	\boldsymbol{\widetilde{\om}}=\essosc_Q w
	\quad\mbox{and}\quad
	\theta=\boldsymbol{\widetilde{\om}}^{2-p}.
$$
Then, there exist constants $\beta_1$ in $(0,1)$ and $\boldsymbol\gm>1$ depending only on the data
$N,p,\widetilde C_o,\widetilde C_1$ and $\sig$ (but independent of $w$), such that there holds:
whenever  we have
\begin{equation*}
	\essosc_{Q_{\sig\varrho}(\theta)}w\le\boldsymbol{\widetilde{\om}},
\end{equation*}
then the oscillation decay estimate
\[
	\essosc_{Q_r(\theta)}w\le\boldsymbol\gm\boldsymbol{\widetilde{\om}}\left(\frac{r}{\rho}\right)^{\be_1}
	+\boldsymbol\gm\boldsymbol\om_{w_o}(\rho)
\]
holds true for all $0<r<\rho$. Here, we use the notation $Q_r(\theta)=B_r\times(0,\theta r^p]$. 
\end{proposition}
As in Sections~\ref{S:5:3} and \ref{S:4:4} (distinguishing the cases $1<p<2$ and $p>2$), one quickly checks that there exist absolute constants $c,\,C>0$, such that
$c\le\boldsymbol{\widetilde{\om}}\le C$ and
the condition in Proposition~\ref{Prop:7:1} is fulfilled for some proper $\sig$.
Moreover, the lower/upper bounds of $\boldsymbol{\widetilde{\om}}$ actually allow us to obtain the set inclusion
\[
	Q_{r}(\theta_o)\subset Q_r(\theta)\quad\text{ where }
	\theta_o=\min\left\{c^{2-p},\,C^{2-p}\right\}.
\]
Using this set inclusion and rephrasing the oscillation decay of Proposition~\ref{Prop:7:1} in terms of $u$, we have
\[
	\essosc_{Q_r(\theta_o)}u
	\le
	\boldsymbol\gm\boldsymbol\om\left(\frac{r}{\rho}\right)^{\be_1}
	+\boldsymbol\gm\boldsymbol\om_{u_o}(\rho),
\]
whenever  $0<r<\rho $; here we argue similarly to the proof of \eqref{Eq:case12}.
Now we revert to using the suffix $j$. The above oscillation estimate then reads as
\begin{equation*}
	\essosc_{Q_r(\theta_o)}u
	\le
	\boldsymbol\gm\boldsymbol\om_j\left(\frac{r}{\rho_j}\right)^{\be_1}
	+\boldsymbol\gm\boldsymbol\om_{u_o}(\rho_j)
\end{equation*}
for all $0<r<\rho_j$.
Combining the above two cases, we arrive at the desired conclusion, i.e., for all $0<r<\rho$, there holds
\[
	\essosc_{Q_r(\theta_o)}u
	\le
	\boldsymbol\gm\boldsymbol\om\left(\frac{r}{\rho}\right)^{\be}
	+\boldsymbol\gm\boldsymbol\om_{u_o}(\rho)\quad\text{ where }\be=\min\{\beta_o,\beta_1\}.
\]
Observe that we may replace $\rho$ by any $\tilde{\rho}\in(r,\rho)$.
In particular, we may set $\tilde\rho=\sqrt{r\rho}$. In this way, we end up with
\[
	\essosc_{Q_r(\theta_o)}u
	\le
	\boldsymbol\gm\boldsymbol\om\left(\frac{r}{\rho}\right)^{\frac\be2}
	+\boldsymbol\gm\boldsymbol\om_{u_o}(\sqrt{r\rho}).
\]

\subsection{Proof of Theorem~\ref{Thm:1:3}}
For $A\ge1$ to be determined, consider the cylinder  $Q_o=K_{\rho}(x_o)\times(t_o-A\rho^p,t_o]$
whose vertex $(x_o,t_o)$ is attached to $S_T$.
The number $\rho$ is so small that $t_o-A\rho^p>0$ and $\rho<\rho_o$
where $\rho_o$ is the constant from the property of positive geometric density \eqref{geometry}.
We may also assume $(x_o,t_o)=(0,0)$ and
set
\begin{equation*}
	\boldsymbol \mu^+=\essup_{Q_o\cap E_T}u,
	\quad
	\boldsymbol\mu^-=\essinf_{Q_o\cap E_T}u,
	\quad
	\boldsymbol\om=\boldsymbol\mu^+-\boldsymbol\mu^-.
\end{equation*}
Like in the proof of interior regularity, there are two main
cases to consider, namely 
\begin{equation}\label{Hp-main-Dirichlet}
\left\{
\begin{array}{c}
\mbox{when $u$ is near zero:  $\boldsymbol\mu^-\le\xi\boldsymbol\om$ and 
$\boldsymbol\mu^+\ge-\xi\boldsymbol\om$};\\[5pt]
\mbox{when $u$ is away from zero: $\boldsymbol\mu^->\xi\boldsymbol\om$ or $\boldsymbol\mu^+<-\xi\boldsymbol\om$.}
\end{array}\right.
\end{equation}
Here $\xi$ is the positive constant fixed in Proposition~\ref{Prop:1:1} through dependence on the data and $\al=\al_*$,
whereas $\al_*$ comes from the property of positive geometric density of $\pl E$.

Let us suppose the first case holds. 
The proof continues with a comparison to the boundary datum $g$, i.e., we may assume
\begin{equation*}
	\mbox{either}\quad\boldsymbol \mu^+-\tfrac14\boldsymbol \om>\sup_{Q_o\cap S_T}g\;
	\quad\mbox{or} \quad
	\; \boldsymbol \mu^-+\tfrac14\boldsymbol \om<\inf_{Q_o\cap S_T}g.
\end{equation*}
For otherwise, we would arrive at
\[
\essosc_{Q_o}u\le2\essosc_{Q_o\cap S_T}g.
\]
Let us suppose for instance the second inequality holds. To proceed, we turn our attention to the energy estimates in Proposition~\ref{Prop:2:2} for super-solutions.
Since $(u-k)_-$ vanishes on $Q_o\cap S_T$ for all $k\le\boldsymbol \mu^-+\tfrac14\boldsymbol \om$
(i.e.~$k$ satisfies \eqref{Eq:3:3}$_2$ with $Q_{R,S}$ replaced by $Q_o$),
we may extend all integrals in the energy estimates to zero outside of $E_T$.
The extended $(u-k)_-$ will be denoted by the same symbol and it is still
a member of the functional space in \eqref{Eq:1:3p} within $Q_o$.

The proofs of Lemma~\ref{Lm:3:2} and Lemma~\ref{Lm:3:3}
can be carried over to the current situation with properly chosen parameters, bearing in mind that
we have assumed $\pl E$ fulfills the property of positive geometric density \eqref{geometry}, and therefore
for any $k\le\boldsymbol \mu^-+\tfrac14\boldsymbol \om$, we have
\begin{equation}\label{Eq:7:3}
	\big|\big \{(u(\cdot, t)-k)_-=0\big\}\cap K_\rho(x_o)\big|\ge\al_*\big|K_\rho\big |\quad
	\mbox{for all $t\in(-A\rho^p,0]$.}
\end{equation}
Thus the conclusion of Proposition~\ref{Prop:1:1} can be reached. As a result
the oscillation is reduced in the case $1<p<2$ under the condition \eqref{Hp-main-Dirichlet}$_1$
with $\xi=1$ just like in Section~\ref{S:5},
whereas this is true for $p>2$ only when $|\boldsymbol\mu^-|<\xi\boldsymbol\om$
with some very small $\xi$.
As a result, one still needs to handle the situation when $\boldsymbol\mu^-<-\xi\boldsymbol\om$
since this is not excluded in \eqref{Hp-main-Dirichlet}$_1$, for $p>2$.

Like in Section~\ref{S:6}, there seems to be some technical complication due to
the smallness of the parameter $\xi$ in the case $p>2$. However, the property \eqref{Eq:7:3}
offers considerable simplification. Indeed, we do not need to split the proof into two parts. Our current hypothesis to continue consist of \eqref{Eq:7:3}
and $-2\boldsymbol\om<\boldsymbol\mu^-<-\xi\boldsymbol\om$
as we have assumed $\boldsymbol\mu^+\ge-\xi\boldsymbol\om$ in \eqref{Hp-main-Dirichlet}$_1$.
They are analogues of \eqref{Eq:4:3a} and \eqref{Eq:4:4} formulated near the infimum instead of the supremum, with which one can run the machinery
employed in Lemma~\ref{Lm:A:4} and Lemma~\ref{Lm:5:3}.
The only difference is that Lemma~\ref{Lm:A:4} and Lemma~\ref{Lm:5:3}
have been presented in terms of sub-solutions near the supreme,
whereas now one needs to reproduce similar arguments in terms of 
super-solutions near the infimum.
Therefore we can reduce the oscillation under the condition \eqref{Hp-main-Dirichlet}$_1$, for $p>2$ as well.

Next, we can proceed by induction just like the interior case until a certain index $j$,
when the second case of  \eqref{Hp-main-Dirichlet} happens for the first time.
Starting from $j$, the equation will behave like the parabolic $p$-Laplacian type equation within $Q_j\cap E_T$.
We may render this point technically just like in the interior case.
Accordingly, we need the following result near the lateral boundary (cf. \cite{DB}).
\begin{proposition}\label{Prop:7:2}
Let $p>1$ and $\sig\in (0,1)$.
Suppose $w$ is a bounded, local, weak solution to 
\eqref{Eq:5:6} -- \eqref{Eq:5:7} in $Q_\rho\cap E_T$
and $w= g$ on $Q_\rho\cap S_T$ in the sense defined in Section~\ref{S:1:4:3}.
Assume $g$ is continuous on $S_T$ with modulus of continuity $\boldsymbol\om_{g}(\cdot)$.  Define
$$
	\boldsymbol{\widetilde{\om}}=\essosc_{Q_\rho\cap E_T} w\quad
	\mbox{and}\quad \theta=\boldsymbol{\widetilde{\om}}^{2-p}.
$$
 Then, there exist constants $\beta_1$ in $(0,1)$ and $\boldsymbol\gm>1$ depending only on the data
$N,p,\widetilde C_o, \widetilde C_1$ and $\sig$ (but independent of $w$), such that there holds: if
\begin{equation}
	\essosc_{Q_{\sig\varrho}(\theta)\cap E_T}w\le\boldsymbol{\widetilde{\om}}
	,
\end{equation}
then,  the oscillation decay estimate 
\[
	\essosc_{Q_r(\theta)\cap E_T}w\le\boldsymbol\gm\boldsymbol{\widetilde{\om}}\left(\frac{r}{\rho}\right)^{\be_1}
	+\boldsymbol\gm\boldsymbol\om_{g}(\rho)
\]
holds true for all $0<r<\rho$.
\end{proposition}
We refrain from further elaboration due to the similarity of the arguments.
The proof may be concluded as in the previous section.
\begin{remark}\upshape
We have omitted actual computations due to similarities with the proof of interior regularity.
Nevertheless the conclusions in the interior, such as Proposition~\ref{Prop:1:1}, can be applied directly to the current situation
without repeating their proofs, thanks to their emphasis on the distinct roles of sub-solutions and super-solutions,
provided we can extend $u$ properly to the outside of $E_T$ and generate
sub(super)-solutions across the lateral boundary.
In this regard, we refer to Lemma~\ref{Lm:A:2} for such extensions.
\end{remark}
\subsection{Proof of Theorem~\ref{Thm:1:4}}
First of all, we observe that the proof of interior regularity (Theorem~\ref{Thm:1:1}) 
consists of two main components, namely, the expansion of positivity (Proposition~\ref{Prop:1:1}) which is based solely on
the energy estimates in Proposition~\ref{Prop:2:1} and a corresponding
H\"older estimate for solutions to the parabolic $p$-Laplacian type equation.

This observation is also essentially the gist in the proofs of Theorem~\ref{Thm:1:2} -- \ref{Thm:1:3}.
Similar calculations have to be reproduced 
mainly due to the variant energy estimates in Proposition~\ref{Prop:2:2} -- \ref{Prop:2:3}
that have incorporated either initial data or Dirichlet data.
In particular, a key ingredient -- the Sobolev imbedding (cf. \cite[Chapter I, Proposition~3.1]{DB}) -- was used in all
these situations, assuming the functions $(u-k)_{\pm}\z^p$ vanish on the lateral boundary of the domain of integration.
This assumption in turn is fulfilled either by choosing a proper cutoff function $\z$ 
or by restricting the value of the level $k$ according to the Dirichlet data as in \eqref{Eq:3:3}
or the initial data as in \eqref{Eq:3:2}.

The main difference in the current situation lies in that such a Sobolev imbedding cannot be used
because in general the functions $(u-k)_{\pm}\z^p$ under conditions of Proposition~\ref{Prop:2:4} do not vanish on $S_T$.
However, a similar Sobolev imbedding (cf. \cite[Chapter. I, Proposition 3.2]{DB})
 that does not require functions to vanish on the boundary
still holds for the functional space
\begin{equation*}  
	u\in C\big(0,T;L^p(E)\big)\cap L^p\big(0,T; W^{1,p}(E)\big).
\end{equation*}
It is remarkable that the imbedding constant now depends on $N$, 
the structure of $\pl E$ and the ratio $T/|E|^{\frac{p}N}$,
which is invariant for cylinders of the type $Q_\rho=K_{\rho}\times(-\rho^p,0]$
and $Q_\rho\cap E_T$ as well, provided $\pl E$ is smooth enough.

As an example,  we exhibit in the following how to modify the proof of Lemma~\ref{Lm:3:3} technically.
Based on Proposition~\ref{Prop:2:4} and under the similar notations in Lemma~\ref{Lm:3:3}, 
with the interior cylinders replaced by their intersection with $E_T$,
the energy estimate \eqref{Eq:sample} becomes, assuming $|\boldsymbol\mu^-|\le 8M$,
\begin{equation*}
\begin{aligned}
	\frac{M^{p-2}}{2^{p(n+3)}} &\essup_{-\theta\tilde{\varrho}_n^p<t<0}
	\int_{\widetilde{K}_n\cap E} (u-\tilde{k}_n)_-^2\,\dx
	+
	\iint_{\widetilde{Q}_n\cap E_T}|D(u-\tilde{k}_n)_-|^p \,\dx\dt\\
	&\le
	\boldsymbol\gm \frac{2^{pn}}{\varrho^p}M^{p}|A_n|+\boldsymbol\gm_* |A_n|,
\end{aligned}
\end{equation*}
where we have abbreviated
\[
A_n=\{u<k_n\}\cap Q_n\cap E_T.
\]
The term with $\boldsymbol\gm_*$ comes from the extra term generated by the Neumann datum $\psi$
and $\boldsymbol\gm_*$ depends on $C_2$ through \eqref{N-data}. We may assume that
the first term on the right dominates the second, for otherwise we would have
$M\le(\frac{\boldsymbol\gm_*}{\boldsymbol\gm})^\frac1p\rho$. When $p>2$, 
the first integral on the left may be estimated from below by
\[
\int_{\widetilde{K}_n\cap E} (u-\tilde{k}_n)_-^2\,\dx\ge M^{2-p}\int_{\widetilde{K}_n\cap E} (u-\tilde{k}_n)_-^p\,\dx.
\]
When $1<p<2$, we introduce $\hat{k}_n=\frac34 k_{n+1}+\frac14 k_n<\tilde{k}_n$ and estimate
\[
\int_{\widetilde{K}_n\cap E} (u-\tilde{k}_n)_-^2\,\dx\ge (\tilde{k}_n-\hat{k}_n)^{2-p}\int_{\widetilde{K}_n\cap E} (u-\hat{k}_n)_-^p\,\dx
=\left(\frac{M}{2^{n+4}}\right)^{2-p}\int_{\widetilde{K}_n\cap E} (u-\hat{k}_n)_-^p\,\dx.
\]
In all cases, the energy estimate becomes
\begin{equation*}
\begin{aligned}
	\frac{1}{2^{3p+2n}} \essup_{-\theta\tilde{\varrho}_n^p<t<0}
	\int_{\widetilde{K}_n\cap E} (u-\hat{k}_n)_-^p\,\dx
	+
	\iint_{\widetilde{Q}_n\cap E_T}|D(u-\hat{k}_n)_-|^p \,\dx\dt
	\le
	\boldsymbol\gm \frac{2^{pn}}{\varrho^p}M^{p}|A_n|.
\end{aligned}
\end{equation*}
Then one may proceed to use the previously mentioned Sobolev imbedding 
(cf. \cite[Chapter. I, Proposition 3.2]{DB}) to establish a recursive inequality
of fast geometric convergence for $\boldsymbol{Y}_n=|A_n|/|Q_n\cap E_T|$
as in Lemma~\ref{Lm:3:3}.

As shown above, whenever we use the energy estimate in Proposition~\ref{Prop:2:4},
 the extra term containing $C_2$ is always absorbed into
other terms via the assumption $M>(\frac{\boldsymbol\gm_*}{\boldsymbol\gm})^\frac1p\rho$,
and as such it contributes in the proof of H\"older regularity only by
an extra control on the oscillation via $\boldsymbol\om\le \boldsymbol\gm \rho$
with $\boldsymbol\gm$ depending also on $C_2$. This remark also holds for the proof
of a result like Lemma~\ref{Lm:3:1}.

Finally, we remark that the use of De Giorgi's isoperimetric inequality (cf. \cite[Chapter I, Lemma~2.2]{DB})
is permitted for all convex domains. This is not restrictive in our case upon a local flattening of $\pl E$. 
In other words, since $\pl E$ is of class $C^1$, the portion of $\pl E$ within $K_R(x_o)$
 can be represented in a local coordinate system as part of the hyperplane $x_N=0$
 and $K_R(x_o)\cap E\subset \{x_N>0\}$.
 Set
 $K_R^+:=K_R\cap\{x_N>0\}$ and $Q_{R,S}^+:=Q_{R,S}\cap\{x_N>0\}$.
 Without loss of generality we may assume that the weak formulation in Section~\ref{S:1:4:4}
is written in such a coordinate system. Consequently, the energy estimate in Proposition~\ref{Prop:2:4}
is written with $K_R(x_o)\cap E$ and $Q_{R,S}\cap E_T$ replaced by $K_R^+$ and $Q_{R,S}^+$ respectively.
Thus the machinery used in Lemma~\ref{Lm:3:2} can be reproduced
with modifications as indicated above.

As usual, another main component of the induction argument will be a corresponding result
from the regularity theory for the parabolic $p$-Laplacian type equation (cf. \cite{DB}),
which we record in the following.
\begin{proposition}\label{Prop:7:3}
Let $p>1$ and $\pl E$ be of class $C^1$.
Suppose $w$ is a bounded, local, weak solution to 
\eqref{Eq:5:6} -- \eqref{Eq:5:7} in $Q_\rho\cap E_T$
 with the Neumann datum $\psi$ on $Q_\rho\cap S_T$ taken in the sense defined in Section~\ref{S:1:4:4}.
Assume $\psi$ satisfies \eqref{N-data} and define
$$
	\boldsymbol{\widetilde{\om}}=\essosc_{Q_\rho\cap E_T} w.
$$
If for some  constant $\sig$ in $(0,1)$, there holds
\begin{equation}
	\essosc_{Q_{\sig\varrho}(\theta)\cap E_T}w\le\boldsymbol{\widetilde{\om}}
	\quad
	\text{ where }
	\theta=\boldsymbol{\widetilde{\om}}^{2-p}.
\end{equation}
then, there exist constants $\beta_1$ in $(0,1)$ and $\boldsymbol\gm>1$ depending only on the data
$N,p,\widetilde C_o, \widetilde C_1, C_2, \sig$ and the structure of $\pl E$, such that
for all $0<r<\rho$, we have
\[
	\essosc_{Q_r(\theta)\cap E_T}w\le\boldsymbol\gm\boldsymbol{\widetilde{\om}}\left(\frac{r}{\rho}\right)^{\be_1}.
\]
\end{proposition}

\appendix

\section{More on the Notion of Parabolicity}\label{Append:1}
\begin{lemma}\label{Lm:A:1}
Let $u$ be a local weak sub(super)-solution to \eqref{Eq:1:1} -- \eqref{Eq:1:2p}.
Then, 
for any $k\in\rr$, the truncation $k\pm(u-k)_\pm$
is a local weak sub(super)-solution to \eqref{Eq:1:1} -- \eqref{Eq:1:2p}.
\end{lemma}
\begin{proof}
Without loss of generality, let $u$ be a local weak sub-solution to \eqref{Eq:1:1} -- \eqref{Eq:1:2p}.
We show that $k+(u-k)_+$
is a local weak sub-solutions to \eqref{Eq:1:1} -- \eqref{Eq:1:2p}.
Write down the mollified equation \eqref{mol-eq} and follow the introduction
of $Q_{R,S}\Subset E_T$, the function $w_h$ and the piecewise smooth functions $\z$ and $\psi_{\eps}$ in the proof of 
Proposition~\ref{Prop:2:1}. Instead of \eqref{Test}, we choose here, for some $\sig>0$, the test function
\begin{equation}\label{Test1}
	Q_{R,S}\ni(x,t)\mapsto
	\varphi(x,t) 
	= 
	\zeta^p(x,t)\psi_\varepsilon(t) \frac{\big(u(x,t)-k\big)_+}{\big(u(x,t)-k\big)_++\sig}. 
\end{equation}
Like in the proof of Proposition~\ref{Prop:2:1}, we treat the various terms in \eqref{mol-eq}.
First of all, we consider the time part. We have
\begin{align*}
	\iint_{E_T} 
	\partial_t \power{w_h}{p{-}1} \varphi \,\dx\dt 
	&= 
	\iint_{Q_{R,S}}
	\zeta^p  \psi_\varepsilon \partial_t \power{w_h}{p{-}1} \frac{(w_h-k)_+}{(w_h-k)_++\sig}
	\dx\dt \\
	&\phantom{=\,} +
	\iint_{Q_{R,S}}
	\zeta^p  \psi_\varepsilon \partial_t \power{w_h}{p{-}1} 
	\bigg[ \frac{(u-k)_+}{(u-k)_++\sig}-\frac{(w_h-k)_+}{(w_h-k)_++\sig}\bigg]
	\dx\dt \\
	&\ge 
	\iint_{Q_{R,S}}
	\zeta^p  \psi_\varepsilon 
	\partial_t \mathfrak h_+ (w_h,\sig,k)  \dx\dt \\
	& = 
	- \iint_{Q_{R,S}} 
	\big(\zeta^p  \psi_\varepsilon'+ \psi_\varepsilon\partial_t\zeta^p\big) \mathfrak h_+ (w_h,\sig,k) 
	\dx\dt,
\end{align*}
where we have defined
\[
	 \mathfrak h_+ (u,\sig,k)
	 \df{=} 
	 (p-1)\bigg[\power{k}{p{-}1} + 
	 \int_k^{u}\frac{|s|^{p-2}(s-k)_+}{(s-k)_++\sig}\d s\bigg].
\]
Note that $\lim_{\sig\downarrow0} \mathfrak h_+(u(x,t),\sig,k)=\power{[k+(u-k)_+]}{p{-}1}$.
As in the proof of Proposition~\ref{Prop:2:1}, we have used the fact that the second line in the above estimates
has a non-negative contribution, due to \eqref{dt-wh} and 
the fact that the map 
\[
	s\mapsto\frac{(s^{\frac1{p-1}}-k)_+}{(s^{\frac1{p-1}}-k)_++\sig}
\]
 is a monotone increasing function. 
 We now send $h\downarrow0$
and then $\eps\downarrow0$, as in the proof of Proposition~\ref{Prop:2:1}, to obtain
 \begin{align*}
	\lim_{\eps\downarrow0}&\bigg(\liminf_{h\downarrow0}\iint_{E_T} 
	\partial_t \power{w_h}{p{-}1} \varphi \,\dx\dt \bigg)\\
	&\ge
	\int_{K_R} \zeta^p(x,t_1) \mathfrak h_+(u(x,t_1),\sig,k)\,\dx 
	-
	\int_{K_R} \zeta^p(x,t_2) \mathfrak h_+(u(x,t_2),\sig, k)\,\dx\\
	&\quad-
	\iint_{K_R\times (t_1,t_2)} \partial_t\zeta^p\mathfrak h_+(u,\sig, k)\,\dx\dt.
\end{align*}
Next, we consider the diffusion term.
To this end, we again send $h\downarrow0$
and then $\eps\downarrow0$, and use \eqref{Eq:1:2p}$_1$ to obtain
\begin{align*}
	\lim_{\eps\downarrow 0}&\bigg(\lim_{h\downarrow 0}
	\iint_{E_T} 
	\llbracket\mathbf A(x,t,u,Du)\rrbracket_h\cdot D\varphi\, \dx\dt \bigg)\\
	&=\iint_{K_R\times(t_1,t_2)}
	\mathbf A(x,t,u,Du)\cdot \bigg[D\z^p\frac{(u-k)_+}{(u-k)_++\sig}
	+\z^p\frac{\sigma D(u-k)_+}{\big((u-k)_++\sig\big)^2}\bigg]\, \dx\dt \\
	&\ge\iint_{K_R\times(t_1,t_2)}
	\mathbf A(x,t,u,Du)\cdot D\z^p\frac{(u-k)_+}{(u-k)_++\sig}\, \dx\dt.
\end{align*}
Combining all above estimates gives
\begin{align*}
\int_{K_R} &\zeta^p(x,t) \mathfrak h_+(u(x,t),\sig,k)\,\dx \bigg|^{t_2}_{t_1}
	-
	\iint_{K_R\times (t_1,t_2)} \partial_t\zeta^p\mathfrak h_+(u,\sig, k)\,\dx\dt\\
	&+\iint_{K_R\times(t_1,t_2)}
	\mathbf A(x,t,u,Du)\cdot D\z^p\frac{(u-k)_+}{(u-k)_++\sig}\, \dx\dt\le0.
\end{align*}
Finally we send $\sig\downarrow0$ to finish the proof.
\end{proof}
The above Lemma~\ref{Lm:A:1} has an analog near the lateral boundary $S_T$.
Suppose  $u$ is a sub(super)-solution to \eqref{Dirichlet}.
The cylinder $Q_{R,S}=K_R(x_o)\times(t_o-S,t_o)$ has its vertex $(x_o,t_o)$
attached to $S_T$ and the level $k$ satisfies \eqref{Eq:3:3}. We define
the following truncated extension of $u$ in $Q_{R,S}$:
\begin{equation*}
u_k^{\pm}\df{=}\left\{
\begin{array}{cl}
k\pm(u-k)_\pm\quad&\text{ in }Q_{R,S}\cap E_T,\\[5pt]
k\quad&\text{ in }Q_{R,S}\setminus E_T.
\end{array}\right.
\end{equation*}
An extension of $\bl{A}$ can be defined as
\begin{equation*}
\widetilde{\mathbf A}(x,t,u,\z)\df{=}\left\{
\begin{array}{cl}
\mathbf A(x,t,u,\z)\quad&\text{ in }Q_{R,S}\cap E_T,\\[5pt]
|\z|^{p-2}\z\quad&\text{ in }Q_{R,S}\setminus E_T.
\end{array}\right.
\end{equation*}
In this way, $\widetilde{\mathbf A}$ is a Caratheodory function satisfying \eqref{Eq:1:2p} with $C_o$
and $C_1$ replaced by $\min\{1,C_o\}$ and $\max\{1,C_1\}$ respectively.
Furthermore, we have
\begin{lemma}\label{Lm:A:2}
Suppose $u$ is a sub(super)-solution to \eqref{Dirichlet} with \eqref{Eq:1:2p}
and the level $k$ satisfies \eqref{Eq:3:3}.
Let $u_k^{\pm}$ be defined as above.
Then $u^{\pm}_k$ is a local weak sub(super)-solution to \eqref{Eq:1:1} with $\widetilde{\mathbf A}$
in $Q_{R,S}$.
\end{lemma}
\begin{proof}
The calculations are similar to the proof of Lemma~\ref{Lm:A:1}.
One only has to notice that due to our choice of $k$ satifying \eqref{Eq:3:3},
the test function \eqref{Test1} is still admissible if we extend it to zero 
in $Q_{R,S}\setminus E_T$ (cf. \cite[Lemma~2.1]{GLL}).
Notice also this extension does not require any smoothness of the lateral boundary $S_T$ {\it a priori}.
In this way, all the subsequent integrals are carried over into the whole $Q_{R,S}$.
\end{proof}
\begin{remark}\upshape
We point out that Lemma~\ref{Lm:A:2} exhibits an important character of parabolic equations.
So-extended sub(super)-solutions across the lateral boundary often play
a basic role in investigating the boundary regularity of solutions on rough domains.
See for instance \cite{GL, GLL} in this regard.
\end{remark}
\section{Harnack's Inequality}\label{Append:2}
First of all, let us rephrase Proposition~\ref{Prop:1:1} for non-negative super-solutions.
In such a case, $\boldsymbol \mu^-=0$ and local boundedness is not needed.
In addition, the either-or alternative does not appear.
\begin{proposition}\label{Prop:B:1}
	Let $u$ be a non-negative, local, weak super-solution to \eqref{Eq:1:1} -- \eqref{Eq:1:2p} in $E_T$.
	Suppose for some $(x_o,t_o)\in E_T$, $M>0$, $\al\in(0,1)$ and $\varrho>0$ we have \eqref{Eq:1:5} and
	\begin{equation*}
		\left|\left\{u(\cdot, t_o)\ge M\right\}\cap K_\varrho(x_o)\right|
		\ge
		\al \big|K_\varrho\big|.
	\end{equation*}
There exist constants $\dl$ and $\eta$ in $(0,1)$ depending only on the data and $\al$,
such that 
\begin{equation*}
	u\ge\eta M
	\quad
	\mbox{a.e.~in $K_{2\varrho}(x_o)\times\big(t_o+\dl(\tfrac12\varrho)^p,t_o+\dl\varrho^p\big].$}
\end{equation*}
\end{proposition}
It is worth mentioning that a similar remark as Remark~\ref{Rem:1.1}
also holds under current circumstance. 
The following Harnack's inequality has been shown in \cite{GV, KK, Trud0}.
However, we give an alternative proof based on Proposition~\ref{Prop:1:1}
and Theorem~\ref{Thm:1:1}.
\begin{theorem}\label{Thm:Harnack}
Let $u$ be non-negative, continuous, local weak solution to \eqref{Eq:1:1} -- \eqref{Eq:1:2p} in $E_T$.
Assume the set inclusion
\[
K_{2\rho}(x_o)\times(t_o-(2\rho)^p,t_o+(2\rho)^p]\Subset E_T.
\]
There exist $\theta \in (0,1)$ and $\boldsymbol\gm>1$ depending only on the data,
such that
\[
\boldsymbol\gm^{-1}\sup_{K_{\rho}(x_o)}u(\cdot,t_o-\theta\rho^p)\le u(x_o,t_o)\le 
\boldsymbol\gm\inf_{K_{\rho}(x_o)}u(\cdot,t_o+\theta\rho^p).
\]
\end{theorem}
\begin{proof}
We only prove the right-hand inequality, as the left-hand one is a direct consequence 
(cf. \cite[Chapter~5, Section~3]{DBGV-mono}).
Introduce a new function
\[
v(x,t)\df{=}\frac{u(x_o+\rho x, t_o+\rho^p t)}{u(x_o,t_o)},
\]
which satisfies the same type of equation as \eqref{Eq:1:1} -- \eqref{Eq:1:2p} in $K_2\times(-2^p,2^p]$.
Thus we only need to show that there exist $\theta \in (0,1)$ and $\boldsymbol\gm>1$, such that
\begin{equation*}
\inf_{K_1}v(\cdot, \theta)\ge\boldsymbol\gm^{-1}.
\end{equation*}
To this end, we introduce, for $\tau\in(0,1)$, the family of nested cylinders $\{Q_\tau\}$
and the families of non-negative numbers $\{M_\tau\}$ and $\{N_\tau\}$
as follows:
\[
M_{\tau}=\sup_{Q_{\tau}}v,\quad
N_\tau=(1-\tau)^{-\sig},
\]
where $\sig>1$ is to be chosen.
The two functions $[0,1)\ni\tau\to M_\tau,\,N_\tau$
are increasing, and $M_o=N_o=1$ since $v(0,0)=1$.
Moreover, $N_\tau\to\infty$ as $\tau\to1$
whereas $M_\tau$ is bounded since $v$ is locally bounded. 
Therefore the equation $M_\tau=N_\tau$
has roots and we denote the largest one as $\tau_*$.
By the continuity of $v$, there exists $(y,s)\in Q_{\tau_*}$, such that
\[
v(y,s)=M_{\tau_*}=N_{\tau_*}=(1-\tau_*)^{-\sig}.
\]
Moreover,
\[
(y,s)+Q_R\subset Q_{\frac{1+\tau_*}2}\subset Q_1,\quad\text{ where }R\df{=}\frac{1-\tau_*}2.
\]
Therefore by the definition of $\tau_*$,
\[
\sup_{(y,s)+Q_R}v\le\sup_{Q_{\frac{1+\tau_*}2}}v\le 2^{\sig}(1-\tau_*)^{-\sig}\df{=}M_*.
\]
Now let $\eps_*\in(0,1)$ and set $r=\eps_*R$.
By Theorem~\ref{Thm:1:1} (note also Remark~\ref{Rmk:1:1}), for all $r<R$ and for all $x\in K_{r}(y)$, we have
\begin{align*}
v(x,s)-v(y,s)\ge-\boldsymbol\gm M_*\left(\frac{r}R\right)^{\be}
=-\boldsymbol\gm 2^{\sig}(1-\tau_*)^{-\sig}\left(\frac{r}R\right)^{\be}
\ge-\frac12(1-\tau_*)^{-\sig},
\end{align*}
provided we choose $\eps_*$ so small that
\[
\boldsymbol\gm 2^{\sig}\eps_*^{\be}\le\frac12.
\]
This in turn gives
\[
v(x,0)\ge\frac12(1-\tau_*)^{-\sig}\df{=}M\quad\text{ for all }x\in K_{r}(y).
\]
From this we may start employing Proposition~\ref{Prop:B:1} with $\al=1$
to conclude that there exist positive constants $\eta$ and $\dl$
 as indicated, such that 
\[
v\ge\eta M
\]
in the cylinder
\[
Q^{(1)}\df{=}
K_{2r}(y)\times \left[s+\dl r^{p},
s+\dl(2r)^{p}\right].
\]
Repeating this process, we conclude that for any positive integer $n$,  
\[
v\ge\eta^n M
\]
in the cylinder
\[
Q^{(n)}\df{=}
K_{2^n r}(y)\times \left[s+\dl (2^{n-1}r)^{p},
s+\dl (2^{n}r)^{p}\right].
\]
We may assume $\eps_*(1-\tau_*)$ is a negative, integral power of $2$.
Then choose $n$ such that $2^n r=2$.
In this way, we calculate
\[
\eta^n M=\tfrac12(1-\tau_*)^{-\sig}\eta^n=\tfrac12(2R)^{-\sig}\eta^n
=\tfrac12\left(\frac{2r}{\eps_*}\right)^{-\sig}\eta^n
=2^{-2\sig-1}\eps_*^{\sig}(2^\sig\eta)^n.
\]
Finally, we may choose $\sig$ such that $2^{\sig}\eta=1$. As a result,
setting $\boldsymbol\gm^{-1}=2^{-2\sig-1}\eps_*^{\sig}$, we have
\[
v\ge\boldsymbol\gm^{-1}\quad\text{ in }
Q^{(n)}=
K_2(y)\times \left[s+\tfrac12\dl,
s+\dl\right].
\]
A further application of Proposition~\ref{Prop:B:1} (note also Remark~\ref{Rem:1.1})
gives us the desired conclusion.
\end{proof}


\begin{thebibliography}{99}
\bibitem{Acerbi-Fusco}
E. Acerbi and N. Fusco, {\it Regularity for minimizers of nonquadratic functionals: the case $1<p<2$}, 
J. Math. Anal. Appl., {\bf140}(1), (1989), 115--135.
\bibitem{DSW1} R. Alonso, M. Santillana and C. Dawson,
 {\it On the diffusive wave approximation of the shallow water
equations}, European J. Appl. Math., {\bf19}(5), (2008), 575--606.
\bibitem{ChenDB92} Y.-Z. Chen and E. DiBenedetto,
{\it H\"older estimates of solutions of singular parabolic equations with measurable coefficients},
Arch. Rational Mech. Anal., {\bf 118}(3), (1992), 257--271.
\bibitem{DB-RA} E. DiBenedetto, 
``Real analysis", Second edition, 
 Birkh\"auser/Springer, New York, 2016.
\bibitem{DB} E. DiBenedetto, ``Degenerate Parabolic 
Equations", Universitext, Springer-Verlag, New York, 1993.  
\bibitem{DB86} E. DiBenedetto, {\it On the local behaviour of solutions of degenerate 
parabolic equations with measurable coefficients},
Ann. Scuola Norm. Sup. Pisa Cl. Sci. (4), {\bf13}(3), (1986), 487--535.
\bibitem{DBGV-Liouville} E. DiBenedetto, U. Gianazza and  V. Vespri, 
{\it Liouville-type theorems for certain degenerate and singular parabolic equations},
C. R. Math. Acad. Sci. Paris, {\bf348}(15-16), (2010),  873--877.
\bibitem{DBGV-mono} E. DiBenedetto, U. Gianazza and V. Vespri, 
``Harnack's Inequality for Degenerate and Singular Parabolic 
Equations", Springer Monographs in Mathematics, Springer-Verlag, 
New York, 2012.
\bibitem{DSW2} K. Feng and F.J. Molz,
{\it A 2-d diffusion based, wetland flow model}, J. Hydrol., {\bf196}, (1997), 230--250.
\bibitem{GL} U. Gianazza and N. Liao, 
{\it A boundary estimate for degenerate parabolic diffusion equations},
Potential Analysis, in press, 19pp.
https://doi.org/10.1007/s11118-019-09794-8
\bibitem{GLL} U. Gianazza, N. Liao and T. Lukkari, 
{\it A boundary estimate for singular parabolic diffusion equations}, 
NoDEA Nonlinear Differential Equations Appl., {\bf 25}(4), (2018), 24pp.
\bibitem{GV} U. Gianazza and V. Vespri, 
{\it A Harnack inequality for solutions of doubly nonlinear parabolic equations}, 
J. Appl. Funct. Anal., {\bf1}(3), (2006), 271--284.
\bibitem{Giaquinta-Modica}
M. Giaquinta and G. Modica, 
{\it Remarks on the regularity of the minimizers of certain degenerate functionals}, 
Manuscripta Math., {\bf57}(1), (1986), 55--99.
\bibitem{DSW3} T.V. Hromadka, C.E. Berenbrock, J.R. Freckleton and G.L. Guymon, {\it A twodimensional dam-break
flood plain model}, Adv. Water Resour., {\bf8}, (1985), 7--14.
\bibitem{Ivanov-1}
A.~V.~Ivanov, 
{\it The classes $\mathcal B_{m,1}$ and H\"older estimates for quasilinear parabolic equations that admit double degeneration.} (Russian. English summary) 
Zap. Nauchn. Sem. S.-Peterburg. Otdel. Mat. Inst. Steklov. (POMI) 197 (1992), Kraev. Zadachi Mat. Fiz. Smezh. Voprosy Teor. Funktsiĭ. 23, 42--70, 179--180; translation in 
J. Math. Sci., {\bf75}(6), (1995), 2011--2027.
\bibitem{Ivanov-2}
A.~V.~Ivanov, 
{\it H\"older estimates for equations of fast diffusion type.} (Russian) Algebra i Analiz 6 (1994), no. 4, 101--142; translation in St. Petersburg Math. J., {\bf6}(4), (1995), 791--825.
\bibitem{Ivanov-Mkrtychyan}
A.~V.~Ivanov and P.~Z.~Mkrtychyan,
{\it On the regularity up to the boundary of generalized solutions of the first initial-boundary value problem for quasilinear parabolic equations that admit double degeneration.} (Russian) Zap. Nauchn. Sem. Leningrad. Otdel. Mat. Inst. Steklov. (LOMI) 196 (1991), Modul. Funktsii Kvadrat. Formy. 2, 83--98, 173--174; translation in J. Math. Sci., {\bf70}(6), (1994), 2112--2122.
\bibitem{KK} J. Kinnunen and T. Kuusi, 
{\it Local behaviour of solutions to doubly nonlinear parabolic equations}, 
Math. Ann., {\bf337}(3), (2007), 705--728.
\bibitem{KL} J. Kinnunen and P. Lindqvist,
{\it Pointwise behaviour of semicontinuous supersolutions to a quasilinear parabolic equation},
 Ann. Mat. Pura Appl. (4), {\bf185}(3), (2006),  411--435.
\bibitem{Trud1} T. Kuusi, J. Siljander and J.M. Urbano,
{\it Local H\"older continuity for doubly nonlinear parabolic equations}, 
Indiana Univ. Math. J., {\bf61}(1), (2012), 399--430.
\bibitem{Trud2} T. Kuusi, J. Siljander, R. Laleoglu and J. M. Urbano,
{\it H\"older continuity for Trudinger's equation in measure spaces},
Calc. Var. Partial Differential Equations, {\bf45}(1-2), (2012), 193--229.
\bibitem{LSU} O.A. Ladyzhenskaya, V.A. Solonnikov and N.N. Ural'tseva, 
``Linear and Quasilinear Equations of Parabolic Type'', Translations of Mathematical Monographs, Vol. 23,
American Mathematical Society, Providence, R.I., 1968.
\bibitem{pipe}
G. Leugering and G. Mophou, {\it Instantaneous Optimal Control of Friction Dominated Flow in a
Gas-Network}, in "Shape Optimization, Homogenization and Optimal Control",
International Series of Numerical Mathematics, Vol. 169, Birkh\"auser, Cham, 2018.
\bibitem{Liao} N. Liao,
{\it A unified approach to the H\"older regularity of solutions to degenerate and singular parabolic equations},
J. Differential Equations, {\bf268}(10), (2020), 5704--5750.
\bibitem{LL} E. Lindgren and P. Lindqvist, 
{\it On a comparison principle for Trudinger’s equation}, arXiv:1901.03591.
\bibitem{glacier} M.W. Mahaffy. A three-dimensional numerical model of ice sheets: Tests on the Barnes ice cap, northwest
territories, J. Geophys. Res, {\bf81}(6), (1976), 1059--1066.
\bibitem{Trud0}N.S. Trudinger,
{\it Pointwise estimates and quasilinear parabolic equations}, 
Comm. Pure Appl. Math., {\bf21}(7), (1968), 205--226.
\bibitem{Vespri} V. Vespri, 
{\it On the local behaviour of solutions of a certain class of doubly nonlinear parabolic equations}, Manuscripta Math, {\bf75}(1), 
(1992), 65--80.
\bibitem{Vespri-Vestberg}
V. Vespri and M. Vestberg, 
{\it An extensive study of the regularity properties of solutions to doubly singular equations}, 
to appear, Adv. Calc. Var., arXiv:2001.04141.


\end{thebibliography}
\end{document}